\documentclass[preprint]{elsarticle}

\usepackage{stile_PCC_graph_else_v2}

\newcommand{\oppf}{\op{opp}}

\begin{document}

\begin{frontmatter}

\title{{On the largest planar graphs with \\ everywhere positive combinatorial curvature \\ (Extended arxiv version)}}
\author{Luca Ghidelli}\fnref{luca}
\ead{luca.ghidelli@uottawa.ca}
\address{Department of Mathematics and Statistics, University of Ottawa, Canada}
\date{\today}

\begin{abstract}
A planar PCC graph is a simple connected planar graph with everywhere positive combinatorial curvature which is not a prism or an antiprism and with all vertices of degree at least 3. We prove that every planar PCC graph has at most 208 vertices, thus answering completely a question raised by DeVos and Mohar. The proof is based on a refined discharging technique and on an accurate low-scale combinatorical description of such graphs.
We also prove that all faces in a planar PCC graph have at most 41 sides, and this result is sharp as well.
\end{abstract}

\begin{keyword}
\texttt    planar graph\sep combinatorial curvature \sep positive curvature;\  discharging \sep linear optimization \sep local-global 
\MSC[2010] Primary:\ 05C10\sep 05C30;\ Secondary:\ 90C05\sep 57M15\sep 05B45 
\end{keyword}

\end{frontmatter}

\section*{Introduction}

Let $\mcl S$ be a surface (connected 2-dimensional manifold) and let $G$ be a graph 2-cell embedded in $\mcl S$, without loops or multiple edges.
Then there is on $\mcl S$ an induced structure of polyhedral surface, that is an abstract metric space made of regular polygons with some of the vertices and edges identified.
For every vertex $v$ of $G$ we consider the sum $\theta(v)$ of the angles incident in $v$, and we define its \emph{combinatorial curvature} by $K(v) = 1-\frac {\theta(v)}{2\pi}$. See formula \eqref{def:curvature} for an equivalent definition. 
The interested reader is referred to \cite{curvature:modern,curvature:higuchi,curvature:survey:hua,curvature:survey:kamtue} and the introduction of  \cite{ChenChen} for historical notes and comparison with other notions of curvature on graphs.


If all the vertices of $G$ have strictly positive combinatorial curvature and have degree at least 3, then $G$ is necessarily finite, and $\mcl S$ is either the sphere or the projective plane \cite{DeVosMohar:3444,Chen}. There are four infinite families of such graphs: the prisms, the antiprisms and their projective analogues. All other graphs with the above properties will be called \emph{PCC graphs}, and there is only a finite number of them.

In this work we mainly focus on the planar case, i.e. with $G$ embedded in the sphere, because every projective PCC graph can be lifted to a planar one. 
DeVos and Mohar \cite{DeVosMohar:3444} proved that all planar PCC graphs have at most 3444 vertices, and they asked for a sharp bound.
The first conjectured answer was  120, corresponding to the great rhombic icosidodecahedron, but then the lower bound was improved to 138 in \cite{Reti:138} and to 208 in \cite{Sneddon:208,Oldridge:244}. 


On the other hand, as it was already observed by DeVos and Mohar, much more effort is required to ameliorate the upper bound on the number of vertices.  The paper \cite{Zhang:579} lowers it to 579, but unfortunately it contains a mistake \cite[Sec. 8]{Oldridge:244}. 
Oldridge \cite{Oldridge:244} lowered the bound to 244, conditionally on a result that we prove in \cref{sec:big}, and 
an unconditional upper bound of 380 vertices was recently given by Oh \cite{Oh:380}. 

The purpose of this article is to provide a complete solution to the problem by showing that the bound 208 is optimal. 
In \cref{sec:notation,sec:preliminary} we set some preliminary notation and lemmas, and in \cref{sec:main} we outline our strategy.
The full proof occupies everything from \cref{sec:main}  to \cref{sec:208}.
Our result settles also the analogous problem in the projective setting, namely that all projective PCC graphs have at most 104 vertices.

In \cref{sec:big} we prove a useful result of independent interest in the classification of PCC graphs: every face of a planar PCC graph can have at most 41 edges. 
In the literature about PCC graphs the faces with at least 42 edges are called \emph{big faces} \cite{DeVosMohar:3444} or \emph{monster faces} \cite{Oldridge:244}. 
These faces appear very often as annoying special cases that require ad-hoc arguments to be dealt with. 
Zhang \cite{Zhang:579} was able to prove that a big face in a PCC graph has at most 290  vertices, while Oh \cite{Oh:380} showed that a PCC graph has at most one big face, with no more than 190 vertices. 
Our result show that these faces do not, in fact, exist. 

In \cref{sec:example} we present some examples which show that our results are sharp. 
First, we exhibit the known examples of planar PCC graphs with exactly 208 vertices. 
Then we show a systematic way to construct, for any given $N\in\{3,\dots,41\}$, a PCC graph $G_N$ containing a face with size $N$. 
This construction shows that our result on big faces is sharp, it disproves a conjecture made in \cite{Reti:138} and solves a problem raised by Oldridge \cite{Oldridge:244}. 

Another important theme in this paper is the notion of {\red}triangles, see \cref{def:red}. We discover that the {\red}triangles in a very large PCC graph tend to organize in cyclical structures, which we call \emph{chains}. 
We take advantage of this phenomenon in \cref{sec:208} to prove that there are no PCC graphs with exactly 209 vertices, via a simple argument.  
We also use chains of {\red}triangles to find one of the graphs with 208 vertices and to construct the graphs $G_N$. 

There is an active area of research that explores, as in the present paper,  structural theorems on polyhedral graphs with curvarure bounds. 
The interested reader is referred to \cite{planar:fullerenes,planar:nonnegative:gap,planar:nonnegative:set,planar:nonnegative:total,planar:thesis:marissa} for further research on planar graphs with nonnegative curvature and to \cite{hyp:regular,sph:areas} for graphs on spherical and hyperbolic polyhedral surfaces. 

The main technique that is used in this field, as well as in the present paper, is called \emph{discharging}. 
The discharging method  is a flexible technique in structural graph theory that is used to reduce a ``global'' statement to a number of ``local'' verifications. It was introduced more than a century ago \cite{discharging:1904} and it has been used as an essential tool in the proof of celebrated results such as the Four Color Map Theorem. 
We refer to \cite{discharging:coloring,discharging:how:to} and the first section of \cite{discharging:sec:1} for more on this technique. 

To apply the discharging method, one is required to define suitable \emph{discharging rules}. 
In the present paper, this is done in \cref{sec:pair:1,sec:pair:2}. 
The choice of weights in the discharging rules is essentially the result of a linear optimization problem. 
Therefore, in theory, a discharging argument may be performed in an automated way by a computer program, see Oldridge's thesis \cite{Oldridge:244}. 
The author believes that the ideas of the present paper, together with the methods of Oldridge, will enable further results in the classification of PCC graphs. 

\begin{remark}
	In this arxiv version of the present paper we perform the lengthy case-analysis in great detail and we accompany the text with several tables. All the numerical computations can be verified with the aid of a hand-held calculator. A shorter and more readable version of the paper will appear in print \cite{my:print}.
\end{remark}



\tableofcontents

\section{Notation for graph-theoretic objects and multisets}\label{sec:notation}


\subsection{Combinatorial curvature and basic notation}\label{sec:notation:basic}

Let $G$ be a finite simple (i.e. without loops or multiple edges) connected planar graph, and let $\Verts$, $\Edges$, $\Faces$ be respectively the set of vertices, edges and faces.
Since $G$ is finite, we consider $G$ to be 2-cell embedded in the sphere, and we include in $\Faces$ the \emph{outer face} (i.e. the one containing the point at infinity).
Given $x\in\Verts$, $y\in\Edges$ and $z\in\Faces$ we define $\vedges(x)$ to be the set of edges meeting at $x$, $\vfaces(x)$ the multiset of faces touching $x$, $\everts(y)$ the set (pair) of endpoints of $y$, $\efaces(y)$ the multiset (of cardinality 2) of faces touching $y$ by either side, $\fverts(z)$ the multiset of vertices in the boundary of $z$, and $\fedges(z)$ the multiset of edges in the boundary of $z$. If $v_1,v_2\in\Verts$ are adjacent in $G$, 
we denote the connecting edge by $v_1v_2\in\Edges$. The \emph{degree} $\deg(v)$ of a vertex $v\in \Verts$ is the cardinality of $\vedges(v)$ (or of $\vfaces(v)$). The \emph{size} $\abs{\sigma}$ of a face $\sigma\in\Faces$ is the cardinality of $\fedges(\sigma)$ (or of $\fverts(\sigma)$). In general, if $S$ is a (multi)set, we denote its cardinality by $\# S$. A \emph{triangle} is a face $\face\in\Faces$ of size 3. The \emph{face vector} $\vtype(v)$ of $v\in \Verts$ and the \emph{side vector} $\etype(e)$ of $e\in\Edges$ are the multisets of the sizes of the elements respectively of $\vfaces(v)$ and $\efaces(e)$. 
The \emph{combinatorial curvature} of $v\in\Verts$ is 
\begin{equation}\label{def:curvature}
		K(v):= 1-\frac{\deg(v)}{2}+\sum_{\face\in\vfaces(v)} \frac{1}{\abs{\face}}. 
\end{equation}
A prism (of order $N$) is a planar graph with exactly $2N$ vertices, two faces of size $N$ and $N$ faces of size 4, such that $\vtype(v)=\{4,4,N\}$ for every vertex $v$. 
An antiprism (of order $N$) is a planar graph with exactly $2N$ vertices, two faces of size $N$ and $2N$ faces of size 3, such that $\vtype(v)=\{3,3,3,N\}$ for every vertex $v$.
\begin{definition}\label{def:PCC}
A finite simple planar graph $G$ is a (planar) \emph{PCC graph} if 
\begin{enumerate}[(i)]
\item 
$K(v)\great 0$ for every $v\in\Verts$;
\item
$\deg(v)\geq 3$ for every $v\in\Verts$;
\item
$G$ is not a prism or an antiprism.
\end{enumerate}
\end{definition}
As we remarked in the introduction, there are projective analogues of the definitions above. We denote by $\mathbb P^2$ the projective plane and by $p:S^2\to\mathbb P^2$ the 2-fold covering of $\mathbb P^2$ by the sphere. 
Given a graph $G'$ that is 2-cell embedded in $\mathbb P^2$ we may consider the pull-back of its points, edges and faces through $p$. 
Since these are simple connected subsets of $\mathbb P^2$, their preimages through $p$ consist of a pair of subsets of $S^2$ homeomorphic to them. Altogether, these preimages form a planar graph $G$ which we call the \emph{pull-back} of $G'$ (through $p$). 
For every vertex $v$ of a 2-cell embedded graph in the projective plane we may consider its curvature $K(v)$ and its degree $\deg(v)$ as before. 
\begin{definition}\label{def:PCC:proj}
	A finite simple graph $G'$ that is 2-cell embedded in the projective plane $\mathbb P^2$ is a \emph{projective} PCC graph 
	if 
	\begin{enumerate}[(i)]
		\item 
		$K(v)\great 0$ for every vertex $v$ of $G'$;
		\item
		$\deg(v)\geq 3$ for every vertex $v$ of $G'$;
		\item
		the pull-back $G$ of $G'$ through $p:S^2\to\mathbb P^2$ is not a prism or an antiprism.
	\end{enumerate} 
\end{definition}

Equivalently, we may say that a projective PCC graph is a graph $G'$ embedded in $\mathbb P^2$ whose pull-back is a PCC graph.
In the remainder of the article the graph $G$ will denote a planar PCC graph, unless we explicitly state otherwise.


\subsection{Brackets notation for multisets}\label{sec:notation:multiset}

In this article we use repeatedly the notion of a multiset. By cardinality we mean the number of elements, multiplicities taken into account. When we list the elements of a multiset, multiple elements occur more than once in the list, according to their multiplicity. 
However, we employ three different types of brackets to  contain such a list.
\begin{itemize}
\item
Given an orientation $\pi$ of the sphere in which $G$ is embedded, there is a canonically induced \emph{cyclic order} on the multisets $\vedges(v)$, $\vfaces(v)$, $\vtype(v)$, $\fverts(\face)$ and $\fedges(\face)$,  for every $v\in \Verts$ and $\sigma\in\Faces$. When we want to list their elements counterclockwise according to this order, we write them between angle brackets $\langle \cdots \rangle_\pi$, the first element being arbitrarily chosen. We drop the reference to $\pi$ if the list can be obtained from one of the two possible orientations, but we don't emphasize which one.
\item
The usual linear order of $\mathbb N$ induces a partial order on the multisets $\vtype(v)$ and $\etype(e)$, for every $v\in \Verts$ and $e\in\Edges$, which in turn induces a partial order on the multisets $\vfaces(v)$ and $\efaces(e)$, by considering the size of their elements. We use round brackets $(\cdots )$ to list the elements of these multisets increasingly according to this order. In case an element is not strictly greater that another, we choose arbitrarily which one to list first.
\item
Finally, we use curly brackets $\{\cdots\}$ if we don't specify a particular order for the elements.
\end{itemize}

A multiset $\mcl A=\{a_1,\ld,a_n\}$ is a submultiset of $\mcl B$, if for all $i\in\{1,\ld,n\}$ the element $a_i$ appears in $\mcl B$ with multiplicity greater than or equal to the multiplicity of $a_i$ in $\mcl A$. 
In this case we write $\mcl A\subseteq \mcl B$ or $\{a_1,\ld,a_n\}\subseteq\mcl B$. 
If $\mcl A$ and $\mcl B$ are linearly ordered, we write $(a_1,\ld,a_n)\subseteq \mcl B$ to emphasize that the order in $\mcl A$ equals the one induced by $\mcl B$. 
If instead $\mcl B$ has a cyclic order, we write $\langle a_1,\ld,a_n\rangle\subseteq\mcl B$ to say that $a_1,\ld,a_n$ appear in $\mcl B$ as \emph{consecutive} elements. 
For example: if $\mcl B:=\langle 4,3,3,5\rangle_\pi$, both $\langle 3,5,4\rangle_\pi\subseteq \mcl B$ and $\langle 5,4,3\rangle_\pi\subseteq \mcl B$ are true statements.
We say that $v_1,\ld,v_n\in\Verts$ (resp. $e_1,\ld,e_n\in\Edges$) are consecutive on $\face\in\Faces$ if $\langle v_1,\ld,v_n\rangle\subseteq\fverts(\face)$ (resp. $\langle e_1,\ld,e_n\rangle\subseteq\fedges(\face)$).


\subsection{Other definitions}

\begin{figure}[ht]
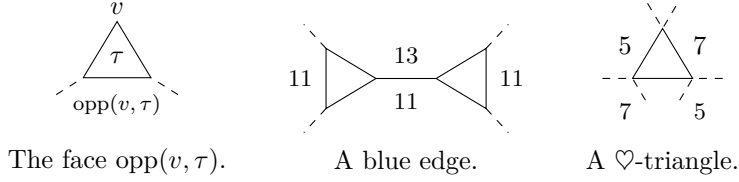
\centering
\input{def_opp.tikz}
\quad
\input{def_blue.tikz}
\quad
\input{def_red.tikz}
\caption{Illustrations for three definitions.}
\label{fig:def}
\end{figure}

The following notion will be useful in several places to indicate the face situated opposite to a vertex with respect to a triangle.

\begin{definition}\label{def:opp}
Let $v\in\Verts$ with $3\in\vtype(v)$ and let $\tau\in\vfaces(v)$ with $\abs{\tau}=3$. Then by \cref{lemma:distinct} below there are well-defined $w_1,w_2\in\Verts$ and $\face\in\Faces$ with $\fverts(\tau)=\{v,w_1,w_2\}$ and $\efaces(w_1w_2)=\{\tau,\face\}$. Then we define
\[
\oppf(v,\tau):=\face.
\]
\end{definition}

In \cref{sec:pair:1}, \cref{sec:11} and \cref{sec:13} we will make use of the following notion of $blue$-edges, $\alpha$-vertices and $\beta$-vertices.
 
\begin{definition}\label{def:blue}
Let $e\in\Edges$ and $v_1,v_2\in\Verts$ with $\etype(e)=(11,13)$ and $\everts(e)=\{v_1,v_2\}$. Then $\vtype(v_1)=\vtype(v_2)=(3,11,13)$ (see the table of admissible vertices in a PCC graph in \cref{sec:list}). Then for $i=1,2$ let $\tau_i\in\vfaces(v_i)$ with $\abs{\tau_i}=3$. If both $\abs{\oppf(v_1,\tau_1)}=11$ and $\abs{\oppf(v_2,\tau_2)}=11$ we say that $e$ is a $blue$-edge and we say that its endpoints $v_1,v_2$ are $\beta$-vertices. Otherwise we say that $v_1$ and $v_2$ are $\alpha$-vertices.
\end{definition}

Finally, in \cref{sec:208} we will make use of the following notion of a {\red}triangle.

\begin{definition}\label{def:red}
We say that  $\tau\in\Faces$ with $\abs{\tau}=3$ is a \emph{{\red}triangle} if 
\[
\forall \ v\in\fverts(\tau)\ \ \ \vtype(v)=(3,3,5,7).
\]
\end{definition}


\section{Admissible vertices and preliminary lemmas}\label{sec:preliminary}

\subsection{The list of admissible vertices}\label{sec:list}

It's not difficult to check that a vertex of a PCC graph can have degree at most 5. 
Essentially, this is true because an internal angle of a regular polygon is at least $1/6$ the measure of a full angle. 
Hence, it is a straightforward arithmetic calculation to list all the admissible face vectors $\vtype(v)$ for a vertex $v\in\Verts$ with positive combinatorial curvature.
This is done, for example, in \cite[Table 1]{DeVosMohar:3444}. 
For the reader's convenience, we copy this list in \cref{table:admissible}.

\begin{table}[H]\centering
\begin{tabular}{cc|cc}
\toprule
\ \ $\vtype(v)$
	\ \ \ &\ \ \ 
where 
&
\ \ $\vtype(v)$
	\ \ \ &\ \ \ 
where \\

\midrule

$(3,a,b)$
	\ \ \ &\ \ \ 
$3\leq a\leq 6,\ a\leq b$\ \ \ 
&
$(3,4,4,a)$
	\ \ \ &\ \ \ 
$4\leq a\leq 5$\ \ \ \\

\midrule

$(3,7,a)$
	\ \ \ &\ \ \ 
$7\leq a\leq 41$\ \ \ 
&
$(3,3,3,3,a)$
	\ \ \ &\ \ \ 
$3\leq a\leq 5$\ \ \ \\

\midrule

$(3,8,a)$
	\ \ \ &\ \ \ 
$8\leq a\leq 23$\ \ \ 
&
$(4,4,a)$
	\ \ \ &\ \ \ 
$4\leq a$\ \ \ \\

\midrule

$(3,9,a)$
	\ \ \ &\ \ \ 
$9\leq a\leq 17$\ \ \ 
&
$(4,5,a)$
	\ \ \ &\ \ \ 
$5\leq a\leq 19$\ \ \ \\

\midrule

$(3,10,a)$
	\ \ \ &\ \ \ 
$10\leq a\leq 14$\ \ \ 
&
$(4,6,a)$
	\ \ \ &\ \ \ 
$6\leq a\leq 11$\ \ \ \\

\midrule

$(3,11,a)$
	\ \ \ &\ \ \ 
$11\leq a\leq 13$\ \ \ 
&
$(4,7,a)$
	\ \ \ &\ \ \ 
$7\leq a\leq 9$\ \ \ \\

\midrule

$(3,3,3,a)$
	\ \ \ &\ \ \ 
$3\leq a$\ \ \ 
&
$(5,5,a)$
	\ \ \ &\ \ \ 
$5\leq a\leq 9$\ \ \ \\

\midrule

$(3,3,4,a)$
	\ \ \ &\ \ \ 
$4\leq a\leq 11$\ \ \ 
&
$(5,6,6)$
	\ \ \ &\ \ \ 
\phantom{aaaaaaaaa}
\ \ \ \\

\midrule

$(3,3,5,a)$
	\ \ \ &\ \ \ 
$5\leq a\leq 7$\ \ \ 
&
$(5,6,7)$
	\ \ \ &\ \ \ 
\ \ \ \\

\bottomrule
\end{tabular}
%
\caption{Table of admissible face vectors.} \label{table:admissible}
\end{table}

Throughout the article, we say that a face vector $\vtype(v)$ is \emph{admissible} if it is one of the multisets listed in the above table. Notice that for every vertex $v\in\Verts$ of a PCC graph we have that $\vtype(v)$ is admissible. Conversely, if $\vtype$ is an admissible multiset, then there is a finite simple planar graph $G$ and $v\in\Verts$ with $K(v)\great 0$ and $\vtype(v)=\vtype$. However, such $G$ need not be a PCC graph, see \cref{sec:big}.


\subsection{Multisets that are actually sets}\label{sec:lemma:distinct}

We already remarked that if $G$ is a finite simple planar graph and $\face\in\Faces$, then $\fverts(\face)$ and $\fedges(\face)$ are a priori just multisets. 
However, PCC graphs exhibit remarkable rigidities, both at local and global scale. 
This implies that most of these multisets are actually sets, and this simplifies our exposition.
First of all, a simple graph has, by definition, no multiple edges between two points. 
We now record this basic observation for future reference. 

\begin{lemma}\label{lemma:distinct:path:2}
Let $G$ be a finite simple graph, and let $e_1,e_2\in\Edges$, $v,v_1,v_2\in\Verts$ with $e_1=vv_1$, $e_2=vv_2$ and $e_1\neq e_2$. Then $v_1\neq v_2$.
\end{lemma}

An easy consequence of \cref{lemma:distinct:path:2} is the following.
\begin{lemma}\label{lemma:distinct}
Let $G$ be a PCC graph and let $\face\in\Faces$ such that $\fverts(\face)$ is not a set. Then $7\leq \abs{\face}\leq 11$. 
The same conclusion on $\abs \face$ holds if $\fedges(\face)$ is not a set.
\end{lemma}
\begin{proof}
Notice that the second assertion follows from the first, because if $\fedges(\face)$ is not a set, then also $\fverts(\face)$ is not a set.
Now, without loss of generality, suppose that $\fverts(\face)=\langle v_1,\ld,v_{\abs{\face}}\rangle$ and that $v_1=v_j$ with $2\leq j \leq \frac{\abs{\face}} 2 + 1$. 
Since $G$ has no loops, we cannot have $j=2$.
By \cref{lemma:distinct:path:2} and the fact that $\deg v_2\geq 3$ we also see that $j\neq 3$. 
Then $4\leq j$, and this is not possible if $\abs{\face}\leq 5$.
Moreover, if $\abs\face = 6$ we must have $j=4$.
In this case, since $(6,6)\subseteq \vtype(v_1)$ we have $\vtype(v_1)=(a,6,6)$ for some $a\in\{3,4,5\}$ by  \cref{table:admissible}.
This implies that 
 either $v_2=v_3$ or $v_4=v_5$, but this is a contradiction because $G$ has no loops.
Finally, if $\abs \face \geq 12$ we have that $(\abs \face,\abs \face)\subseteq \vtype(v_1)$ is not admissible by \cref{table:admissible}.
\end{proof}

Equivalently, switching the point of view from faces to vertices, we have

\begin{corollary}\label{lemma:distinct:vfaces}
Let $G$ be a PCC graph, let $v\in\Verts$ and let $\face\in\Faces$ appearing in $\vfaces(v)$ with multiplicity 2.
Then $7\leq \abs{\face}\leq 11$.
\end{corollary}

We remark that one could improve \cref{lemma:distinct} (and so also \cref{lemma:distinct:vfaces}) to show that $10\leq \abs{\face}\leq 11$. This last inequality is sharp, see \cref{fig:multiset:10:11}.

\begin{figure}
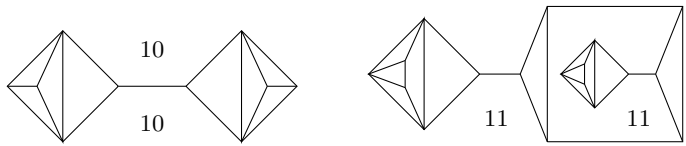

\begin{center}
\input{distinct_10.tikz}
\quad \quad
\input{distinct_11.tikz}
\end{center}
\caption{PCC graphs with faces having multiple edges on their boundary.}
\label{fig:multiset:10:11}
\end{figure}


\subsection{Large faces cannot be too close}\label{sec:lemma:close}

The next lemma essentially says that two large faces in a PCC graph cannot be too close without merging. It is a refinement of \cite[Lemma 4.2]{ChenChen}.

\begin{figure}[ht]
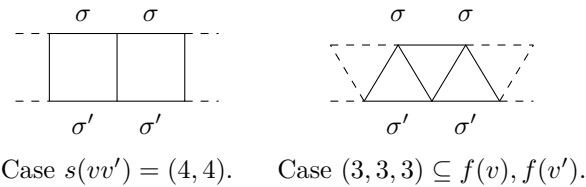
\centering
\input{big_close_44.tikz}
\ 
\input{big_close_333.tikz}
\caption{Illustrations for Lemma 2.4.}
\label{fig:bg:close}
\end{figure}
\begin{lemma}\label{lemma:big:close}
Let $G$ be a PCC graph, let $v,v'\in\Verts$ and let $\face,\face'\in\Faces$ with $\face\in\vfaces(v)$, $\face'\in\vfaces(v')$ and $\abs{\face},\abs{\face'}\geq 20$.
If $vv'\in\Edges$ (i.e. $v,v'$ are adjacent in $G$), then $\face=\face'$ and $vv'\in\fedges(\face)$ (i.e. $v$, $v'$ are consecutive on $\face$).
\end{lemma}
\begin{proof}
Fix an orientation $\pi$ of the sphere, and let $\fverts(\face)=\langle v_1,\ld,v_{\abs \face}\rangle_\pi$ and 
$\fverts(\face')=\langle v'_{\abs{ \face'}},\ld,v'_1\rangle_\pi$, with $v_1=v$ and $v'_1=v'$.
Suppose that $vv'\not\in\fedges(\face)$. 
We have by \cref{table:admissible} the following 4 cases.
\begin{description}
\item[Case $\etype(vv')=(3,a)$ with $4\leq a\leq 8$]
Let $\tau\in\efaces(vv')$ with $\abs \tau=3$, and let $\fverts(\tau)=\{v,v',w\}$. 
We have that $\efaces(vw)=\{\tau,\face\}$ and $\efaces(v'w)=\{\tau,\face'\}$. 
Hence $\langle\face,\tau,\face'\rangle \subseteq \vfaces(w)$, which is not admissible.
\item[Case $\etype(vv')=(4,4)$]
Let $\kappa\in\efaces(vv')$ so that $\fverts(\kappa)=\langle v_1,v_1',v'_2,v_2 \rangle_\pi$. 
In particular $v_2v'_2\in\fedges(\kappa)\subseteq\Edges$. 
We have that $\vtype(v_2)\in\{(3,4,\abs\face),(4,4,\abs\face)\}$ so in any case $v_2v'_2\not\in\fedges(\face)$. 
Since $v_2\in\fverts(\face)$ and $v'_2\in\fverts(\face')$, the arguments of the previous case show that $\etype(v_2v'_2)=(4,4)$.
\item[Case $\vtype(v)=(3,3,\abs\face)$ or $\vtype(v')=(3,3,\abs{\face'})$]
Assume $\vtype(v)=(3,3,\abs\face)$, the other situation being analogous. 
We have $\efaces(vv')=\{\tau_1,\tau_2\}$ with $\abs{\tau_1}=\abs{\tau_2}=3$, $\fverts(\tau_1)=\{v,v',v_{1}\}$, and $\fverts(\tau_2)=\{v,v',v_2\}$. 
We also have that $\vtype(v')\in\{(3,3,\abs{\face'}), (3,3,3,\abs{\face'})\}$, so either $\{\face,\face',\tau_1\}\subseteq \vtype(v_{1})$ or $\{\face,\face',\tau_2\}\subseteq \vtype(v_2)$, both of which are not admissible.
\item[Case $\vtype(v)=(3,3,3,\abs\face)$ and $\vtype(v')=(3,3,3,\abs{\face'})$]
Then there are $\tau_1,\tau_2\in\Faces$ and $w_1,w_2\in\Verts$ such that 
$\abs{\tau_1}=\abs{\tau_2}=3$, $\fverts(\tau_1)=\langle w_1,v',v\rangle_\pi$ and $\fverts(\tau_2)=\langle v,v',w_2\rangle_\pi$. 
Now we have 2 similar subcases. 
\begin{description}
\item[Case $\vfaces(v)=\langle \tau_1,\tau_2,\tau_3,\face\rangle_\pi$ with $\abs{\tau_3}=3$]
Then $\fverts(\tau_3) = \langle v,w_2,v_2 \rangle_\pi$ and $\langle \tau_1,\face\rangle_\pi\subseteq \vfaces(w_1)$. 
Since $\{\face,\face',\tau_1\}\subseteq \vfaces(w_1)$ is not admissible, we necessarily have $\vfaces(v')=\langle \tau_2,\tau_1,\tau_0,\face'\rangle_\pi$ for some triangle $\tau_0\in\Faces$. 
We deduce that $w_2=v'_2$, and so $v_2v'_2\in\fedges(\tau_3)\subseteq \Edges$. 
\item [Case $\vfaces(v)=\langle\tau_0, \tau_1,\tau_2,\face\rangle_\pi$ with $\abs{\tau_0}=3$]
As before, we get that $w_2=v_2$, and since $\{\face,\face',\tau_1\}\subseteq \vfaces(w_2)$ is not admissible, we deduce that $\vfaces(v')=\langle \tau_3, \tau_2,\tau_1,\face'\rangle_\pi$ for some triangle $\tau_3\in\Faces$. 
Then  $\fverts(\tau_3) = \langle v',v'_2,w_2 \rangle_\pi$ and so $v_2v'_2\in\fedges(\tau_3)\subseteq \Edges$.
\end{description}
In both cases we see that $v_2v'_2\in\Edges$ but $v_2v'_2\not \in\fedges(\face)$. 
Moreover, by the previous cases, we get that $\vtype(v_2)=(3,3,3,\abs\face)$ and $\vtype(v'_2)=(3,3,3,\abs{\face'})$.
\end{description}
The above analysis shows by induction that for all $n\in\N$ we have $v_nv'_n\in\Edges$ but $v_nv'_n\not \in\fedges(\face)$ (where the indices in $v_n$ and $v'_n$ are taken modulo $\abs \face$ and $\abs{\face'}$ respectively).
Moreover we have that either for all $i,j$ we have $\vtype(v_i)=(4,4,\face)$ and $\vtype(v'_j)=(4,4,\face')$, or for all $i,j$ we have $\vtype(v_i)=(3,3,3,\face)$ and $\vtype(v'_j)=(3,3,3,\face')$.
In the first case we deduce that $G$ is a prism, while in the second case we get that $G$ is an antiprism. 
In either case, $G$ is not a PCC graph.
\end{proof}

Finally, the next lemma will be useful for the proof in \cref{sec:big}.

\begin{lemma}\label{lemma:two:edges}
Let $\face,\face', \kappa\in\Faces$ with $\abs{\face},\abs{\face'}\geq 20$ and $\abs{\kappa}\leq 6$. 
Let also $e,e'\in\fedges(\kappa)$ with $e\in\fedges(\face)$ and $e'\in\fedges(\face')$. 
Then $\face=\face'$ and $e=e'$.
\end{lemma}
\begin{proof}
Let $\fverts(\kappa)=\langle v_1,\ld,v_{\kappa}\rangle$ and assume, without loss of generality, that $e=v_1v_2$ and $e'=v_jv_{j+1}$, with $j\leq \frac{\abs{\kappa}} 2 +1 $. 
We cannot have $j=2$ because otherwise $\{\face,\kappa,\face'\}\subseteq \vfaces(v_2)$, which is not admissible. 
We cannot have $j=3$ because of \cref{lemma:big:close} applied to $v_2v_3$. 
Therefore $j\geq 4$, which implies that $\abs\kappa = 6$ and $j=4$. 
We have $\vtype(v_2)=(3,6,\abs \face)$ and $\vtype(v_4)=(3,6,\abs{\face'})$, so there are triangles $\tau,\tau'\in\Faces$ with $\langle \tau,\kappa,\tau'\rangle\subseteq \vfaces(v_3)$, $\oppf(v_3,\tau)=\face$ and $\oppf(v_3,\tau')=\face'$. 
Let $\fverts(\tau)=\{v_2,v_3,v_{23}\}$ and $\fverts(\tau')=\{v_3,v_4,v_{34}\}$. 
We rule out the following 2 cases.
\begin{description}
\item[Case $\vtype(v_3)=(3,3,4,6)$]
Let $\kappa'\in\vfaces(v_3)$ with $\abs{\kappa'}=4$ and let $\fverts(\kappa')=\langle v_3, v_{34}, w_3, v_{23}\rangle$. 
We notice that $\vtype(v_{23})=(3,4,\abs \face)$ and $\vtype(v_{34})=(3,4,\abs{\face'})$. 
Therefore $\kappa'$ shares the edge $v_{23}w_3$ with $\face$, and shares the edge $v_{34}w_3$ with $\face'$. 
However we already argued that this is impossible.
\item[Case $\vtype(v_3)=(3,3,5,6)$]
Let $\kappa'\in\vfaces(v_3)$ with $\abs{\kappa'}=5$ and let $\fverts(\kappa')=\langle v_3, v_{34}, w_4,w_2, v_{23}\rangle$. 
We notice that $\vtype(v_{23})=(3,5,\abs \face)$ and $\vtype(v_{34})=(3,5,\abs{\face'})$. 
Therefore $\kappa'$ shares the edge $v_{23}w_2$ with $\face$, and shares the edge $v_{34}w_4$ with $\face'$. 
However we already argued that this is impossible.
\end{description}
Therefore we have $\vtype(v_3)\in\{(3,3,6), (3,3,3,6)\}$ by \cref{table:admissible}. 
We now describe what happens in either case. 
	\begin{figure}[ht]
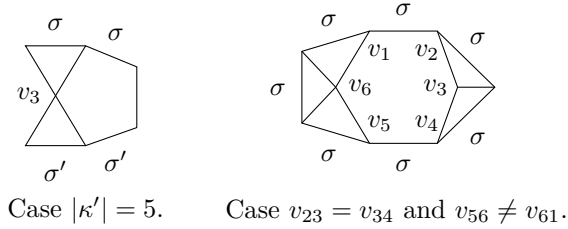
\centering
	\input{distinct_edges_5.tikz}
	\quad
	\input{distinct_edges_6.tikz}
	\caption{Illustrations for Lemma 2.5.}
	\label{fig:two:edges:6}
	\end{figure}
\begin{description}
\item[Case $\vtype(v_3)=(3,3,6)$]
Then $v_{23}=v_{34}$. 
Since $\langle\face,\tau, \tau', \face'\rangle\subseteq \vfaces(v_{23})$ is not admissible, we must have $\face=\face'$ and $\vfaces(v_{23})=\{\tau,\tau',\face\}$.
In particular $v_1,v_2,v_{23},v_4,v_5$ are consecutive on $\face$.
\item[Case $\vtype(v_3)=(3,3,3,6)$]
Let $\tau_3\in\Faces$ such that $\vfaces(v_3)=\langle \tau, \kappa, \tau', \tau_3'\rangle$. 
Then $\fverts(\tau_3)=\{v_{23},v_3,v_{34}\}$. 
Then by \cref{lemma:big:close} we have that $\face=\face'$ and $v_{23}v_{34}\in\fedges(\face)$. 
In particular $v_1,v_2,v_{23}, v_{34},v_4,v_5$ are consecutive on $\face$.
\end{description}
Replacing $v_3$ with $v_6$ and repeating the above arguments we find that there are $v_{56}, v_{61}\in\Verts$ such that $v_5, v_{56}, v_1$ or $v_5,v_{56}, v_{61}, v_1$ are consecutive on $\face$. 
This implies that $\abs \face\leq 8$, contrary to the assumption.
 \end{proof}

The condition $\abs{\kappa}\leq 6$ in \cref{lemma:two:edges} is necessary, as
 as shown in \cref{fig:two:edges:7:20}.
	\begin{figure}[ht]\centering
	\input{distinct_edges_7.tikz}
	\caption{A PCC graph in which two faces $\face, \kappa$ with $\protect\abs\face = 20$, $\protect\abs\kappa=7$ share two edges.}
	\label{fig:two:edges:7:20}
	\end{figure}

%
%
%


\section{The main theorem and outline of the proof}\label{sec:main}


The goal of this article is to prove the following.
\begin{theorem}\label{thm:main}
Let $G$ be a planar PCC graph. Then $\#\Verts\leq 208$.
\end{theorem}
As we remarked in the introduction, and as we will show in \cref{sec:example}, there are examples of PCC graphs with 208 vertices, so \cref{thm:main} is sharp. 
We recall from the remarks following \cref{def:PCC:proj} that a \emph{projective} PCC graph is a finite simple graph $G'$ that 
is 2-cell embedded in the projective plane $\mathbb P^2$ and such that its pull-back $G$, through the 2-fold covering of $\mathbb P^2$ by the sphere, is a planar PCC graph. 
The number of vertices of the pull-back $G$ is twice the number of the vertices of $G'$, therefore \cref{thm:main} has the following consequence. 
\begin{corollary}\label{thm:main:proj}
	A projective PCC graph has at most 104 vertices. 
\end{corollary}
It is easy to see that the large planar PCC graphs discussed in \cref{sec:example} descend to projective PCC graphs with 104 vertices \cite{Sneddon:208,Oldridge:244}, hence \cref{thm:main:proj} is sharp as well. 

In order to prove \cref{thm:main}, we introduce two auxiliary indeterminate objects $\dface$, $\ddface$, which we call \emph{auxiliary faces}, and we consider the set of \emph{discharge faces} $\FFaces$ given by
\[
	\FFaces:=\{\dface,\ddface\}\cup\{\face\in\Faces:\ \abs{\face}\not\in \{3,4,6,8,9,10,12\}\}.
\]
We will construct, in \cref{sec:pair:1,sec:pair:2}, a map $\pair:\ \Verts\times\FFaces\to\Q_{\geq 0}$, called \emph{pairing}, satisfying 
\begin{equation}\label{eq:pairing}
\sum_{\face\in\FFaces} \pair(v,\face) = 1
\quad \quad \forall \, v\in\Verts.
\end{equation} 

For every $v\in\Verts$ we denote $\c_v:=K(v)-\frac{2}{209}$ and for every $\face\in\FFaces$ we define
\[
  \pair(\face):=\sum_{v\in\Verts} \pair(v,\face)
	\quad\text{ and }\quad
	\c(\face):=\sum_{v\in\Verts} \c_v\,\pair(v,\face).
\]
The following quantities will also be useful in estimating $\c(\face)$, for $\face\in\FFaces$:
\[
\c_+(\face):=\sum_{\substack{v\in\Verts\\ \c_v\geq 0}} \c_v\,\pair(v,\face)
   \quad\text{ and }\quad
\c_-(\face):=\sum_{\substack{v\in\Verts\\ \c_v\less 0}} \c_v\,\pair(v,\face).
\]

The crucial observation is contained in the following lemma.
\begin{lemma}\label{lemma:main}
We have $\#\Verts\leq 208$ if and only if $\c(G):=\sum_{\face\in\FFaces} \c(\face)\great 0$.
\end{lemma}
\begin{proof}
By the Euler-Poincar\'e formula and a double-counting argument we have
\begin{equation}\label{eq:total_curvature}
\sum_{v\in\Verts} K(v) = 
\sum_{v\in\Verts} 1-\sum_{e\in\Edges}\sum_{v\in\everts(e)}\frac{1}{2} + \sum_{\face\in\Faces}\sum_{v\in\fverts(\face)}\frac{1}{\abs{\face}} = \#\Verts-\#\Edges+\#\Faces = 2.
\end{equation}
Hence by formula \eqref{eq:pairing} we have
\begin{equation}\label{eq:contribution}
\sum_{\face\in\FFaces} \c(\face) = \sum_{\face\in\FFaces} \sum_{v\in\Verts} \c_v\pair(v,\face) = \sum_{v\in\Verts} \c_v = \sum_{v\in\Verts} \left(K(v)-\frac{2}{209}\right) = \frac{2(209-\#\Verts)}{209},
\end{equation}
which is strictly positive if and only if $\#\Verts$ is strictly smaller than 209.
\end{proof}

It is therefore our objective to construct the pairing $\pair:\ \Verts\times\FFaces\to\Q_{\geq 0}$ so that for every $\face\in\FFaces$ with $\pair(\face)\neq 0$ the number $\c(\face)$ is positive or, if negative, very small in absolute value, so that it can be countered in the sum by another positive term. 
The outcome of the required case-by-case analysis is summarized in the following.

\begin{proposition}\label{prop:main}
Let $G$ be a PCC graph and $\face\in\FFaces$ with $\pair(\face)\neq 0$.
\begin{enumerate}[(i)]
\item
If $\abs{\face}=5$, then $\c(\face)\geq 0.002$.
\item
If $\abs{\face}=7$, then $\c(\face)\geq 0.0095$.
\item
If $\abs{\face}=11$, then $\c(\face)\geq 0.0003$.
\item
If $\abs{\face}=13$, then $\c(\face)\geq 0.00003$.
\item
If $14\leq \abs{\face}\leq 39$ and $\abs{\face}\neq 19$, then $\c(\face)\geq 0.0002$.
\item
If $\abs{\face}=19$, then $\c(\face)\geq 0.0065$.
\item
If $\abs{\face}=40$ or $\abs{\face}=41$, then $\c(\face)\geq 0.011$.
\item
If $\face = \ddface$, then $\c(\face)\geq 0.0006$.
\item
If $\face=\dface$, then $\c(\face)\geq\frac{-2}{209}\great -0.0096$.
\end{enumerate}
\end{proposition}

In addition to this, we will prove in \cref{sec:big} that there is no $\face\in\Faces$ with $\abs{\face}\geq 42$. 
Then \cref{prop:main} is enough to conclude that $\#\Verts\leq 210$, as we show in \cref{sec:210}. 
Moreover, it implies that $\#\Verts=210$ if and only if for all $v\in\Verts$ we have $\vtype(v)\in\{(3,3,5,6),(5,6,7)\}$, but we show in \cref{sec:209} that this is impossible.
With a more accurate estimation of the local contributions to $\c(G)$, and of the global obstructions, it could be possible to apply \cref{lemma:main} in all cases, thus establishing $\#\Verts\leq 208$. 
However, this would have made the case analysis of this article even longer, so we decided to conclude otherwise, in a nicer way, using the notion of \red-triangles (see \cref{def:red}).
In \cref{sec:208} we discover that \red-triangles contained in very large PCC graphs tend to organize in a cyclical pattern.
Using this phenomenon we are able to prove that if there exists a PCC graph $G$ with 209 vertices, then there exists another PCC graph $G'$ with at least 210 vertces. This will conclude the proof of \cref{thm:main}.

The most delicate part of our proof is the definition of a discharging pairing $\pair:\Verts\times\FFaces\to\Q_{\geq 0}$ that is suitably optimized for our goals.  
The full construction of $\pair$ is quite elaborate, and will require two full sections to be described. 
However, most of the special discharging rules are actually introduced only to overcome the difficulties that we will encounter in \cref{sec:11,sec:4041} 
(i.e. for the cases $\abs{\face}\in\{11,40,41\}$ of \cref{prop:main}).  
A much shorter proof could be given, say, just for the upper bound $\#\Verts\leq 264$ (see \cref{sec:pair:rmk} for the description of simplified versions of $\pair$). 

\subsection{Some remarks on the method}\label{sec:main:rmk}

\subsubsection{Local-to-global inequalities} 
In the paragraphs of this section we explain on a more conceptual level the discharging strategy that we outlined above. 
One first fundamental observation is that the total combinatorial curvature of a finite planar graph is equal to 2 by \cref{eq:total_curvature}. 
Then the ``global statement'' $\#\Verts\leq 208$ of the main theorem can be rewritten as 
\[
 \op{avg} K := \frac 1 {\#\Verts} \sum_{v\in\Verts} K(v)\geq \frac 2 {208}.
\]
Since $\#\Verts$ is an integer, it is actually sufficient to prove the strict inequality $\op{avg} K> 2/209\approx 0.0095$.  
The quantity $\op{avg} K$ can be estimated via ``local computations''. 
For example, in view of our theorem \cref{thm:big} and of the list of admissible face vectors listed in \cref{table:admissible}, we have that the curvature at a vertex $v\in \Verts$ is minimal if $\vtype(v)= (3,7,41)$. 
In this case we have $K(v)=1/1722\approx 0.00058$, and so $\op{avg} K\geq \min K\geq 2/3444$, or $\#\Verts \leq 3444$. 
Here we notice that the local-to-global inequality $\op{avg} K\geq \min K$ does not furnish a sharp estimate of $\op{avg} K$ because the function $K:\Verts\to\R$ is allowed to vary considerably in magnitude: 
it can be as small as $K(v)\approx 0.00058$, as we just saw, and it can get as big as $K(v)=0.5$, when $\vtype(v)=(3,3,3)$. 

\subsubsection{Transportation of curvature} 
As we mentioned in the introduction, the discharging method is a general-purpose technique to obtain more refined local-to-global estimates as the one above. 
A basic approach towards discharging is the \emph{redistribution of the curvature} among the vertices: every vertex $v\in\Verts$ ``gives'' part of its curvature  to the neighboring vertices. 
The result of this process is a new function $K':\Verts\to\R$ with $\op{avg} K'=\op{avg} K$. 
If the discharging rules are designed properly, then we would also have $\min K'>\min K$ because the vertices $v\in \Verts$ with the least value of $K(v)$ would receive some extra curvature from the neighboring vertices. 
Moreover the values of $K'$ could be still determined by a ``local computation'', where we analyze the possible combinatorial configurations in the neighborhood of vertices. 
On a formal level, the ``discharging weights'' used in the above process can be encoded as a function $\pair_{\text{flow}}:\Verts\times\Verts\to \R_{\geq 0}$ that satisfies
$
 \sum_{w\in\Verts}\pair_{\text{flow}}(v,w) = 1
$. 
In transportation theory parlour we may view $K$ as a discrete measure on $\Verts$ with total mass equal to $2$ and $\pair_{\text{flow}}$ as a \emph{transportation plan} from $K$ to a ``smoother'' measure $K'$ on $\Verts$. 

\subsubsection{Discharging as fuzzy partitioning}
There are many ways of ``spreading out'' as above the combinatorial curvature, making it flow from vertices with large curvature to vertices with small curvature. 
One natural way of doing so is via discrete analogs of the so-called Ricci flow. 
However, it seems to the author that these methods do not easily imply estimates that are as precise as the ones that are needed in the present manuscript. 
Another basic idea is to \emph{partition the vertices} of the graph so that in each partition the ``bad'' vertices with small curvature appear together with the ``good'' vertices with large curvature. 
If the partitions are indexed by a set $\mcl I$, then the partitioning $\{\Verts_i\}_{i\in\mcl I}$ of $\Verts$ can be encoded 
with the boolean function $\pair_{\text{part}}:\Verts\times\mcl I \to \{0,1\}$ such that $\pair_{\text{part}}(v,i)=1$ if $v\in\Verts_i$.   
Then $\op{avg} K$ can be computed as a weighted mean of the averages $K'(i):=\op{avg} K|_{\Verts_i}$ of the function $K$ restricted to the $\Verts_i$.  
If the partitions are somewhat constructed on the basis of ``local'' graph-theoretic configurations, then $\op{avg} K\geq \min_{i\in\mcl I} K'$ could be seen as another local-to-global inequality.
Yet greater flexibility is achieved if we summarize the above two approaches by considering a function $\pair_{\text{fuzzy}}:\Verts\times\mcl W\to \R_{\geq 0}$ that satisfies
$$
 \sum_{w\in\mcl W}\pair_{\text{flow}}(v,w) = 1.
$$ 
This function can be seen as a transportation plan between the measure $K$ on the set $\Verts$ and its push-forward on \emph{another set} $\mcl W$. 
Alternatively, we may see it as a \emph{fuzzy} partitioning of the set $\Verts$ in which the element $v\in\Verts$ belongs to the ``partition'' $w\in\mcl W$ with ``probability'' $\pair_{\text{fuzzy}}(v,w)$.


\section{The pairing: part 1}\label{sec:pair:1}

We will define the pairing $\pair$ as the pointwise sum of two functions $\pair_1,\pair_2:\,\Verts\times\FFaces\to\Q_{\geq 0}$. 
In this section we construct the function $\pair_1$, whereas $\pair_2$ will be defined in \cref{sec:pair:2}. 
See \cref{sec:pair:rmk} to read about some heuristics for the complicated definition of $\pair$.

\subsection{Seven types of vertices}

In order to better describe the construction of $\pair$, we divide the vertices into seven categories, which we call \emph{types}:
 $\ddface$-vertices, $\dface$-vertices, big vertices, regular vertices, semi-regular vertices, TS-vertices and potentially-special vertices.
In the next paragraphs we will explain how to assign a type to each $v\in \Verts$, the attribution depending solely on the face vector $\vtype(v)$. 
Meanwhile, we will describe for every $v\in\Verts$ how to build the functions $(\pair_1)|_{\{v\}\times \FFaces}$ and $(\pair_2)|_{\{v\}\times \FFaces}$.
The first will usually depend only on $\vfaces(v)$, while the second will often depend on more detailed geometric and combinatorial configurations surrounding the vertex $v$.
The following facts will be true:
\begin{description}
\item[(F1)]
for every $\ddface$-vertex $v$ we have $\pair_1(v,\ddface)=1$;
\item[(F2)]
for every $\dface$-vertex $v$ we have $\pair_1(v,\dface)=1$;
\item[(F3)]
for every regular vertex $v$ there is exactly one $\face\in\vfaces(v)$ with $\pair_1(v,\face)=1$;
\item[(F4)]
for every semi-regular vertex $v$ and every $\face\in\FFaces$ we have $\pair_2(v,\face)=0$; moreover if $\pair_1(v,\face)\neq 0$, then $\face\in\vfaces(v)$ or $\face=\ddface$;
\item[(F5)]
for every TS-vertex $v$ and every $\face\in\FFaces$ we have $\pair_1(v,\face)=0$.
\end{description}

The case of TS-vertices and potentially-special vertices will be longer to describe, so we will take care of them in the next section.
 
\subsection{$\dface$-vertices}\label{sec:pair:1:dface}

We say that $v\in\Verts$ is a $\dface$-vertex if and only if  $\vtype(v)=(5,6,7)$ or $\vtype(v)=(3,3,5,7)$. For every $\dface$-vertex $v$ we define $\pair_1(v,\dface)=1$ and $\pair_2(v,\dface)=0$, while for every $\face\in\FFaces$ with $\face\neq\dface$ we define $\pair_1(v,\face)=\pair_2(v,\face)=0$. In \cref{table:dface} we list the possible face vectors of $\dface$-vertices, together with a mnemonic for their pairings.

\begin{table}[H]\centering
\begin{tabular}{cccc}
\toprule
$\vtype(v)$
	\ \ \ &\ \ \ 
 
	\ \ \ &\ \ \ 
$\pair_1(v,\cdot)$
	\ \ \ &\ \ \ 
$\pair_2(v,\cdot)$
\\

\midrule

$(5,6,7)$
	\ \ \ &\ \ \ 

\ \ \ &\ \ \ 
$1\cdot [\dface]$
\ \ \ &\ \ \ 
$-$
\\

\midrule

$(3,3,5,7)$
	\ \ \ &\ \ \ 

\ \ \ &\ \ \ 
$1\cdot [\dface]$
\ \ \ &\ \ \ 
$-$
\\

\bottomrule
\end{tabular}
\caption{Mnemonics for $\dface$-vertices.}
\label{table:dface}
\end{table}

\subsection{$\ddface$-vertices}

We say that $v\in\Verts$ is a $\ddface$-vertex if and only if  $\vtype(v)$ is listed in the following table. For every $\ddface$-vertex $v$ we define $\pair_1(v,\ddface)=1$ and $\pair_2(v,\ddface)=0$, while for every $\face\in\FFaces$ with $\face\neq\ddface$ we define $\pair_1(v,\face)=\pair_2(v,\face)=0$. In \cref{table:ddface} we list the possible face vectors of $\ddface$-vertices, together with a mnemonic for their pairings.

\begin{table}[H]\centering
\begin{tabular}{cccc}
\toprule
$\vtype(v)$
	\ \ \ &\ \ \ 
where 
	\ \ \ &\ \ \ 
$\pair_1(v,\cdot)$
	\ \ \ &\ \ \ 
$\pair_2(v,\cdot)$
	\\

\midrule

$(3,3,a)$
	\ \ \ &\ \ \ 
$5\leq a\leq 10$
\ \ \ &\ \ \ 
$1\cdot [\ddface]$
\ \ \ &\ \ \ 
$-$
\\

\midrule

$(3,5,a)$
	\ \ \ &\ \ \ 
$5\leq a\leq 10$
\ \ \ &\ \ \ 
$1\cdot [\ddface]$
\ \ \ &\ \ \ 
$-$
\\

\midrule

$(3,6,a)$
	\ \ \ &\ \ \ 
$6\leq a\leq 10,\ a\neq 7$
\ \ \ &\ \ \ 
$1\cdot [\ddface]$
\ \ \ &\ \ \ 
$-$
\\

\midrule

$(3,a,12)$
	\ \ \ &\ \ \ 
$5\leq a\leq 10$
\ \ \ &\ \ \ 
$1\cdot [\ddface]$
\ \ \ &\ \ \ 
$-$
\\

\midrule

$(3,a,13)$
	\ \ \ &\ \ \ 
$5\leq a\leq 10$
\ \ \ &\ \ \ 
$1\cdot [\ddface]$
\ \ \ &\ \ \ 
$-$
\\

\midrule

$(3,a,19)$
	\ \ \ &\ \ \ 
$a\in\{3,6,7,8\}$
\ \ \ &\ \ \ 
$1\cdot [\ddface]$
\ \ \ &\ \ \ 
$-$
\\

\midrule

$(4,4,a)$
	\ \ \ &\ \ \ 
$6\leq a\leq 41,\, a\not\in\{7,11,13,19\}$
\ \ \ &\ \ \ 
$1\cdot [\ddface]$
\ \ \ &\ \ \ 
$-$
\\

\midrule

$(4,6,a)$
	\ \ \ &\ \ \ 
$a\in\{6,8,9,10\}$
\ \ \ &\ \ \ 
$1\cdot [\ddface]$
\ \ \ &\ \ \ 
$-$
\\

\midrule

$(5,5,a)$
	\ \ \ &\ \ \ 
$5\leq a\leq 9$
\ \ \ &\ \ \ 
$1\cdot [\ddface]$
\ \ \ &\ \ \ 
$-$
\\

\midrule

$(5,6,6)$
	\ \ \ &\ \ \ 

\ \ \ &\ \ \ 
$1\cdot [\ddface]$
\ \ \ &\ \ \ 
$-$
\\

\midrule

$(3,3,3,19)$
	\ \ \ &\ \ \ 

\ \ \ &\ \ \ 
$1\cdot [\ddface]$
\ \ \ &\ \ \ 
$-$
\\

\midrule

$(3,3,4,a)$
	\ \ \ &\ \ \ 
$8\leq a\leq 10$
\ \ \ &\ \ \ 
$1\cdot [\ddface]$
\ \ \ &\ \ \ 
$-$
\\

\midrule

$(3,3,5,a)$
	\ \ \ &\ \ \ 
$a\in\{5,6\}$
\ \ \ &\ \ \ 
$1\cdot [\ddface]$
\ \ \ &\ \ \ 
$-$
\\

\bottomrule
\end{tabular}
\caption{Mnemonics for $\ddface$-vertices.}
\label{table:ddface}
\end{table}

\subsection{Big vertices}

We say that $v\in\Verts$ is a big vertex if and only if there is $N\in\vtype(v)$  with $N\geq 42$.
For every big-vertex $v$ we define $\pair_1(v,\ddface)=1$ and $\pair_2(v,\ddface)=0$, while for every $\face\in\FFaces$ with $\face\neq\ddface$ we define $\pair_1(v,\face)=\pair_2(v,\face)=0$. 
In \cref{table:big} we list the possible face vectors of big-vertices, together with a mnemonic for their pairings.%

\begin{table}[H]\centering
\begin{tabular}{cccc}
\toprule
$\vtype(v)$
	\ \ \ &\ \ \ 
where 
	\ \ \ &\ \ \ 
$\pair_1(v,\cdot)$
	\ \ \ &\ \ \ 
$\pair_2(v,\cdot)$
	\\

\midrule

$(3,a,N)$
	\ \ \ &\ \ \ 
$3\leq a\leq 6$, $N\geq 42$
\ \ \ &\ \ \ 
$1\cdot [\ddface]$
\ \ \ &\ \ \ 
$-$
\\

\midrule

$(4,4,N)$
	\ \ \ &\ \ \ 
$N\geq 42$
\ \ \ &\ \ \ 
$1\cdot [\ddface]$
\ \ \ &\ \ \ 
$-$
\\

\midrule

$(3,3,3,N)$
	\ \ \ &\ \ \ 
$N\geq 42$
\ \ \ &\ \ \ 
$1\cdot [\ddface]$
\ \ \ &\ \ \ 
$-$
\\

\bottomrule
\end{tabular}
\caption{Mnemonics for big vertices.}
\label{table:big}
\end{table}

We remark that, although \cref{table:admissible} shows that the above face vectors are arithmetically admissible, we will prove in \cref{thm:big} that actually a PCC cannot contain any big vertex. 

\subsection{Regular vertices}

We say that $v\in\Verts$ is a regular vertex if and only if  $\vtype(v)$ is listed in the following table. For every regular vertex $v$ we choose an integer $n_v$ appearing in $\vtype(v)$ with multiplicity one.
Then there exist a well defined $\face_v\in\vfaces(v)$ with $\abs{\face_v}=n_v$. We define $\pair_1(v,\face_v)=1$, $\pair_2(v,\face_v)=0$ and $\pair_1(v,\face)=\pair_2(v,\face)=0$ for every $\face\in\FFaces$ with $\face\neq\face_v$. In \cref{table:regular} we list the possible face vectors of regular vertices, together with a mnemonic for their pairings. More precisely, in the column relative to $\pair_1$ we write $1\cdot [n_v]$, where $n_v$ is the element of $\vtype(v)$ that we choose.

\begin{table}[H]\centering
\begin{tabular}{cccc}
\toprule
$\vtype(v)$
	\ \ \ &\ \ \ 
where 
	\ \ \ &\ \ \ 
$\pair_1(v,\cdot)$
	\ \ \ &
$\pair_2(v,\cdot)$
	\\

\midrule

$(3,3,a)$
	\ \ \ &\ \ \ 
$13\leq a\leq 41,\ a\neq 19$
\ \ \ &\ \ \ 
$1\cdot [a]$
\ \ \ &\ \ \ 
$-$
\\

\midrule

$(3,5,a)$
	\ \ \ &\ \ \ 
$a\in\{40,41\}$
\ \ \ &\ \ \ 
$1\cdot [a]$
\ \ \ &\ \ \ 
$-$
\\

\midrule

$(3,6,a)$
	\ \ \ &\ \ \ 
$14\leq a\leq 41,\ a\neq 19$
\ \ \ &\ \ \ 
$1\cdot [a]$
\ \ \ &\ \ \ 
$-$
\\

\midrule

$(3,7,a)$
	\ \ \ &\ \ \ 
$14\leq a\leq 41,\ a\neq 19$
\ \ \ &\ \ \ 
$1\cdot [a]$
\ \ \ &\ \ \ 
$-$
\\

\midrule

$\{3,a,11\}$
	\ \ \ &\ \ \ 
$6\leq a\leq 12,\ a\neq 11$
\ \ \ &\ \ \ 
$1\cdot [11]$
\ \ \ &\ \ \ 
$-$
\\

\midrule

$(3,8,a)$
	\ \ \ &\ \ \ 
$14\leq a\leq 22,\ a\neq 19$
\ \ \ &\ \ \ 
$1\cdot [a]$
\ \ \ &\ \ \ 
$-$
%
\\
\midrule

$(3,9,a)$
	\ \ \ &\ \ \ 
$8\leq a\leq 10,\ 14\leq a\leq 17$
\ \ \ &\ \ \ 
$1\cdot [a]$
\ \ \ &\ \ \ 
$-$
\\

\midrule

$(3,10,14)$
	\ \ \ &\ \ \ 

\ \ \ &\ \ \ 
$1\cdot [a]$
\ \ \ &\ \ \ 
$-$
\\

%

\midrule

$(4,4,a)$
	\ \ \ &\ \ \ 
$a\in\{5,7,11,13,19\}$
\ \ \ &\ \ \ 
$1\cdot [a]$
\ \ \ &\ \ \ 
$-$
\\

\midrule 

$(4,5,a)$
	\ \ \ &\ \ \ 
$8\leq a\leq 18,\ a\neq 11$
\ \ \ &\ \ \ 
$1\cdot [5]$
\ \ \ &\ \ \ 
$-$
\\

\midrule

$(4,6,7)$
	\ \ \ &\ \ \ 

\ \ \ &\ \ \ 
$1\cdot [7]$
\ \ \ &\ \ \ 
$-$
\\

\midrule

$(4,6,11)$
	\ \ \ &\ \ \ 

\ \ \ &\ \ \ 
$1\cdot [11]$
\ \ \ &\ \ \ 
$-$
\\

\midrule

$(4,7,a)$
	\ \ \ &\ \ \ 
$8\leq a\leq 9$
\ \ \ &\ \ \ 
$1\cdot [7]$
\ \ \ &\ \ \ 
$-$
\\

\midrule
 
$(3,3,3,a)$
	\ \ \ &\ \ \ 
$13\leq a\leq 41,\ a\neq 19$
\ \ \ &\ \ \ 
$1\cdot [a]$
\ \ \ &\ \ \ 
$-$
\\

\midrule

$(3,3,4,11)$
	\ \ \ &\ \ \ 
\ \ \ &\ \ \ 
$1\cdot [11]$
\ \ \ &\ \ \ 
$-$
\\

\bottomrule
\end{tabular}
\caption{Mnemonics for regular vertices.}
\label{table:regular}
\end{table}

\subsection{Semi-regular vertices}

We say that $v\in\Verts$ is a semi-regular vertex if and only if  $\vtype(v)$ is listed in the following table. We define $\pair_2(v,\face)=0$ for every semi-regular vertex $v$ and every $\face\in\FFaces$. For every semi-regular vertex $v$ we choose either only one integer $n_v\in\vtype(v)$, or a rational number $0 \less r_v \less 1$ and two integers $m_v,n_v$ appearing in $\vtype(v)$. Four possibilities can occur.
\begin{enumerate}[(i)]
\item 
If we choose only one integer $n_v$ appearing in $\vtype(v)$ with multiplicity one, then there exists a well defined $\face_v\in\vfaces(v)$ with $\abs{\face_v}=n_v$. We define $\pair_1(v,\face_v)=\frac{1}{2}$, $\pair_1(v,\ddface)=\frac{1}{2}$ and $\pair_1(v,\face)=0$ for every $\face\in\FFaces\setminus\{\face_v,\ddface\}$.
\item
If we choose only one integer $n_v$ appearing in $\vtype(v)$ with multiplicity two, it may happen that there exists a single face $\face_v\in\Faces$ with $\abs{\face_v}=n_v$ appearing in $\vfaces(v)$ with multiplicity two. In this case we define $\pair_1(v,\face_v)=\frac{1}{2}+\frac{1}{2}=1$, and $\pair_1(v,\face)=0$ for every $\face\in\FFaces\setminus\{\face_v\}$.
\item 
If we choose only one integer $n_v$ appearing in $\vtype(v)$ with multiplicity two, it may also happen that there exist two distinct elements $\face_{v,1},\face_{v,2}\in\vfaces(v)$ with $\abs{\face_{v,1}}=\abs{\face_{v,2}}=n_v$ each appearing with multiplicity one. In this case we define $\pair_1(v,\face_{v,1})=\pair_1(v,\face_{v,2})=\frac{1}{2}$, and $\pair_1(v,\face)=0$ for every $\face\in\FFaces\setminus\{\face_{v,1},\face_{v,2}\}$.
\item
If we choose two integers $m_v,n_v$ appearing in $\vtype(v)$, each with multiplicity one, then there exist two well defined and distinct elements $\face_{v,1},\face_{v,2}\in\vfaces(v)$ with $\abs{\face_{v,1}}=m_v$, $\abs{\face_{v,2}}=n_v$. In this case we define $\pair_1(v,\face_{v,1})=r_v$, $\pair_1(v,\face_{v,2})=1-r_v$, and $\pair_1(v,\face)=0$ for every $\face\in\FFaces\setminus\{\face_{v,1},\face_{v,2}\}$.
\end{enumerate} 
In \cref{table:semi-regular} we list the possible face vectors of semi-regular vertices, together with a mnemonic for their pairings. More precisely, in case (i) we write $\frac{1}{2} [n_v]+\frac{1}{2}[\ddface]$, in case (ii) and (iii) we write $\frac{1}{2} [n_v]+\frac{1}{2}[n_v]'$, and in case (iv) we write $r_v [m_v]+(1-r_v)[n_v]$, where $m_v,n_v$ and $r_v$ are the chosen numbers. In most of the cases the choices depend only on $\vtype(v)$. Only if $\vtype(v)=(3,11,13)$ the number $r_v$ depends on $v$ being an $\alpha$-vertex or a $\beta$-vertex (see \cref{def:blue}).

\begin{table}[H]\centering
\begin{tabular}{cccc}
\toprule
$\vtype(v)$
	\ \ \ &\ \ \ 
where 
	\ \ \ &\ \ \ 
$\pair_1(v,\cdot)$
	\ \ \ &\ \ \ 
$\pair_2(v,\cdot)$
	\\

\midrule

$(3,5,11)$
	\ \ \ &\ \ \ 

\ \ \ &\ \ \ 
$\frac{1}{2} [11]+\frac{1}{2}[\ddface]$
\ \ \ &\ \ \ 
$-$
\\

\midrule

$(3,5,a)$
	\ \ \ &\ \ \ 
$14\leq a\leq 19$
\ \ \ &\ \ \ 
$\frac{1}{2} [5]+\frac{1}{2}[a]$
\ \ \ &\ \ \ 
$-$
\\

\midrule

$(3,5,a)$
	\ \ \ &\ \ \ 
$20\leq a\leq 39$
\ \ \ &\ \ \ 
$\frac{1}{2} [a]+\frac{1}{2}[\ddface]$
\ \ \ &\ \ \ 
$-$
\\

\midrule

$(3,11,11)$
	\ \ \ &\ \ \ 

\ \ \ &\ \ \ 
$\frac{1}{2} [11]+\frac{1}{2}[11]'$
\ \ \ &\ \ \ 
$-$
\\

\midrule

$(3,11,13)$
	\ \ \ &\ \ \ 
$v$ is $\alpha$-vertex
\ \ \ &\ \ \ 
$\frac{1}{7} [11]+\frac{6}{7}[13]$
\ \ \ &\ \ \ 
$-$
\\

\midrule

$(3,11,13)$
	\ \ \ &\ \ \ 
$v$ is $\beta$-vertex
\ \ \ &\ \ \ 
$\frac{3}{7} [11]+\frac{4}{7}[13]$
\ \ \ &\ \ \ 
$-$
\\

\midrule

$(4,5,5)$
	\ \ \ &\ \ \ 

\ \ \ &\ \ \ 
$\frac{1}{2} [5]+\frac{1}{2}[5]'$
\ \ \ &\ \ \ 
$-$
\\

\midrule

$(4,5,a)$
	\ \ \ &\ \ \ 
$a\in\{7,11\}$
\ \ \ &\ \ \ 
$\frac{1}{2} [5]+\frac{1}{2}[a]$
\ \ \ &\ \ \ 
$-$
\\

\midrule

$(4,5,19)$
	\ \ \ &\ \ \ 

\ \ \ &\ \ \ 
$\frac{3}{4} [5]+\frac{1}{4}[19]$
\ \ \ &\ \ \ 
$-$
\\

\midrule

$(4,7,7)$
	\ \ \ &\ \ \ 

\ \ \ &\ \ \ 
$\frac{1}{2} [7]+\frac{1}{2}[7]'$
\ \ \ &\ \ \ 
$-$
\\

\bottomrule
\end{tabular}
\caption{Mnemonics for semi-regular vertices.}
\label{table:semi-regular}
\end{table}

We remark that if $\vtype(v)=(4,5,5)$ the possibility (ii) above cannot occur, by \cref{lemma:distinct}. 
It is possible to show that the same is true in case $\vtype(v)=(4,7,7)$, and that (ii) can occur for $\vtype(v)=(3,11,11)$ in only a  very limited list of PCC graphs. 
Anyway, this more precise analysis would not simplify our proof.

\subsection{Prototype for our discharging function}\label{sec:pair:rmk}
In this paragraph we spend some words to motivate the construction of the pairing $\pair:\Verts\times\FFaces\to\Q_{\geq 0}$.
We recall from \cref{sec:main:rmk} that, for our purposes, the ``bad'' vertices are those whose curvature is less than $2/209$. 
These are the vertices $v\in\Verts$ such that $\vtype(v)\in\mcl A_{4}\cup\mcl A_{5/7}\cup\mcl A_{7}\cup\mcl A_{8}\cup\mcl A_{9/10}\cup\mcl A_{11}$, where
\begin{align*}
 \mcl A_4 &:= \{(4,5,18),(4,5,19),(4,7,9)\};\\
 \mcl A_{5/7} &:= \{(5,6,7), (3,3,5,7)\};\\
 \mcl A_7 &:= \{(3,7,N):\ 30\leq N\leq 41\};\\
 \mcl A_8 &:= \{(3,8,N):\ 20\leq N\leq 23\};\\
 \mcl A_{9/10} &:= \{(3,9,16),(3,9,17),(3,10,14)\};\\
 \mcl A_{11} &:= \{(3,11,12),(3,11,13),(3,3,4,11), (6,4,11)\}.
\end{align*}
Our idea, for a nearly-optimal discharging of the curvature, is to fuzzy-partition the set $\Verts$ so that the partitions roughly correspond to the faces
\[
 \Faces_{\text{naive}}:= \{\face\in\Faces:\ \abs{\face}\in\{5,7,11\}\cup\{13,\dots,41\}\}.
\]
It is naturally attached to $\Faces_{\text{naive}}$ the pairing $\pair_{\text{naive}}:\Verts\times\Faces_{\text{naive}}\cup\{\ddface\}\to \Q_{\geq 0}$ such that: 
\begin{itemize}
 \item $\pair_{\text{naive}}(v,\face)=1$ if $v$ belongs to $\fverts(\face)$ for exactly one $\face\in\Faces_{\text{naive}}$; 
 \item $\pair_{\text{naive}}(v,\ddface)=1$ if $v$ doesn't belong to $\fverts(\face)$ for any $\face\in\Faces_{\text{naive}}$; 
 \item $\pair_{\text{naive}}(v,\face_i)=1/m$ whenever $v$ belongs to $\fverts(\face_i)$ for $m>1$  distinct faces $\face_1,\dots,\face_m\in\Faces_{\text{naive}}$. 
\end{itemize}
In particular $\ddface$ collects only vertices that are ``good''. 
The discharging rules for $\pair_{\text{naive}}$ should be compared with the construction of $\pair_1$, while $\Faces_{\text{naive}}$ should be compared with $\FFaces$.  
Notice that most of the faces in $\Faces_{\text{naive}}$ have odd cardinality: this ensures, for example, that it is not possible to have fuzzy-partitions consisting of only ``bad'' vertices with $\vtype(v)\in\mcl A_4$. 
However, it is possible to arrange $\dface$-vertices (i.e. $\vtype(v)\in\mcl A_{5/7}$) around a pentagon. 
That is why we take care of these vertices separately in \cref{sec:pair:1:dface} above and \cref{sec:210,sec:209,sec:208} below.  
Moreover, a priori it is also possible to arrange, say, only vertices with $\vtype(v)=(3,8,22)$ around a 22-agon (or with $f(v)=(3,7,41)$ around a 41-agon, etc.). 
However, when this happens then we necessarily have some ``good'' vertices with $\vtype(v)=(3,8,8)$ not far from the boundary of the 22-agon. 
Therefore it is advisable to add these vertices to the fuzzy-partition corresponding to the 22-agon (compare the rule \cref{rule:3ab} below and \cref{sec:N}). 
This order of ideas motivates the creation of ``special'' discharging rules as in the next section. 
Last but not least, in order to obtain extremely accurate estimates (compare e.g. the tight inequalities in \cref{sec:13}, where $c(\face)$ is proved to be only barely positive!) we carefully optimized some mission-critical discharging weights. 
For example, an $\alpha$-vertex only discharges $1/7$ of its curvature to the 11-sided face adjacent to it, while a vertex $v$ with $\vtype(v)=(4,5,19)$ discharges $3/4$ of its curvature to the pentagon and only $1/4$ to the 19-agon in $\vfaces(v)$.


\section{The pairing: part 2}\label{sec:pair:2}

A vertex $v\in\Verts$ is called \emph{special} if there exists $\face\in\FFaces\setminus\{\dface,\ddface\}$ with $\pair_2(v,\face)\neq 0$. We will also say that \emph{$v$ is special to $\face$} in this case. Only TS-vertices and potentially-special vertices can be special.
In the following paragraphs we will describe the necessary and sufficient conditions for a vertex to be special.
In order to simplify the exposition, we will indicate the values of $\pair_1(v,\face)$ and $\pair_2(v,\face)$, for $(v,\face)\in\Verts\times\FFaces$, only when they are nonzero.

\specialeTS\label{rule:TS}
A vertex $v\in\Verts$ is a \emph{TS-vertex} if and only if every face $\face\in\vfaces(v)$ is a triangle or a square; in other words, if $\forall n\in\vtype(v)$ we have $n\in\{3,4\}$. For a TS-vertex $v$ we consider:
\[
\begin{aligned}
\ETS(v)&:=\{e\in\Edges\setminus\vedges(v):\ \exists \tau\in\vfaces(v) \text{ with } e\in \fedges(\tau)\},\\  
\FTS(v)&:=\{\face\in\fverts(\face):\ \abs{\face}\in\{11,40,41\}\text{ and }\exists e\in\ETS(v) \text{ with } e\in \fedges(\face)\},\\
\end{aligned}
\]
and we let $\mTS(v):=\#\ETS(v)$ and $\nTS(v):=\#\FTS(v)$.

\begin{lemma}
If $v$ is a TS-vertex, we have $\nTS(v)\leq 3$.
\end{lemma}
\begin{proof}
We say that $e,e'\in\ETS(v)$ with $e\neq e'$ are consecutive, if they have a common vertex, i.e. if $\everts(e)\cap\everts(e')\neq\emptyset$. It's easy to show, with \cref{lemma:distinct:path:2}, that we can cyclically order the elements of $\ETS(v)$ as $e_1,\ld,e_{\mTS(v)}=:e_0$, so that $e_{i-1},e_{i}$ are consecutive, for $1\leq i\leq \mTS(v)$.
By \cref{lemma:distinct} we notice that a triangle $\tau\in\vfaces(v)$ has exactly one edge $e\in\fedges(\tau)\setminus\vedges(v)$, while a square has exactly two. By \cref{table:admissible}, we deduce that $\mTS(v)\leq 7$. Indeed, one could show, with \cref{lemma:distinct:path:2}, that we have $\mTS(v)=7$ if and only if $\vtype(v)=(3,4,4,4)$.
Moreover, we see that if $\face_1,\face_2\in\FTS(v)$ satisfy $e_{i-1}\in\fedges(\face_1)$, $e_{i}\in\fedges(\face_2)$ for some $1\leq i\leq \mTS(v)$, we must have $\face_1=\face_2$, because otherwise for $w\in\everts(e_{i-1})\cap\everts(e_{i})$ we would have $\{\face_1,\face_2\}\subseteq \vfaces(w)$, which is not admissible in a PCC graph.
Since $\mTS(v)\leq 7$, we deduce that $\nTS(v)\leq 3$.
\end{proof}

We now describe the pairing for $v$. If $\nTS(v)\geq 1$, then for every $\face\in\FTS(v)$ we set $\pair_2(v,\face)=\frac 1 3$, so $v$ is special to $\face$.
If otherwise $\nTS(v)=0$, then $v$ is not special.
Regardless of $v$ being special or not, we set $\pair_2(v,\ddface)=1-\frac{\nTS(v)} 3$.
In \cref{table:TS} we provide mnemonics for the above pairing rule.

\begin{table}[H]\centering
\begin{tabular}{ccccl}
\toprule
$\vtype(v)$
	\ &\ \ \ 
$\pair_1(v,\cdot)$
	\ \ \ &\ \ \ 
$\pair_2(v,\cdot)$
	\\

\midrule

$\{$only 3,4$\}$

 \ &\ \ \ 
$-$
\ \ \ &\ \ \ 
$\nTS(v)\times\left(\frac 1 3 \cdot [11/40/41]\right) + \frac{3-\nTS(v)} 3\cdot [\ddface]$
\\

\bottomrule
\end{tabular}
\caption{Mnemonics for TS-vertices.}
\label{table:TS}
\end{table}

\subsection{Potentially-special vertices}

A vertex is \emph{potentially-special} if and only if it is not a $\dface$-vertex, a $\ddface$-vertex, a big vertex, a regular vertex, a semi-regular vertex, nor a TS-vertex. In the next paragraphs we divide the potentially-special vertices according to their face vectors. In every group, we specify the conditions under which a potentially-special vertex $v$ is actually special, or not. In both cases, we describe the related pairing, thus competing the definition of our $\pair$. 
See also \cref{fig:special}.

\begin{figure}[ht]
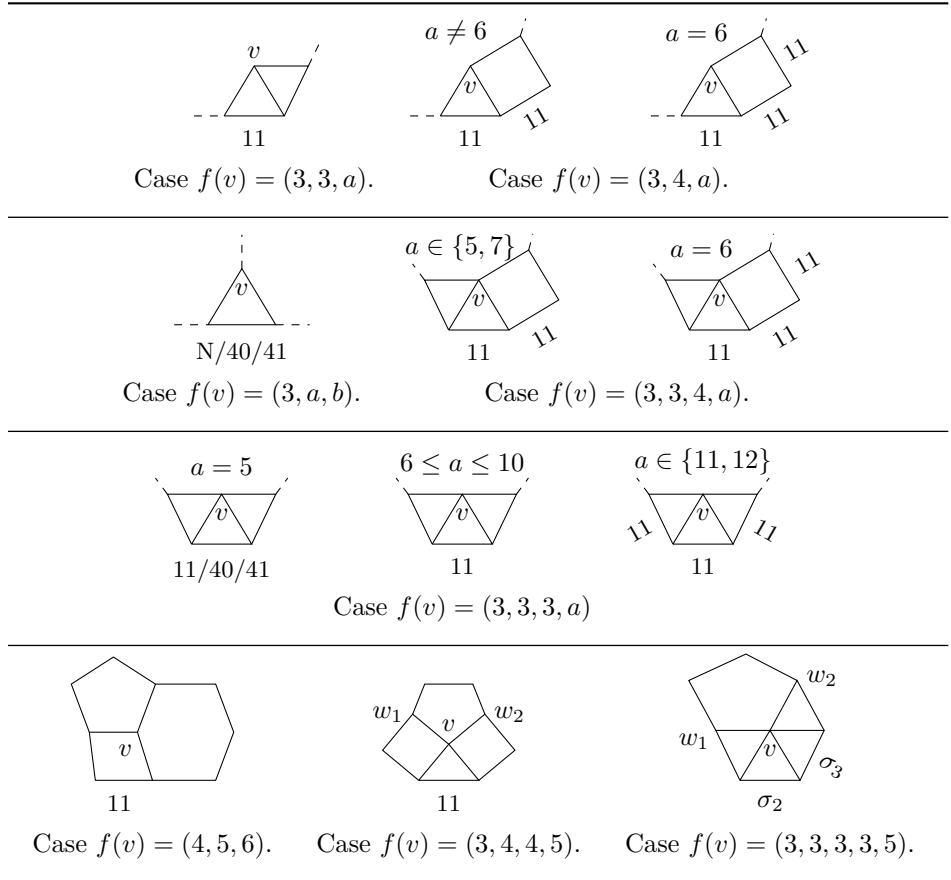
\centering
\begin{tabular}{c}
\toprule
\input{rule_33a.tikz}
\ 
\input{rule_34a.tikz}
\\
\midrule
\input{rule_3ab.tikz}
\ 
\input{rule_334a.tikz}
\\
\midrule 
\input{rule_333a.tikz}
\\
\midrule
\input{rule_456.tikz}
\ 
\input{rule_3445.tikz}
\ 
\input{rule_33335.tikz}
\\
\bottomrule
\end{tabular}
\caption{Illustration of the special rules in the definition of $\protect\pair$.}
\label{fig:special}
\end{figure}

\speciale{(3,3,a)}\label{rule:33a}

Let $v\in\Verts$ with $\vtype(v)=(3,3,a)$, where $a\in\{11,12\}$. 
Let $\vfaces(v)=(\tau_1,\tau_2,\face)$ and let $\face_i=\op{opp}(v,\tau_i)$ for $i\in\{1,2\}$.
\begin{description}
\item[Case $\abs{\face_j}=11$ for some $j\in\{1,2\}$] 
Then $v$ is special to $\face_j$ with $\pair_2(v,\face_j)=\frac 1 2$, if $a=11$, or $\pair_2(v,\face_j)= 1$, if $a=12$.
Notice that in case $\abs{\face_1}=\abs{\face_2}=11$ we must have $\face_1=\face_2$ since otherwise there would be $w\in\Verts$ with $(3,3,11,11)\subseteq\vtype(w)$, which is not admissible in a PCC graph.
\end{description}
In case $\abs{\face_1}\neq 11$ and $\abs{\face_2}\neq 11$,  $v$ is not special, and we set  $\pair_2(v,\ddface)=\frac 1 2$, if $a=11$, or $\pair_2(v,\ddface)= 1$, if $a=12$.
Regardless of $v$ being special or not, we set $\pair_1(v,\face)=\frac 1 2$ if $a=11$.
In \cref{table:33a} we provide mnemonics for the above pairing rules.
\begin{table}[H]\centering
\begin{tabular}{ccccc}
\toprule
$\vtype(v)$
	\ \ \ &\ \ \ 
where 
	\ \ \ &\ \ \ 
$\pair_1(v,\cdot)$
	\ \ \ &\ \ \ 
$\pair_2(v,\cdot)$
	\ \ \ &\ \ \ 
	\\

\midrule

$(3,3,11)$
	\ \ \ &\ \ \ 

\ \ \ &\ \ \ 
$\frac 1 2 \cdot [11]$
\ \ \ &\ \ \ 
$\frac 1 2 \cdot [11/\ddface]$
\ \ \ &\ \ \ 
\\

\midrule

$(3,3,12)$
	\ \ \ &\ \ \ 

\ \ \ &\ \ \ 
$-$
\ \ \ &\ \ \ 
$1 \cdot [11/\ddface]$
\ \ \ &\ \ \ 
\\

\bottomrule
\end{tabular}
\caption{Mnemonics for the case (3,3,a).}
\label{table:33a}
\end{table}

\speciale{(3,4,a)}\label{rule:34a}

Let $v\in\Verts$ with $\vtype(v)=(3,4,a)$, where $5\leq a\leq 41$. 
Let $\vfaces(v)=(\tau,\kappa,\face)$ and $\fverts(\kappa)=\langle v,v_1,v_2,v_3\rangle$ with $\etype(v v_1)=(3,4)$.
Finally, let $\face_1=\op{opp}(v,\tau)$, $\efaces(v_1v_2)=\{\kappa,\face_2\}$ and $\efaces(v_2v_3)=\{\kappa,\face_3\}$.
\begin{description}
\item[Case $\abs{\face_j}=11$ for some $j\in\{1,2\}$, and $a\neq 6$] 
Then $v$ is special to $\face_j$ with $\pair_2(v,\face_j)=1$, if $a\in\{8,9,10,12\}$, or $\pair_2(v,\face_j)= \frac 1 2$ otherwise.
Notice that in case $\abs{\face_1}=\abs{\face_2}=11$ we must have $\face_1=\face_2$ since otherwise there would be $w\in\Verts$ with $(3,4,11,11)\subseteq\vtype(w)$, which is not admissible in a PCC graph.
\item [Case $\abs{\face_j}=11$ for some $j\in\{1,2,3\}$, and $a= 6$]
Consider the set 
\[\mcl A_v:=\{\face'\in\Faces:\ \abs{\face'}=11\text { and }\face'=\face_j\text{ for some } j\in\{1,2,3\}\},\]
and let $a_v:=\#\mcl A_v$.
Then $v$ is special to $\face'$ for every $\face'\in\mcl A_v$, with $\pair_2(v,\pair')=\frac 1 2$.
It's clear that $a_v\leq 3$ and moreover $a_v\neq 3$ because otherwise $(4,11,11)\subseteq \vtype(v_2)$, which is not admissible. 
\end{description}

In case $a\neq 6$, $\abs{\face_1}\neq 11$ and $\abs{\face_2}\neq 11$, $v$ is not special, and we set $\pair_2(v,\ddface)=1$, if $a\in\{8,9,10,12\}$, or $\pair_2(v,\ddface)= \frac 1 2$ otherwise.
In case $a= 6$ and $\abs{\face_i}\neq 11$ for all $i\in\{1,2,3\}$, $v$ is not special, and we set $\pair_2(v,\ddface)=1-\frac {a_v}2$.
Regardless of $v$ being special or not, we set $\pair_1(v,\face)=\frac 1 2$ if $a\not\in\{6,8,9,10,12\}$.
In \cref{table:34a} mnemonics for the above pairing rules.
\begin{table}[H]\centering
\begin{tabular}{cccc}
\toprule
$\vtype(v)$
	&
where 
	&
$\pair_1(v,\cdot)$
	&
$\pair_2(v,\cdot)$
\\

\midrule

$(3,4,a)$
	&
$\begin{cases}5\leq a\leq 41 \\ a\not\in\{6, 8,9,10,12\}\end{cases}\!\!$
&
$\frac 1 2 \cdot [a]$
&
$\frac 1 2 \cdot [11/\ddface]$
\\

\midrule

$(3,4,6)$
	&

&
$-$
&
$\!\!\! a_v\times \left(\frac 1 2 \cdot [11]\right)+\frac {2-a_v} 2 \cdot [\ddface]$
\\

\midrule

$(3,4,a)$
	&
$a\in\{8,9,10,12\}$
&
$-$
&
$1 \cdot [11/\ddface]$
\\

\bottomrule
\end{tabular}
\caption{Mnemonics for the case (3,4,a).}
\label{table:34a}
\end{table}

\speciale{(3,a,b)}\label{rule:3ab}

Let $v\in\Verts$ with $\vtype(v)=(3,a,b)$, where $6\leq a\leq 10$ and $7\leq b\leq 10$. 
Let $\vfaces(v)=(\tau,\face_1,\face_2)$ and let $\face=\op{opp}(v,\tau)$.
\begin{description}
\item[Case $(a,b)=(6,7)$ and $\abs{\face}\in\{40,41\}$] 
Then $v$ is special to $\face$ with $\pair_2(v,\face)=1$.
\item[Case $(a,b)\neq (6,7)$, $14\leq\abs{\face}\leq 41$ and $\abs{\face}\neq 19$]
Then $v$ is special to $\face$, with $\pair_2(v,\face)=\frac 1 2$, if $(a,b)\in\{(7,8),(7,9)\}$, or $\pair_2(v,\face)= 1$ otherwise.
\end{description}
If the above cases don't apply, $v$ is not special, and we set $\pair_2(v,\ddface)=\frac 1 2$, if $(a,b)\in\{(7,8),(7,9)\}$, or $\pair_2(v,\ddface)= 1$. 
Regardless of $v$ being special or not, we set $\pair_1(v,\face_1)=\frac 1 2$ if $(a,b)\in\{(7,8),(7,9)\}$.
In \cref{table:3ab} we provide mnemonics for the above pairing rules.
\begin{table}[H]\centering
\begin{tabular}{cccc}
\toprule
$\vtype(v)$
	\ \ \ &\ \ \ 
where 
	\ \ \ &\ \ \ 
$\pair_1(v,\cdot)$
	\ \ \ &\ \ \ 
$\pair_2(v,\cdot)$
	\\

\midrule

$(3,6,7)$
	\ \ \ &\ \ \ 

\ \ \ &\ \ \ 
$-$
\ \ \ &\ \ \ 
$1 \cdot [40/41/\ddface]$
\\

\midrule

$(3,7,7)$
	\ \ \ &\ \ \ 

\ \ \ &\ \ \ 
$-$
\ \ \ &\ \ \ 
$1 \cdot [N/40/41/\ddface]$
\\

\midrule

$(3,7,a)$
	\ \ \ &\ \ \ 
$a\in \{8,9\}$
\ \ \ &\ \ \ 
$\frac 1 2 \cdot [7]$
\ \ \ &\ \ \ 
$\frac 1 2 \cdot [N/\ddface]$
\\

\midrule
$(3,7,10)$
	\ \ \ &\ \ \ 
\ \ \ &\ \ \ 
$-$
\ \ \ &\ \ \ 
$1 \cdot [N/\ddface]$
\\

\midrule

$(3,a,b)$
	\ \ \ &\ \ \ 
$8\leq a\leq b\leq 10$
\ \ \ &\ \ \ 
$-$
\ \ \ &\ \ \ 
$1 \cdot [N/\ddface]$
\\

\bottomrule
\end{tabular}
\caption{Mnemonics for the case (3,a,b).}
\label{table:3ab}
\end{table}

The letter $N$ stands for $14\leq N\leq 39$, but $N\neq 19$, as in \cref{sec:N}.
Notice that necessarily if $b=8$, then $\abs{\face}\leq 23$, if $b=9$, then $\abs{\face}\leq 17$, and if $b=10$, then $\abs{\face}\leq 14$, by \cref{table:admissible}.

\speciale{(4,5,6)}\label{rule:456}

Let $v\in\Verts$ with $\vtype(v)=(4,5,6)$. 
Let $\vfaces(v)=(\kappa,\face,\face')$, so that $\abs{\kappa}=4$ and $\abs{\face}=5$.
Let $\fverts(\kappa) = \langle v, v_1, v_2, v_3 \rangle$ with $\etype(vv_1)=(4,6)$ and finally
let $\efaces(v_1v_2) = \{\kappa,\face_1\}$. 

\begin{description}
\item[Case $\abs{\face_1}=11$]
Then $v$ is special to $\face_1$ with $\pair_2(v,\face_1)=\frac 1 2$.
\end{description}

If the above condition doesn't hold, $v$ is not special and we set $\pair_2(v,\face_1)=\frac 1 2$.
Regardless of $v$ being special or not, we put $\pair_1(v,\face)=\frac 1 2$.
In \cref{table:456} we provide mnemonics for the above pairing rules.
\begin{table}[H]\centering
\begin{tabular}{cccc}
\toprule
$\vtype(v)$
	\ \ \ &\ \ \ 
where 
	\ \ \ &\ \ \ 
$\pair_1(v,\cdot)$
	\ \ \ &\ \ \ 
$\pair_2(v,\cdot)$
	\\

\midrule

$(4,5,6)$
	\ \ \ &\ \ \ 

\ \ \ &\ \ \ 
$\frac 1 2\cdot [5]$
\ \ \ &\ \ \ 
$\frac 1 2\cdot [11/\ddface]$
\\

\bottomrule
\end{tabular}
\caption{Mnemonics for the case (4,5,6).}
\label{table:456}
\end{table}

\speciale{(3,3,3,a)}\label{rule:333a}

Let $v\in\Verts$ with $\vtype(v)=(3,3,3,a)$, where $5\leq a\leq 12$. 
Let $\vfaces(v)=\langle\tau_1,\tau_2,\tau_3,\face\rangle$ with $\abs{\face}=a$, and let $\face_i=\op{opp}(v,\tau_i)$ for $i\in\{1,2,3\}$. If $a=11$ we let $r=\frac 1 3$, while if $a=12$ we let $r=\frac 1 2$.
\begin{description}
\item[Case $a=5$ and $\abs{\face_2}\in\{11,40,41\}$] 
Then $v$ is special to $\face_2$ with $\pair_2(v,\face_2)=1$.
\item[Case $a\in\{11,12\}$ and $\abs{\face_j}=11$ for some $j\in\{1,2,3\}$]
Then $v$ is special to the elements of $\mcl A = \{\face_j:\ 1\leq j\leq 3, \abs{\face_j}=11\}$. 
In this case we consider $\mcl A$ as a subset of $\Faces$, and not as a multiset. We notice that $\#\mcl A\neq 3$ because otherwise there would be some $w\in\Verts$ with $(3,3,11,11)\subseteq\vtype(w)$, which is not admissible in a PCC graph. 
In case $\#\mcl A=2$, then we put $\pair_2(v,\alpha)=r$ for all $\alpha\in\mcl A$.
In case $\#\mcl A=1$, we put $\pair_2(v,\alpha)=\pair_2(v,\ddface)=r$, where $\mcl A=\{\alpha\}$.
\item[Case $6\leq a\leq 10$ and $\abs{\face_2}=11$] 
Then $v$ is special to $\face_2$ with $\pair_2(v,\face_2)=1$.
\end{description}
If the above cases don't apply, then $v$ is not special, and we set $\pair_2(v,\ddface)=\frac 2 3$ if $a=11$, or $\pair_2(v,\ddface)=1$ otherwise.
Regardless of $v$ being special or not, we set $\pair_1(v,\face)=\frac 1 3$ if $a=11$.
In \cref{table:333a} we provide mnemonics for the above pairing rules.
\begin{table}[H]\centering
\begin{tabular}{cccc}
\toprule
$\vtype(v)$
	\ \ \ &\ \ \ 
where 
	\ \ \ &\ \ \ 
$\pair_1(v,\cdot)$
	\ \ \ &\ \ \ 
$\pair_2(v,\cdot)$
	\\

\midrule

$(3,3,3,5)$
	\ \ \ &\ \ \ 

\ \ \ &\ \ \ 
$-$
\ \ \ &\ \ \ 
$1 \cdot [11/40/41/\ddface]$
\\

\midrule

$(3,3,3,a)$
	\ \ \ &\ \ \ 
$6\leq a \leq 10$
\ \ \ &\ \ \ 
$-$
\ \ \ &\ \ \ 
$1 \cdot [11/\ddface]$
\\

\midrule

$(3,3,3,11)$
	\ \ \ &\ \ \ 

\ \ \ &\ \ \ 
$\frac 1 3 [11]$
\ \ \ &\ \ \ 
$\frac 1 3 \cdot [11/\ddface]+\frac 1 3 \cdot [11/\ddface]$
\\

\midrule 

$(3,3,3,12)$
	\ \ \ &\ \ \ 

\ \ \ &\ \ \ 
$-$
\ \ \ &\ \ \ 
$\frac 1 2 \cdot [11/\ddface]+\frac 1 2 \cdot [11/\ddface]$
\\

\bottomrule
\end{tabular}
\caption{Mnemonics for the case (3,3,3,a).}
\label{table:333a}
\end{table}

\speciale{(3,3,4,a)}\label{rule:334a}

This set of rules is similar to \cref{rule:34a}.
Let $v\in\Verts$ with $\vtype(v)=\langle 3,3,4,a\rangle$, where $a\in\{5,6,7\}$.
Let $\vfaces(v)=\langle\tau_1,\tau_2,\kappa,\face\rangle$ with $\abs{\kappa}=4$ and $\abs{\face}=a$,
and let $\fverts(\kappa)=\langle v,v_1,v_2,v_3\rangle$ with $\etype(v v_1)=(3,4)$.
Then, let $\face_1=\op{opp}(v,\tau_2)$,   $\efaces(v_1v_2)=\{\kappa,\face_2\}$ and $\efaces(v_2v_3)=\{\kappa,\face_3\}$.
Moreover, if $a=5$ let $r=\frac 1 2$, and if $a=7$ let $r=\frac 3 4$.
\begin{description}
\item[Case $a\in\{5,7\}$ and $\abs{\face_j}=11$ for some $j\in\{1,2\}$]
Then $v$ is special to $\face_j$, with $\pair_2(v,\face_j)=r$.
Notice that if $\abs{\face_1}=\abs{\face_2}=11$, then actually $\face_1=\face_2$, since otherwise $(3,4,11,11)\subseteq\vtype(v_1)$, which is not admissible.
\item[Case $a=6$ and $\abs{\face_j}=11$ for some $j\in\{1,2,3\}$]
Consider the set 
\[\mcl A_v:=\{\face'\in\Faces:\ \abs{\face'}=11\text { and }\face'=\face_j\text{ for some } j\in\{1,2,3\}\},\]
and let $a_v:=\#\mcl A_v$.
Then $v$ is special to $\face'$ for every $\face'\in\mcl A_v$, with $\pair_2(v,\pair')=\frac 1 2$.
It's clear that $a_v\leq 3$ and moreover $a_v\neq 3$ because otherwise $(4,11,11)\subseteq \vtype(v_2)$, which is not admissible. 
\end{description}
If $\vtype(v)=\langle 3,4,3,a\rangle$ or if $\vtype(v)=\langle 3,3,4,a\rangle$ but the above cases don't apply, $v$ is not special. 
If $v$ is not special and $a\in\{5,7\}$ we set $\pair_2(v,\ddface)=r$. 
Regardless of $v$ being special or not, we set $\pair_2(v,\ddface)=1-\frac{a_v}2$ if $a=6$.
Regardless of $v$ being special or not, we set $\pair_1(v,\face)=1-r$ if $a\in\{5,7\}$.
In \cref{table:334a} we provide mnemonics for the above pairing rules.
\begin{table}[H]\centering
\begin{tabular}{cccc}
\toprule
$\vtype(v)$
	\ \ \ &\ \ \ 
where 
	\ \ \ &\ \ \ 
$\pair_1(v,\cdot)$
	\ \ \ &\ \ \ 
$\pair_2(v,\cdot)$
	\\

\midrule

$(3,3,4,5)$
	\ \ \ &\ \ \ 

\ \ \ &\ \ \ 
$\frac 1 2 \cdot [5]$
\ \ \ &\ \ \ 
$\frac 1 2 \cdot [11/\ddface]$
\\

\midrule

$(3,3,4,6)$
	\ \ \ &\ \ \ 

\ \ \ &\ \ \ 
$-$
\ \ \ &\ \ \ 
$a_v\times \left(\frac 1 2 \cdot [11]\right)+\frac {2-a_v} 2 \cdot [\ddface]$
\\

\midrule

$(3,3,4,7)$
	\ \ \ &\ \ \ 

\ \ \ &\ \ \ 
$\frac 1 4 \cdot[7]$
\ \ \ &\ \ \ 
$\frac 3 4 \cdot [11/\ddface]$
\\

\bottomrule
\end{tabular}
\caption{Mnemonics for the case (3,3,4,a).}
\label{table:334a}
\end{table}

\speciale{(3,4,4,5)}\label{rule:3445}

Let $v\in\Verts$ with $\vtype(v)=\langle 4,3,4,5\rangle$. 
Let $w_1,w_2\in\Verts$ with $w_1\neq w_2$ such that $\etype(vw_1)=\etype(vw_2)=(4,5)$, let $=\vfaces(v)=(\tau,\kappa_1,\kappa_2,\face)$ and let $\face_1= \op{opp}(v,\tau)$.
Let also $\mcl A:=\{(4,5,a):\ 14\leq a\leq 19\}$.
\begin{description}
\item[Case $\abs{\face_1}=11$ and $\vtype(w_1),\vtype(w_2)\not\in\mcl A$] 
Then $v$ is special to $\face_1$ with $\pair_2(v,\face_1)=1$.
\end{description}
If otherwise $\vtype(v)=\langle 3,4,4,5\rangle$, or if $\vtype(v)=\langle 4,3,4,5\rangle$ but the above condition doesn't hold, $v$ is not special and we set $\pair_1(v,\face)=1$.
In \cref{table:3445} we provide mnemonics for the above pairing rules.
\begin{table}[H]\centering
\begin{tabular}{cccc}
\toprule
$\vtype(v)$
	\ \ \ &\ \ \ 
where 
	\ \ \ &\ \ \ 
$\pair_1(v,\cdot)$
	\ \ \ &\ \ \ 
$\pair_2(v,\cdot)$
	\\

\midrule

$(3,4,4,5)$
	\ \ \ &\ \ \ 
$v$ is special
\ \ \ &\ \ \ 
$-$
\ \ \ &\ \ \ 
$1 \cdot [11]$
\\

\midrule

$(3,4,4,5)$
	\ \ \ &\ \ \ 
$v$ is not special
\ \ \ &\ \ \ 
$1\cdot [5]$
\ \ \ &\ \ \ 
$-$
\\

\bottomrule
\end{tabular}
\caption{Mnemonics for the case (3,4,4,5).}
\label{table:3445}
\end{table}

\speciale{(3,3,3,3,5)}\label{rule:33335}

Let $v\in\Verts$ with $\vtype(v)=(3,3,3,3,5)$. 
Let $\vfaces(v)=\langle \tau_1,\tau_2,\tau_3,\tau_4,\face\rangle$ with $\abs{\face}=5$, and let $\mcl A:=\{(3,4,5),(3,3,4,5),(3,4,4,5)\}$.
For $i\in\{1,4\}$ let $w_i\in\Verts$ such that $\efaces(vw_i)=\{\tau_i,\face\}$, and for $j\in\{2,3\}$ let $\face_j= \op{opp}(v,\tau_j)$.
\begin{description}
\item[Case $\vtype(w_1)\in\mcl A$ and $\abs{\face_2}=11$] 
Then $v$ is special to $\face_2$ with $\pair_2(v,\face_2)=1$.
\item[Case $\vtype(w_4)\in\mcl A$ and $\abs{\face_3}=11$] 
Then $v$ is special to $\face_3$ with $\pair_2(v,\face_3)=1$.
We notice that if $\abs{\face_2}=\abs{\face_3}=11$ we necessarily have $\face_2=\face_3$.
\item[Case $\abs{\face_j}\in\{40,41\}$ for some $j\in\{2,3\}$]
Then $v$ is special to $\face_j$ with $\pair_2(v,\face_j)=1$.
\end{description}
If the above conditions don't hold, $v$ is not special and we set $\pair_2(v,\ddface)=1$.
Notice that if both $\abs{\face_2},\abs{\face_3}\in\{11,40,41\}$, we must have $\face_2=\face_3$.
In \cref{table:33335} we provide mnemonics for the above pairing rules.
\begin{table}[H]\centering
\begin{tabular}{cccc}
\toprule
$\vtype(v)$
	\ \ \ &\ \ \ 
where 
	\ \ \ &\ \ \ 
$\pair_1(v,\cdot)$
	\ \ \ &\ \ \ 
$\pair_2(v,\cdot)$
	\\

\midrule

$(3,3,3,3,5)$
	\ \ \ &\ \ \ 

\ \ \ &\ \ \ 
$-$
\ \ \ &\ \ \ 
$1\cdot [11/40/41/\ddface]$
\\

\bottomrule
\end{tabular}
\caption{Mnemonics for the case (3,3,3,3,5).}
\label{table:33335}
\end{table}


\section{Analysis of faces with 5 edges}\label{sec:5}

Let $\face\in \FFaces$ with $\abs{\face}=5$ and $\pair(\face)\neq 0$.
With our construction of the pairing, we have $\pair(v,\face)\neq 0$ if and only if $v\in\fverts(\face)$,  $v$ is not a special vertex of type $(3,4,4,5)$, and either $4\in \vtype(v)$ or $\vtype(v)=(3,5,a)$ with $14\leq a\leq 19$. We list all the possibilities in \cref{table:5}, and we define the constants $c_1=-0.00525$, 
 $c_2=0.038$, 
$c_3=0.023$, 
$c_4=0.015$. 

\begin{table}[H]\centering
\begin{tabular}{ccccl}
\toprule
$\vtype(v)$
	\ \ \ &\ \ \ 
\text{where}
	\ &\
$\pair(v,\face)$
	\ \ \ & \quad
 $c_v\pair(v,\face)$ &  \\

\midrule

$(3,4,5)$
	\ \ \ &\ \ \ 

	 \ & \ 
$\frac{1}{2}$
	\ \ \ &\ \ \ 
$\great 0.136$ &\\

\midrule

$(3,5,a)$
	\ \ \ &\ \ \ 
$14\leq a\leq 19$
	 \ & \ 
$\frac{1}{2}$
	\ \ \ &\ \ \ 
$\great 0.038$ &$= c_2$\\

\midrule

$\{4,5,a\}$
	\ \ \ &\ \ \ 
$8\leq a\leq 13, a\neq 11$
	 \ & \ 
$1$
	\ \ \ &\ \ \ 
$\great 0.017$ &\\

\midrule

$(4,5,5)$
	\ \ \ &\ \ \ 

	 \ & \ 
$\frac{1}{2}$ or $1$
	\ \ \ &\ \ \ 
$\great 0.070$ &\\

\midrule

$(4,5,6)$
	\ \ \ &\ \ \ 

	 \ & \ 
$\frac{1}{2}$
	\ \ \ &\ \ \ 
$\great 0.053$ &\\

\midrule

$(4,5,7)$
	\ \ \ &\ \ \ 

	 \ & \ 
$\frac{1}{2}$
	\ \ \ &\ \ \ 
$\great 0.041$ &\\

\midrule

$(4,5,11)$
	\ \ \ &\ \ \ 

	 \ & \ 
$\frac{1}{2}$
	\ \ \ &\ \ \ 
$\great 0.015$ &$= c_4$\\

\midrule

$(4,5,a)$
	\ \ \ &\ 
$14\leq a\leq 18$
	 \ & \ 
$1$
	\ \ \ &\ \ \
$\great -0.00402$ &\\

\midrule
$(4,5,19)$
	\ \ \ &\ \ \ 

	 \ & \ 
$\frac{3}{4}$
	\ \ \ &\ \ \ 
$\great -0.00521$ &$= c_1$\\

\midrule 

$(3,3,4,5)$
	\ \ \ &\ \ \ 
	
	 \ & \ 
$\frac{1}{2}$
	\ \ \ &\ \ \ 
$\great 0.053$ &\\

\midrule

$(3,4,4,5)$
	\ \ \ &\ \ \ 
	\text{$v$ not special}
	 \ & \ 
$1$
	\ \ \ &\ \ \ 
$\great 0.023$ &$= c_3$\\

\bottomrule
\end{tabular}
\caption{Vertices contributing nontrivially when $\protect\abs\face=5.$}
\label{table:5}
\end{table}

\vspace{3pt}

Let $A=\#\mcl{A}$, $B=\#\mcl{B}$ and $L=\#\mcl{L}$, where
\[
\begin{aligned}
\mcl{A}&:=\{v\in\fverts(\face):\ \vtype(v)=(4,5,a), \ 14\leq a\leq 19\},\\  
\mcl{B}&:=\{v\in\fverts(\face):\ \vtype(v)=(3,5,a), \ 14\leq a\leq 19\},\\
\mcl{L}&:=\{e\in\fedges(\face):\ \etype(e)=(5,a),\ 14\leq a \leq 19\}.
\end{aligned}
\]
We notice immediately that $c_-(\face)\great A c_1$; moreover, $2L=A+B\leq 4$, since it is both even and less than 5.
We now prove that $c(\face)\great 0$ considering the following 4 cases.

\begin{figure}[ht]\centering 
\input{5_A2.tikz}
\quad\quad
\input{5_A4.tikz}
\caption{Illustrations for $\protect\abs \face = 5$.}
\label{fig:5}
\end{figure}

\begin{description}
\item[Case $A=0$]
Here $c(\face)=c_+(\face)$ and so, by inspection of the table, 
$c(\face)\great c_4=0.015$.
\item[Case $B\geq 1$]
In this case $\exists v\in\mcl{B}$, so $c_+(\face)\geq c_v\pair(v,\face)\great c_2$. Moreover, $A\leq 3$, hence $c(\face)\great c_2+3c_1\great 0.022$.
\item[Case $B=0$ and $A=2$]
Let $\fverts(\face)=\langle v_1,v_2,v_3,v_4,v_5\rangle$ and suppose that $\mcl{A}=\{v_2,v_3\}$. Then $4\in\vtype(v_1)$ and $4\in\vtype(v_4)$. Moreover, according to rule \cref{rule:3445}, neither $v_1$ nor $v_4$ can be a special vertex of type $(3,4,4,5)$, since they are both consecutive in $\face$ to a vertex in $\mcl{A}$. Therefore, according to the table, both $v_1$ and $v_4$ contribute at least $c_4$ to $c_+(\face)$. Hence $c(\face)\great  2c_4+2c_1 \great 0.019$.
\item[Case $A=4$] 
Let $v$ be the only element of $\fverts(\face)$ which is not in $\mcl{A}$.
By rule \cref{rule:3445}, $v$ is not a special vertex of type $(3,4,4,5)$, since it is consecutive in $\face$ to an element of $\mcl{A}$.
Moreover, the only possibilities for $\vtype(v)$ are $(4,4,5)$ and $(3,4,4,5)$. 
In either case, $c_v\pair(v,\face)\great c_3$, and so $c(\face)\great 4c_1+c_3= 0.002$.
\end{description}

In any case we obtain $c(\face)\great 0.002$.


\section{Analysis of faces with 7 edges}\label{sec:7}

Let $\face\in \FFaces$ with $\abs{\face}=7$ and $\pair(\face)\neq 0$.
With our construction of the pairing, we have $\pair(v,\face)\neq 0$ if and only if $v\in\fverts(\face)$,  and either $4\in \vtype(v)$, $8\in \vtype(v)$, or $9\in \vtype(v)$. We list all the possibilities in \cref{table:7}, and we define the constants $c_1:=-0.0057$, 
 $c_2=0.038$, 
$c_3=0.012$, 
$c_4=0.133$. 

\begin{table}[H]\centering 
\begin{tabular}{ccccl}
\toprule
$\vtype(v)$
	\ \ \ &\ \ \ 
\text{where}
	\ &\
$\pair(v,\face)$
	\ \ \ & \ 
 $c_v\pair(v,\face)$& \\

\midrule

$(3,4,7)$
	\ \ \ &\ \ \ 

	 \ & \ 
$\frac{1}{2}$
	\ \ \ &\ \ \ 
$\great 0.108$& \\

\midrule

$(3,7,a)$
	\ \ \ &\ \ \ 
$a=8,9$
	 \ & \ 
$\frac{1}{2}$
	\ \ \ &\ \ \ 
$\great 0.038$ & $=c_2$\\

\midrule

$(4,4,7)$
	\ \ \ &\ \ \ 

	 \ & \ 
$1$
	\ \ \ &\ \ \ 
$\great 0.133$ &$=c_4$\\

\midrule

$(4,5,7)$
	\ \ \ &\ \ \ 

	 \ & \ 
$\frac{1}{2}$
	\ \ \ &\ 
$\great 0.041$ &\\

\midrule

$(4,6,7)$
	\ \ \ &\ 

	 \ & \ 
$1$
	\ \ \ &\ \ \
$\great 0.049$ &\\

\midrule

$(4,7,7)$
	\ \ \ &\ \ \ 
	
	 \ & \ 
$\frac{1}{2}$ or 1
	\ \ \ &\ \ \ 
$\great 0.013$ &\\

\midrule

$(4,7,a)$
	\ \ \ &\ \ \ 
	$a=8,9$
	 \ & \ 
$1$
	\ \ \ &\ \ \ 
$\great -0.0057$ &$=c_1$\\

\midrule

$(3,3,4,7)$
	\ \ \ &\ \ \ 
	
	 \ & \ 
$\frac{1}{4}$
	\ \ \ &\ \ \ 
$\great 0.012$ &$=c_3$\\

\bottomrule
\end{tabular}
\caption{Vertices contributing nontrivially when $\protect\abs\face=7.$}
\label{table:7}
\end{table}

Let $A=\#\mcl{A}$, $B=\#\mcl{B}$ and $L=\#\mcl{L}$, where
\[
\begin{aligned}
\mcl{A}&:=\{v\in\fverts(\face):\ \vtype(v)=(4,7,a), \ 8\leq a\leq 9\},\\  
\mcl{B}&:=\{v\in\fverts(\face):\ \vtype(v)=(3,7,a), \ 8\leq a\leq 9\},\\
\mcl{L}&:=\{e\in\fedges(\face):\ \etype(e)=(7,a),\ 8\leq a \leq 9\}.
\end{aligned}
\]
We notice immediately that $c_-(\face)\great A c_1$; moreover, $2L=A+B\leq 6$, since it is both even and less than 7.
We now prove that $c(\face)\great 0$ considering the following 4 cases.

\begin{figure}[ht]\centering 
\input{7_A2.tikz}
\quad
\input{7_A4.tikz}
\quad
\input{7_A6.tikz}
\caption{Illustrations for $\protect\abs\face=7.$}
\label{fig:7}
\end{figure}
\begin{description}
\item[Case $A=0$]
Here $c(\face)=c_+(\face)$ and so, by inspection of the table, 
$c(\face)\great c_3=0.012$.
\item[Case $B\geq 1$]
In this case  $\exists v\in\mcl{B}$ and $c_+(\face)\geq c_v\pair(v,\face)\great c_2$. Moreover $A\leq 5$, hence $c(\face)\great c_2+5c_1= 0.0095$.
\item[Case $B=0$ and $A\in\{2,4\}$]
%
Arguing as in \emph{Case $B=0$ and $A=2$} of \cref{sec:5} we see that there are two elements $v,w\in\fverts(\face)$ with $4\in\vtype(v)$, $4\in\vtype(w)$ and $v,v\not\in\mcl{A}$. Then, according to the table, both $v$ and $w$ contribute at least $c_3$ to $c_+(\face)$. Hence $c(\face)\great  2c_3+2c_1 \great 0.012$.
%
%
\item[Case $A=6$]
Let $v$ be the only element of $\fverts(\face)$ which is not in $\mcl{A}$.
Then $\vtype(v)$ can only be $(4,4,7)$, and so $c(\face)\great 6c_1+c_4\great 0.098$.
\end{description}
In any case $c(\face)\great 0.0095$.


\section{Analysis of faces with 11 edges}\label{sec:11}

The faces with 11 edges in a PCC graph exhibit rich combinatorial complexity around them, and many of the vertices in their boundaries may have very small curvature. 
Therefore we are required to perform a more careful analysis than in other sections. In particular, we exploit much more heavily the machinery of special vertices. See also \cref{def:blue} for the notion of blue edges, $\alpha$-vertices and $\beta$-vertices.

Let $\face\in \FFaces$ with $\abs{\face}=11$ and $\pair(\face)\neq 0$.
With our construction of the pairing, we have $\pair(v,\face)\neq 0$ if and only if $v\in\fverts(\face)$ or $v$ is a vertex special to $\face$. For $v\in\fverts(\face)$, we list all the possible face vectors, pairings with $\face$, and contributions to $c(\face)$  in \cref{table:11:1}. 
A vertex can be special to $\face$ as a consequence of all rules except \cref{rule:3ab}.
We list all the possibilities for $v\in\Verts$ special to $\face$  in \cref{table:11:2}.
We also define the constants 
$c_0=0.023$,
$c_1^{\alpha}=-0.00121$, 
$c_1^{\beta}=-0.00361$, 
$c_2=-0.002$, 
$c_3=0.00279$, 
$c_4=0.027$, 
$c_5=0.027$, 
$c_6=0.027$.

\begin{table}[H]\centering 
\begin{tabular}{ccccl}
\toprule
$\vtype(v)$
	\ \ \ &\ \ \ 
\text{where}
	\ &\
$\pair_1(v,\face)$
	\ \ \ & \quad
 $c_v\pair_1(v,\face)$ &  \\

\midrule

$(3,a,11)$
	\ \ \ &\ \ \ 
$a\in\{3,4,5\}$
	 \ & \ 
$\frac{1}{2}$
	\ \ \ &\ \ \ 
$\great 0.057$& \\

\midrule

$(3,6,11)$
	\ \ \ &\ \ \ 

	 \ & \ 
$1$
	\ \ \ &\ \ \ 
$\great 0.081$& \\

\midrule

$(3,a,11)$
	\ \ \ &\ \ \ 
$a\in\{7,8,9,10\}$
	 \ & \ 
$1$
	\ \ \ &\ \ \ 
$\great 0.014$& $=c_6$\\

\midrule

$(3,11,11)$
	\ \ \ &\ \ \ 

	 \ & \ 
$\frac{1}{2}$ or 1
	\ \ \ &\ \ \ 
$\great 0.00279$&$=c_3$ \\

\midrule

$(3,11,12)$
	\ \ \ &\ \ \ 

	 \ & \ 
$1$
	\ \ \ &\ \ \ 
$\great -0.0020$&$=c_2$ \\

\midrule

$(3,11,13)$
	\ \ \ &\ \ \ 
$v$ is an $\alpha$-vertex
	 \ & \ 
$\frac{1}{7}$
	\ \ \ &\ \ \ 
$\great -0.00121$& $=c_1^{\alpha}$\\

\midrule

$(3,11,13)$
	\ \ \ &\ \ \ 
$v$ is a $\beta$-vertex
	 \ & \ 
$\frac{3}{7}$
	\ \ \ &\ \ \ 
$\great -0.00361$&$=c_1^{\beta}$ \\

\midrule

$(4,4,11)$
	\ \ \ &\ \ \ 
	
	 \ & \ 
$1$
	\ \ \ &\ \ \ 
$\great 0.081$& \\

\midrule

$(4,5,11)$
	\ \ \ &\ \ \ 

	 \ & \ 
$\frac{1}{2}$
	\ \ \ &\ \ \ 
$\great 0.015$&$=c_5$ \\

\midrule

$(4,6,11)$
	\ \ \ &\ \ \ 

	 \ & \ 
$1$
	\ \ \ &\ \ \ 
$\great -0.0020$&$=c_2$ \\

\midrule

$(3,3,3,11)$
	\ \ \ &\ \ \ 

	 \ & \ 
$\frac{1}{3}$ 
	\ \ \ &\ \ \ 
$\great 0.027$&$=c_4$ \\

\midrule

$(3,3,4,11)$
	\ \ \ &\ \ \ 

	 \ & \ 
$1$
	\ \ \ &\ \ \ 
$\great -0.0020$&$=c_2$ \\

\bottomrule
\end{tabular}
\caption{Vertices in $v\in\fverts(\face)$ contributing nontrivially, when $\protect\abs\face=11$.}
\label{table:11:1}
\end{table}


\begin{table}[H]\centering 
\begin{tabular}{ccccl}
\toprule
$\vtype(v)$
	 &
\text{Special rule}
	 &
$\pair_2(v,\face)$
	 & 
 $c_v\pair_2(v,\face)$ &  \\

\midrule

$\{\text{only 3,4}\}$
	 &
\cref{rule:TS}
	 &
$\frac{1}{3}$
	 &
$\great 0.024$& \\

\midrule

$(3,3,11)$
	 & 
\cref{rule:33a}
	 &  
$\frac{1}{2}$
	 &
$\great 0.124$& \\

\midrule

$(3,3,12)$
	&
\cref{rule:33a}
          & 
$1$
	&
$\great 0.240$& \\

\midrule

$(3,4,a)$
	 & 
\cref{rule:34a}, $\begin{cases}5\leq a\leq 41 \\ a\not\in\{8,9,10,12\}\end{cases}\!\!$
	 & 
$\frac{1}{2}$
	& 
$\great 0.049$& \\

\midrule

%

$(3,4,a)$
 &
\cref{rule:34a}, $a\in\{8,9,10,12\}$
	 &
$1$
	 & 
$\great 0.157$& \\

\midrule

$(4,5,6)$
	 &
\cref{rule:456}
	  & 
$\frac 1 2$
	 &
$\great 0.053$& \\

\midrule

$(3,3,3,a)$
	 & 
\cref{rule:333a}, $5\leq a\leq 10$
	 &  
$1$
	 &
$\great 0.090$& \\

\midrule

$(3,3,3,11)$
	& 
\cref{rule:333a}
	 &  
$\frac{1}{3}$ 
	&
$\great 0.027$& \\

\midrule

$(3,3,3,12)$
	 & 
\cref{rule:333a}
	 & 
$\frac{1}{2}$
	 &
$\great 0.036$& \\

\midrule

%

$(3,3,4,a)$
	& 
\cref{rule:334a}, $a\in\{5,6\}$
	  &  
$\frac 1 2$
	 &
$\great 0.036$& \\

\midrule

$(3,3,4,7)$
	  &
\cref{rule:334a}
	  &
$\frac{3}{4}$
	 &
$\great 0.037$& \\

\midrule

$(3,4,4,5)$
	 & 
\cref{rule:3445}
	 & 
$1$
	 &
$\great 0.023$& $\!=c_0$\\

\midrule

$(3,3,3,3,5)$
	&
\cref{rule:33335}
	 & 
$1$
	 &
$\great 0.023$& $\!=c_0$\\

\bottomrule
\end{tabular}
\caption{Special vertices contributing nontrivially when $\protect\abs\face=11$.}
\label{table:11:2}
\end{table}

Let $A=\#\mcl{A}$, $B=\#\mcl{B}$, $C=\#\mcl{C}$ and $D=\#\mcl{D}$, where
\[
\begin{aligned}
	\mcl{A}&:=\{v\in\fverts(\face):\ v \text{ is an $\alpha$-vertex}\},\\
	\mcl{B}&:=\{v\in\fverts(\face):\ v \text{ is a $\beta$-vertex}\},\\
	\mcl{C}&:=\{v\in\fverts(\face):\ \vtype(v)=(3,11,11)\},\\
	\mcl{D}&:=\{v\in\fverts(\face):\ \vtype(v)\in \{(3,11,12), (4,6,11), (3,3,4,11)\}\}.
\end{aligned}
\]

Let also
\[
\begin{aligned}
	\mcl{W}_1&:=\{v\in\fverts(\face):\ \vtype(v)\in \{(3,a,11):\ a\leq 6\}\cup\{(4,4,11), (3,3,3,11)\}\},\\
	\mcl{W}_2&:=\{v\in\fverts(\face):\ \vtype(v)=(4,5,11)\},\\
	\mcl{W}_3&:=\{v\in\fverts(\face):\ \vtype(v)=(3,a,11),\ 7\leq a\leq 10\},\\
	\mcl{D}_1&:=\{v\in\fverts(\face):\ \vtype(v)=\langle 3,4,3,11\rangle\},\\
	\mcl{D}_2&:=\{v\in\fverts(\face):\ \vtype(v)=\langle 3,3,4,11\rangle \text{ (or } \vtype(v)=\langle 4,3,3,11\rangle \text{)}\},\\
	\mcl{D}_6&:=\{v\in\fverts(\face):\ \vtype(v)=(4,6,11)\},\\
	\mcl{D}_{12}&:=\{v\in\fverts(\face):\ \vtype(v)=(3,11,12)\}.\\
\end{aligned}
\]

Finally, let $\mcl{L}:=\{e\in\fedges(\face):\ e \text{ is a blue edge}\}$ and, for $k\in\{3,11,13\}$, let
\[ \mcl M_k:=\{e\in\fedges(\face):\ \etype(e)=\{11,k\}\}\]

We start by giving a quick estimate of $c_-(\face)$. Clearly, $c_-(\face)\great B c_1^{\beta}+(A+D) \min\{c_1^{\alpha},c_2\}\geq B c_1^{\beta}+(11-B-C) c_2 $. We observe, directly from \cref{def:blue}, that the elements of $\mcl{B}$ are exactly the endpoints of the elements of $\mcl{L}$. This implies that $B=2(\#\mcl L)$, so $B$ is even.
Similarly, $A+B=2(\#\mcl M_{13})$ and $C=2(\#\mcl M_{11})$, so also $A$ and $C$ are even.
Moreover, if $v_1,v_2,v_3,v_4,v_5,v_6\in\fverts(\face)$ are consecutive vertices of $\face$ and $v_3,v_4\in\mcl{B}$, then $v_1,v_2,v_5,v_6\in\mcl{C}$. This implies that $B\leq C$, and so $B\leq \frac{11}{2}\less 6$. Moreover, if $B=2$ or $B=4$, we see that $C\geq B+2$. Then we have the following 3 cases.

\begin{description}
\item[Case $B=0$]
Then $c_-(\face)\great 11c_2 = -0.022$.
\item[Case $B=2$]
Then $C\geq 4$, so $c_-(\face)\great 2c_1^{\beta} +(11-6)c_2\great -0.018$.
\item[Case $B=4$]
Then $C\geq 6$, so $c_-(\face)\great 4c_1^{\beta} +(11-10)c_2= -0.01644$. 
\end{description}

In any case, we obtain $c_-(\face)\great -0.022$. In addition to this, we deduce the following lemma, that we will repeatedly use in the subsequent analysis.

\begin{lemma}\label{lemma:11:0}
Suppose that $\exists w\in\Verts$ which is special to $\face$. Then $c(\face)\great 0.001$.
\end{lemma} 
\begin{proof}
According to the second table, in this case we get $c_+(\face)\geq c_w\pair(w,\face)\great c_0=0.023$.
\end{proof}

We now rule out the following 3 cases, which don't require the use of special vertices.

%
%
%

\begin{figure}[ht]\centering 
\input{11_W2.tikz}
\quad \quad
\input{11_W3.tikz}
\caption{Illustration of first cases for $\protect\abs\face=11$.}
\label{fig:11:1}
\end{figure}

\begin{description}
\item[Case $\exists\, v\in \mcl{W}_1$]
Then $c_+(\face)\geq c_v\pair(v,\face)\great c_4=0.027$, hence $c(\face)\great 0.005$.
\item[Case $\exists\, v\in \mcl{W}_2$]
Let $w\in\fverts(\face)$ with $vw\in\fedges(\face)$ and $\etype(vw)=(5,11)$. 
Since the possibility $\vtype(w)=(3,5,11)$ has already been ruled out, we must have $w\in\mcl{W}_2$ as well. 
Then $c_+(\face)\geq c_v\pair(v,\face)+c_w\pair(w,\face)\great 2c_5=0.030$, hence $c(\face)\great 0.008$.
\item[Case $\exists\, v\in \mcl{W}_3$]
Let $w\in\fverts(\face)$ with $vw\in\fedges(\face)$ and  $\etype(vw)=(a,11)$, for some $a\in\{7,8,9,10\}$. 
Then we must have $w\in\mcl{W}_3$ as well. 
Therefore $c_+(\face)\geq c_v\pair(v,\face)+c_w\pair(w,\face)\great 2c_6=0.028$, hence $c(\face)\great 0.006$.
\end{description}

If one of the above cases applies, we obtain $c(\face)\great 0.005$. Otherwise, we must have that every element of $\fverts(\face)$ belongs to either $\mcl{A}$, $\mcl{B}$, $\mcl{C}$ or $\mcl{D}$.
Let now $e\in\fedges(\face)$ and $\everts(e)=\{v_1,v_2\}$. 
We rule out the following 5 cases.
\begin{figure}[ht]\centering 
\input{11_D12D1.tikz}
\quad
\input{11_D12D2.tikz}
\quad
\input{11_D1D1.tikz}

\input{11_AD1.tikz}
\quad\quad 
\input{11_D2D2.tikz}
\caption{Illustration of more cases for $\protect\abs\face=11$.}
\label{fig:11:2}
\end{figure}
\begin{description}
\item[Case $v_1\in\mcl D_{12}$ and $v_2\in\mcl{D}_1$]
Then $\efaces(e)=\{\tau,\face\}$ with $\abs{\tau}=3$. Let $w\in\Verts$ such that $\everts(\tau)=\{v_1,v_2,w\}$. We notice that $(3,4,12)\subseteq \vtype(w)$, so $\vtype(w)=(3,4,12)$, and by \cref{rule:34a} we have that $w$ is special to $\face$. 
\item[Case $v_1\in\mcl D_{12}$ and $v_2\in\mcl{D}_2$]
Then $\efaces(e)=\{\tau,\face\}$ with $\abs{\tau}=3$. Let $w\in\Verts$ such that $\everts(\tau)=\{v_1,v_2,w\}$. We notice that $(3,3,12)\subseteq \vtype(w)$, so $\vtype(w)=(3,3,12)$ or $\vtype(w)=(3,3,3,12)$ and either by \cref{rule:33a} or \cref{rule:333a} we have that $w$ is special to $\face$. 
\item[Case $v_1,v_2\in\mcl{D}_1$]
Then $\efaces(e)=\{\tau,\face\}$ with $\abs{\tau}=3$. Let $w\in\Verts$ such that $\everts(\tau)=\{v_1,v_2,w\}$. We notice that $(3,4,4)\subseteq \vtype(w)$, so either $w$ is a TS-vertex or $\vtype(w)=\langle 4,3,4,5\rangle$. In the first case by \cref{rule:TS} we have that $w$ is special to $\face$. 
Otherwise, let $\tau_1,\tau_2,s_1,s_2\in\Faces$ with $\abs{\tau_1}=\abs{\tau_2}=3$ and $\abs{s_1}=\abs{s_2}=4$, such that $\vfaces(v_1)=\langle \tau, s_1, \tau_1, \face \rangle$ and $\vfaces(v_2)=\langle \tau_2, s_2, \tau, \face \rangle$. Moreover, let $w_1,w_2,w_3,w_4,w_5,w_6\in\Verts$ be such that $\fverts(\tau_1)=\langle w_1,v_1,w_2 \rangle$, $\fverts(s_1)=\langle w_2,v_1,w,w_3\rangle$, $\fverts(\tau_2)=\langle v_2,w_6,w_5 \rangle$ and $\fverts(s_2)=\langle w,v_2,w_5,w_4\rangle$. We claim that $\fverts(w_3)\neq(4,5,a)$ for  $14\leq a \leq 19$. Suppose the contrary: then we deduce that $\fverts(w_2)=(3,4,a)$ and so also $\fverts(w_1)=(3,11,a)$, which is not admissible. Analogously, we find that $\fverts(w_4)\neq(4,5,a)$ for  $14\leq a \leq 19$. Therefore, by \cref{rule:3445}, we have that $w$ is special to $\face$. 
\item[Case $v_1\in\mcl A$ and $v_2\in\mcl{D}_1$]
Then $\efaces(e)=\{\tau,\face\}$ with $\abs{\tau}=3$. Let $w\in\Verts$ such that $\everts(\tau)=\{v_1,v_2,w\}$. We notice that $(3,4,13)\subseteq \vtype(w)$, so $\vtype(w)=(3,4,13)$, and by \cref{rule:34a} we have that $w$ is special to $\face$. 
\item[Case $v_1,v_2\in\mcl{D}_2$ and $\etype(e)=(3,11)$]
Let $\efaces(e)=\{\tau,\face\}$ with $\abs{\tau}=3$ and let $w\in\Verts$ such that $\everts(\tau)=\{v_1,v_2,w\}$. Let also $\tau_1,\tau_2,s_1,s_2\in\Faces$ with $\abs{\tau_1}=\abs{\tau_2}=3$ and $\abs{s_1}=\abs{s_2}=4$, such that $\vfaces(v_1)=\langle \tau, \tau_1, s_1, \face \rangle$ and $\vfaces(v_2)=\langle s_2, \tau_2, \tau,\face \rangle$. Finally, let $w_1,w_2\in\Verts$ such that $\fverts(\tau_1)=\langle v_1,w,w_1\rangle$ and $\fverts(\tau_2)=\langle w,v_2,w_2\rangle$. We notice that $(3,3,3)\subseteq \vtype(w)$, so we consider the following 4 cases.
	\begin{description}
	\item[Case $w$ is a TS vertex]
	Then by \cref{rule:TS} $w$ is special to $\face$. 
	\item[Case $\vtype(w)=(3,3,3,a)$ with $5\leq a\leq 12$]
	Then by \cref{rule:333a} we have that $w$ is special to $\face$. 
	\item[Case $\vtype(w)=(3,3,3,a)$ with $a\geq 13$]
	Then we have $\vtype(w_1)=(3,4,a)$. Moreover, by \cref{thm:big} below, we cannot have $a\geq 42$. Hence, by \cref{rule:34a}, $w_1$ is special to $\face$. 
	\item[Case $\vtype(w)=(3,3,3,3,5)$]
	Then by \cref{rule:33335} $w$ is special to $\face$, since either $(3,4,5)\subseteq \vtype(w_1)$ or $(3,4,5)\subseteq \vtype(w_2)$. 
	\end{description}
\end{description}
In any case we obtain $c(\face)\great 0.001$ by \cref{lemma:11:0}.
We now rule out 6 more cases.
Let $v_1,v_2,v_3,v_4\in\fverts(\face)$ be consecutive vertices on $\face$. 

\begin{figure}[ht]\centering 
\input{11_D2D1D2.tikz}
\quad
\input{11_D1D2D6.tikz}

\input{11_D1D2D2A.tikz}
\quad
\input{11_CD2D6.tikz}

\input{11_CD2D2A.tikz}
\quad
\input{11_CD2D2D1.tikz}
\caption{Illustration of even more cases for $\protect\abs\face=11$.}
\label{fig:11:3}
\end{figure}

\begin{description}
\item[Case $v_1,v_3\in\mcl{D}_2$ and $v_2\in\mcl{D}_1$]
Let $\tau_1,\tau_2,s_1\in\Faces$ with $\abs{\tau_1}=\abs{\tau_2}=3$ and $\abs{s_1}=4$, such that $\vfaces(v_2)=\langle \tau_2, s_1,\tau_1, \face \rangle$ and let $w_1,w_2,w_3\in\Verts$ such that $\fverts(\tau_1)=\langle v_1,v_2,w_1\rangle$, $\fverts(\tau_2)=\langle v_2,v_3,w_3\rangle$ and $\fverts(s_1)=\langle w_1,v_2,w_3,w_2\rangle$. 
We see that it is impossible to have both $\vtype(w_1)=(3,3,4,a)$ and $\vtype(w_3)=(3,3,4,b)$ with $8\leq a,b\leq 11$. 
Indeed otherwise $\{4,a,b\}\subseteq\vtype(w_2)$, which is not admissible. 
Hence, for some $i\in\{1,3\}$ we have that $w_i$ is a TS-vertex, or that $\vtype(w_i)=(3,3,4,a)$ with $5\leq a\leq 7$.
Then, either by \cref{rule:TS} or \cref{rule:334a} $w_i$ is special to $\face$. 
\item[Case $v_1\in\mcl{D}_1$, $v_2\in\mcl{D}_2$ and $v_3\in\mcl D_6$]
Let $\tau_1,\tau_2,s_1\in\Faces$ with $\abs{\tau_1}=\abs{\tau_2}=3$ and $\abs{s_1}=4$, such that $\vfaces(v_2)=\langle s_1,\tau_2, \tau_1,\face \rangle$ and let $w_1,w_2,w_3\in\Verts$ such that $\fverts(\tau_1)=\langle v_1,v_2,w_1\rangle$, $\fverts(\tau_2)=\langle w_1,v_2,w_2\rangle$ and $\fverts(s_1)=\langle v_2,v_3,w_3,w_2\rangle$.
We rule out the following 3 cases.
		\begin{description}
		\item[Case $\etype(w_1w_2)=(3,a)$ with $a=3,4$]
		Then $w_1$ is necessarily a TS-vertex, so by \cref{rule:TS} it is special to $\face$. 
		\item[Case $\etype(w_1w_2)=(3,5)$]
		Then $\vtype(w_1)=(3,3,4,5)$, so by \cref{rule:334a} $w_1$ is special to $\face$. 
		\item[Case $\vtype(w_2)=(3,4,a)$ with $5\leq a$]
		By \cref{thm:big} below we must have $a\leq 41$. Then by \cref{rule:34a} $w_2$ is special to $\face$. 
		\end{description}
In each of these cases we obtain $c(\face)\great 0.001$. Then we may suppose that $\etype(w_1w_2)=(3,a)$ with $a\geq 6$, and since $(3,4,a)\varsubsetneq \vtype(w_2)$ we deduce $\vtype(w_2)=\langle 3,4,3,a\rangle$ and $\etype(w_2w_3)=(3,4)$. Therefore $\langle 3,4,6\rangle \subseteq \vtype(w_3)$. We now see that either \cref{rule:34a} or \cref{rule:334a} applies, so $w_3$ is special to $\face$.
\item[Case $v_1\in\mcl{D}_1$, $v_2,v_3\in\mcl{D}_2$ and $v_4\in\mcl A$]
This case is very similar to the previous one. We define $\tau_1,\tau_2,s_1\in\Faces$ and $w_1,w_2,w_3\in\Verts$ exactly as before, and we rule out the same 3 cases. After that, we see that $\langle 3,4,3\rangle\subseteq \vtype(w_3)$. 
Let also $\efaces(v_3v_4) =(\tau_4,\face)$, so that $\abs{\tau_4}=3$ and $\fverts(\tau_4)=\{v_3,v_4,w_4\}$ for some $w_4\in\Verts$. We cannot have $\vtype(w_3)=(3,3,4,a)$ for $5\leq a\leq 11$, because otherwise we see that $(3,3,a,13)\subseteq \vtype(w_4)$, which is not admissible. We deduce that $w_3$ must be a TS-vertex and by \cref{rule:TS} it is special to $\face$. 
\item[Case $v_1\in\mcl C$, $v_2\in\mcl D_2$ and $v_3\in\mcl{D}_6$]
Let $\tau_1,\tau_2,s_1\in\Faces$ and $w_1,w_2,w_3\in\Verts$ exactly as before.
We have $\langle 11,3,3\rangle\subseteq \vtype(w_1)$.
If $\vtype(w_1)\in\{(3,3,11),(3,3,3,11)\}$ then $w_1$ is special to $\face$ by \cref{rule:33a} or \cref{rule:333a}.
If $\vtype(w_1)=\langle 11,3,3,4\rangle$, then $\langle 4,3,4\rangle\subseteq\vtype(w_2)$.
If $w_2$ is a TS-vertex, it is special to $\face$ by \cref{rule:TS}.
Otherwise, we must have $\vtype(w_2)=\langle 4, 3, 4, 5\rangle$, and so $\vtype(w_3)=(4,5,6)$.
Then $w_3$ is special to $\face$ by \cref{rule:456}.
\item[Case $v_1\in\mcl C$, $v_2,v_3\in\mcl D_2$ and $v_4\in\mcl A$]
Let $\tau_1,\tau_2,s_1\in\Faces$ and $w_1,w_2,w_3\in\Verts$ exactly as before.
Arguing as in the prevous case, we have that either $w_i$ is special to $\face$ for some $i\in\{1,2\}$, or $\langle 5, 4 , 3\rangle \subseteq \vtype(w_3)$.
Let $\efaces(v_3v_4) =(\tau_4,\face)$, so that $\abs{\tau_4}=3$ and $\fverts(\tau_4)=\{v_3,v_4,w_4\}$ for some $w_4\in\Verts$.
We cannot have $\vtype(w_3)=\langle 5, 4 , 3,4\rangle$ because otherwise $(3,3,4,13)\subseteq \vtype(w_4)$, which is not admissible.
Therefore, if $w_1$ and $w_2$ are not special to $\face$, we must have $\vtype(w_3)=\langle 5, 4 , 3\rangle$ or $\langle 5, 4 , 3,3\rangle$.
In either case, $w_3$ is special to $\face$, by \cref{rule:34a} or \cref{rule:334a}.
\item[Case $v_1\in\mcl C$, $v_2,v_3\in\mcl D_2$ and $v_4\in\mcl D_1$]
Let $\tau_1,\tau_2,s_1\in\Faces$ and $w_1,w_2,w_3\in\Verts$ exactly as before.
Arguing as in the prevous case, we have that either $w_i$ is special to $\face$ for some $i\in\{1,2\}$, or $\langle 5, 4 , 3\rangle \subseteq \vtype(w_3)$.
Moreover, if $\vtype(w_3)\in\{\langle 5, 4 , 3\rangle,\langle 5, 4 , 3,3\rangle\}$, then $w_3$ is special to $\face$, by \cref{rule:34a} or \cref{rule:334a}.
As before, let $\efaces(v_3v_4) =(\tau_4,\face)$, so that $\abs{\tau_4}=3$ and $\fverts(\tau_4)=\{v_3,v_4,w_4\}$ for some $w_4\in\Verts$.
Then, if $w_1$, $w_2$ and $w_3$ are not special to $\face$, we must have $\vtype(w_3)=\langle 5, 4 , 3,4\rangle$, and so $\vtype(w_4)=(3,3,4,4)$.
In this last case, $w_4$ is a TS-vertex, and so is special to $\face$, by \cref{rule:TS}.
\end{description}

We now give an approximate description of the remaining cases, by means of 2 lemmas.

\begin{figure}[ht]\centering 
\input{11_lemma1.tikz}
\caption{Illustration for Lemma 8.2.}
\label{fig:11:lemma:1}
\end{figure}

\begin{lemma}\label{lemma:11:1}
Suppose the previous 11 cases don't apply. Then $\exists\, v_1,v_2,v_3\in\fverts(\face)$ consecutive on $\face$ with $v_1,v_3\in \mcl{C}$ and $v_2\in\mcl{D}_1$. In particular $C\geq 4$ and $D\geq 1$.
\end{lemma}
\begin{proof}
Let 
\[
\mcl{M}=\{e\in\fedges(\face):\ \etype(e)\in\{(4,11),(11,11),(11,12),(11,13)\}\}.
\]
Then $\#\mcl{A}+\#\mcl{B}+\#\mcl{C}+\#\mcl{D}_2+\#\mcl{D}_6+\#\mcl{D}_{12}=2\cdot\#\mcl{M}$. 
Since $11=A+B+C+D$, this implies that there is an odd number of elements of $\mcl{D}_1$. 
Let $ v_1,v_2,v_3,v_4,v_5,v_6\in\fverts(\face)$ be consecutive vertices on $\face$ with $v_2\in\mcl{D}_1$. 
Then the above case analysis shows that $v_1,v_3\in\mcl{C}\cup\mcl{D}_2$ and that we cannot have both $v_1,v_3\in\mcl{D}_2$.
Moreover, if $v_1\in\mcl{C} $ and $v_3\in\mcl{D}_2$, we cannot have $v_4\in\D_6$, so $v_4\in\D_2$. 
We also see that in this case we must have $v_5\in\D_1$
and so $v_6\in\mcl{C}$. 
We deduce that there must be an even number of elements of $\D_1$ which are consecutive on $\face$ to an element of $\D_2$. 
Since $\#\D_1$ is odd, we conclude that there exists an element $v_2\in\D_1$ consecutive on $\face$ only to elements of $\mcl{C}$.
\end{proof}

In addition to the conclusion of \cref{lemma:11:1}, we can make the following easy observation, which we will repeatedly use in the proof of \cref{lemma:11:2}. 

\begin{figure}[ht]\centering 
\input{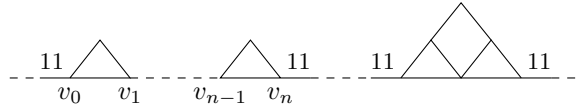}
\caption{Illustration for Remark 8.3.}
\label{fig:11:rmk}
\end{figure}

\begin{remark}\label{remark:11}
Suppose that we can find, for $3\leq n\leq 6$,  vertices $v_0,v_1,\ldots,v_n\in\fverts(\face)$  consecutive on $\face$, with $v_0,v_n\in\mcl{C}$ and $v_i\not\in\mcl{C}$ for $i=1,\ld,n-1$.
Then, under the hypothesis of \cref{lemma:11:1}, we can deduce that actually $C\geq 6$.
\end{remark}

\begin{lemma}\label{lemma:11:2}
Suppose the previous 11 cases don't apply. Then $C\geq 6$ or $A\geq 4$.
\end{lemma}
\begin{proof}
We consider the following 5 cases.
\begin{figure}[ht]\centering 
\input{11_B.tikz}
\quad
\input{11_D2_A.tikz}

\input{11_D12.tikz}
\quad
\input{11_D2_C.tikz}

\input{11_D6.tikz}
\quad
\input{11_D2_D.tikz}
\caption{Illustration for Lemma 8.4.}
\label{fig:11:lemma:2}
\end{figure}
\begin{description}
\item[Case $\exists w\in \mcl{B}$]
If there is a $\beta$-vertex in $\fverts(\face)$, then there is also a blue edge in $\fedges(\face)$, and we have already remarked at the beginning of this section that this implies the existence of $v_0,v_1,v_2,v_3\in\fverts(\face)$ consecutive on $\face$, with $v_0,v_3\in\mcl{C}$ and $v_1,v_2\in\mcl{B}$. Therefore $C\geq 6$ by \cref{remark:11}.
\item[Case $\exists w\in \D_{12}$]
Let $v_0,v_1,v_2,v_3\in\fverts(\face)$ be consecutive on $\face$, with $v_1\in\D_{12}$ and $\etype(v_1v_2)=(11,12)$. 
Then also $v_2\in\D_{12}$. 
By the above case analysis we have that $v_0,v_3\not\in\D_1\cup\D_2$. 
Let $\tau\in\Faces$ with $\abs{\tau}=3$ and $w_1\in\Verts$ such that $\efaces(v_0v_1)=\{\tau,\face\}$ and  $\fverts(\tau)=\{v_0,v_1,w_1\}$.
We cannot have $\vtype(v_0)\in\mcl A\cup \mcl D_{12}$ because otherwise $(3,12,a)\subseteq \vtype(w_1)$ with $a\in\{12,13\}$, which is not admissible.
Analogously $\vtype(v_3)\not\in\mcl A\cup \mcl D_{12}$. 
Since moreover $v_0,v_3\not\in\D_6$, we necessarily have $v_0,v_3\in\mcl{C}$, which implies $C\geq 6$ by \cref{remark:11}.
\item[Case $\exists w\in \D_6$]
Fix an orientation $\pi$ of the sphere. By the above case analysis, if $w_1,w_2,w_3\in\fverts(\face)$ are consecutive on $\face$ with $\etype(w_2w_3)=4$ and $w_3\in\D_6$, then $w_2\in\D_6$, or $w_2\in\mcl \D_2$ and $w_1\in\mcl{A}$. Since $\D_6$ is not all of $\fverts(\face)$, visiting counterclockwise the perimeter of $\face$ we must find $v_0,v_1,v_2,v_3\in\fverts(\face)$ clockwise consecutive on $\face$ with $v_0,v_1\in\mcl{A}$, $v_2\in\D_2$ and $v_3\in\D_6$. Analogously, visiting clockwise the perimeter of $\face$, starting from $v_3$, we must find $v_4,v_5,v_6,v_7\in\fverts(\face)$ clockwise consecutive on $\face$ with $v_6,v_7\in\mcl{A}$, $v_5\in\D_2$ and $v\in\D_6$ for all the edges between $v_3$ and $v_4$ (clockwise). By \cref{lemma:11:1} there must be other vertices in the clockwise segment of $\fverts(\face)$ between $v_7$ and $v_0$. This shows that $v_0,v_1,v_6,v_7$ are four distinct elements of $\mcl{A}$ (they are also distinct as elements of $\Verts$, although we don't need it).
\item[Case $\exists w\in \D_2$]
Let $v_0,v_1,v_2,v_3,v_4,v_5\in\fverts(\face)$ be consecutive on $\face$, with $v_2\in\D_2$ and $\etype(v_2v_3)=(4,11)$. If $v_3\in\D_6$ we proceed as in the previous case, so we can assume $v_3\in\D_2$. By the above case analysis we have that $v_1,v_4\in\mcl{A}$, or $v_1,v_4\in\mcl C$, or $v_1,v_4\in\mcl{D}_1$. In the first case we have that $A\geq 4$, in the second case we have $C\geq 6$ by \cref{remark:11}, while in the third case we have, as in the proof of \cref{lemma:11:1}, that $v_0,v_5\in\mcl{C}$, which implies $C\geq 6$ by \cref{remark:11}.
\item[Case $A+C+\#\D_1=11$]
In this case, as observed in the proof of \cref{lemma:11:1},  
$\#\D_1$ is an odd number and for all $v_1\in\mcl D_1$ there are $v_2,v_3\in\mcl C$ such that $v_1,v_2,v_3$ are consecutive on $\face$. 
Hence, if $\#\D_1\geq 3$ we get $C\geq 6$. 
On the other hand if $\#\D_1 = 1$ we get $A+C=10$, which implies the thesis as well.
\end{description}
\end{proof}
By \cref{lemma:11:2}, we can conclude considering 2 final cases.

\begin{description}
\item[Case $C\geq 6$]
Then $c_+(\face)\great 6 c_3=0.01674$ and $B\leq 4$. 
Since $c_-(\face)\great Bc_1^\beta+(11-B-C)c_2$, 
we get $c_-(\face)\great 4c_1^{\beta}+c_2=-0.01644$, and so $c(\face)\great 0.0003$.
\item[Case $A\geq 4$]
%
Let $v_1,v_2,v_3,v_4\in\fverts(\face)$ be consecutive vertices on $\face$ with 
$v_2,v_3\in\mcl{A}$. 
We cannot have $\vtype(v_1)= (3,11,13)$ for, otherwise, there would exist $\tau\in\Faces$ with $\abs{\tau}=3$, $\fverts(\tau)=\{v_1,v_2,w\}$ and $(3,13,13)\subseteq \vtype(w)$, which is not admissible. 
Analogously $\vtype(v_4)\neq (3,11,13)$. 
Since $v_2v_3$ is not a blue edge, we must have that either $v_1\not\in\mcl{C}$ or $v_4\not\in\mcl{C}$. 
This implies that there are other elements of $\mcl{D}$ other than the one specified by \cref{lemma:11:1}. 
Since $D=11-A-B-C$ is odd, we deduce $D\geq 3$, so we necessarily get $A=4$, $B=0$, $C=4$, $D=3$ by \cref{lemma:11:1}. 
As a consequence, $c_+(\face)\great 4 c_3=0.01116$ and $c_-(\face)\great 4c_1^{\alpha}+3c_2=-0.01084$, hence $c(\face)\great 0.00032$.
\end{description}
\begin{figure}[ht]\centering 
\input{11_1111.tikz}
\quad
\quad
\input{11_1113.tikz}
\caption{Illustration of the final cases for $\protect \abs\face =11$.}
\label{fig:11:final}
\end{figure}

Thus, summing up, we obtain unconditionally that $c(\face)\great 0.0003$.


\section{Analysis of faces with 13 edges}\label{sec:13}

Let $\face\in \FFaces$ with $\abs{\face}=13$ and $\pair(\face)\neq 0$.
With our construction of the pairing, we have $\pair(v,\face)\neq 0$ if and only if $v\in\fverts(\face)$,  and $\vtype(v)$ is one of the 5 multisets listed in \cref{table:13}. 
As in the previous section, we refer to \cref{def:blue} for the definition of $\alpha$-vertices, $\beta$-vertices and blue edges.
We also define the constants 
$c_1^{\alpha}=-0.00721$, $c_1^{\beta}=-0.00481$ and $c_2=0.06735$. 

\begin{table}[H]\centering
\begin{tabular}{ccccl}
\toprule 
$\vtype(v)$
	\ \ \ &\ \ \ 
\text{where}
	\ &\
$\pair(v,f)$
	\ \ \ & \ 
 $c_v\pair(v,\face)$&\\

\midrule

$(3,3,13)$
	\ \ \ &\ \ \ 

	 \ & \ 
$1$
	\ \ \ &\ \ \ 
$\great 0.234$& \\

\midrule

$(3,4,13)$
	\ \ \ &\ \ \ 

	 \ & \ 
$\frac{1}{2}$
	\ \ \ &\ \ \ 
$\great 0.075$ &\\

\midrule

$(3,11,13)$
	\ \ \ &\ \ \ 
$v$ is $\alpha$-vertex
	 \ & \ 
$\frac{6}{7}$
	\ \ \ &\ \ \ 
$\great -0.00721$&$=c_1^{\alpha}$\\

\midrule 

$(3,11,13)$
	\ \ \ &\ \ \ 
$v$ is $\beta$-vertex
	 \ & \ 
$\frac{4}{7}$
	\ \ \ &\ \ \ 
$\great -0.00481$&$=c_1^{\beta}$\\

\midrule

$(4,4,13)$
	\ \ \ &\ \ \ 

	 \ & \ 
$1$
	\ \ \ &\ \ \ 
$\great 0.06735$&$=c_2$\\

\midrule

$(3,3,3,13)$
	\ \ \ &\ \ \ 

	 \ & \ 
$1$
	\ \ \ &\ \ \ 
$\great 0.06735$&$=c_2$\\

\bottomrule
\end{tabular}
\caption{Vertices contributing nontrivially when $\protect\abs\face = 13$.}
\label{table:13}
\end{table}

\vspace{3pt}

Let $A=\#\mcl{A}$, $B=\#\mcl{B}$, $C=\#\mcl{C}$, $L=\#\mcl{L}$ and $M=\#\mcl{M}$, where
\[
\begin{aligned}
	\mcl{A}&=\{v\in\fverts(\face):\ v \text{ is an $\alpha$-vertex}\},\\
	\mcl{B}&=\{v\in\fverts(\face):\ v \text{ is a $\beta$-vertex}\},\\
	\mcl{C}&=\{v\in\fverts(\face):\ \vtype(v)\in\{(3,3,13), (3,4,13), (4,4,13), (3,3,3,13)\}\},\\
	\mcl{L}&=\{e\in\fedges(\face):\ \etype(e)=(11,13)\},\\
	\mcl{M}&=\{e\in\fedges(\face):\ e \text{ is a blue edge}\}.
\end{aligned}
\]
Clearly $c_-(\face)\geq A c_1^{\alpha}+B c_1^{\beta}$, and $c_+(\face)\geq C c_2$. 
We have $2L=A+B\leq 12$, since it is both even and at most 13. Moreover, $\beta$-vertices are by definition the endpoints of blue edges, and so $B=2M$ is even. We also notice that $\mcl{C}$ is not empty, because otherwise for every $v\in\fverts(\face)$ we would have $\vtype(v)=(a,b,13)$ for some $a\in\{3,4\}$ and $5\leq b\leq 11$. This would imply that half of the edges $e\in\fedges(\face)$ would satisfy $\etype(e)=(3,13)$ or $\etype(e)=(4,13)$, but this is not possible because $\#\fedges(\face)=13$ is odd.
We now prove that $c(\face)\great 0$ considering the following 3 cases.
\begin{figure}[ht]\centering
\input{13_A10.tikz}

\input{13_A12.tikz}
\caption{Illustrations for $\protect\abs\face = 13$.}
\label{fig:13}
\end{figure}
\begin{description}
\item[Case $A+B\leq 8$]
Then $c_-(\face)\great 8c_1^{\alpha}=-0.05768$ and $c_+(\face)\great c_2=0.06735$, so $c(\face)\great 0.00967$.
\item[Case $A+B=10$ and $B\geq 2$]
Then $c_-(\face)\great 8c_1^{\alpha}+2c_1^{\beta}=-0.0673$ and $c_+(\face)\great c_2=0.06735$, so $c(\face)\great 0.00005$.
\item[Case $A=10$ and $B=0$]
By the definition of blue edges, it's not possible to have more than 4 consecutive $\alpha$-vertices on $\face$. 
With this observation, if we let $\fverts(\face)=\langle v_1,v_2,\ldots,v_{12},v_{13}\rangle$, then, up to cyclic reordering, we must have $\fverts(\face)\bsl\mcl{A}=\{v_2, v_7, v_{12}\}$.
Moreover for $i\in\{2,7,12\}$ we also have $\etype(v_{i-1}v_i)=\etype(v_iv_{i+1})=(3,13)$. This implies that for $i\in\{2,7,12\}$ we have $\vtype(v_i)=(3,3,13)$ or $\vtype(v_i)=(3,3,3,13)$. As a consequence we get $C=3$, so $c_+(\face)\great 3 c_2= 0.20205$. Since $c_-(\face)\great 10 c_1^{\alpha}\great -0.0721$, we obtain $c(\face)\great 0.12995$.
\item[Case $A+B=12$]
Write $\fverts(\face)=\langle v_1,\ldots,v_{13}\rangle$ as above. Without loss of generality we suppose that $v_1$ is the only element of $\fverts(\face)$ which is not in $\mcl{A}\cup\mcl{B}$. Then we immediately see that $\mcl{A}=\{v_2,v_3,v_{12},v_{13}\}$, so $A=4$ and $B=8$. Therefore $c_-(\face)\great 4c_1^{\alpha}+8c_1^{\beta}= 0.06732$. Since $c_+(\face)\great c_2=0.06735$ we get $c(\face)\great 0.00003$.
\end{description}

In any case we obtain $c(\face)\great 0.00003$.


\section{Analysis of faces with $N$ edges, where $14\leq N\leq 39$ and $N\neq 19$}\label{sec:N}

Let $\face\in \FFaces$ with $\pair(\face)\neq 0$ and $\abs{\face}=N$, where $14\leq N\leq 39$ and $N\neq 19$.
With our construction of the pairing, we have $\pair(v,\face)\neq 0$ if and only if $v\in\fverts(\face)$ and $3\in \vtype(v)$, or if $v\not\in\fverts(\face)$ but $v$ is special to $\face$ as specified by rule \cref{rule:3ab}. We list all the possibilities in \cref{table:N}, and we set the constants 
$c_1=-0.0078$, 
$c_2=0.016$ and 
$c_3=0.023$. 

\begin{table}[H]\centering
\begin{tabular}{ccccl}
\toprule
$\vtype(v)$
	&
\text{where}
 &
$\pair(v,\face)$
	 & 
 $c_v\pair(v,\face)$& \\

\midrule

$(3,3,N)$
	& 

	 & 
$1$
	&
$\great 0.182$&\\

\midrule

$(3,4,N)$
	& 

	 &
$\frac{1}{2}$
	  
$\great 0.049$&\\

\midrule

$(3,5,N)$
	 &  

	& 
$\frac{1}{2}$
	& 
$\great 0.024$&\\

\midrule

$(3,6,N)$
	 &

 &
$1$
	 &
$\great 0.016$&$=c_2$\\ 

\midrule

$(3,7,a)$
	 & 
\cref{rule:3ab}, $a\in\{7,10\}$
	& 
$1$
	&
$\great 0.066$&\\

\midrule

$(3,7,a)$
	&
\cref{rule:3ab}, $a\in\{8,9\}$
	 & 
$\frac{1}{2}$
	&
$\great 0.038$&\\

\midrule

$(3,7,N)$
	&

	  &
$1$
	&
$\great -0.0078$&$=c_1$\\

\midrule

$(3,a,b)$
	&
\cref{rule:3ab}, $a,b\in\{8,9,10\}$
	 
$1$
	&
$\great 0.023$&$=c_3$\\

\midrule

$(3,8,N)$
	&
$N\leq 23$
	 &  
$1$
	&
$\great -0.0078$&$=c_1$\\

\midrule

$(3,9,N)$
	 &
$N\leq 17$
	 &
$1$
	& 
$\great -0.0064$&\\

\midrule

$(3,10,N)$
	 &
$N=14$
	 & 
$1$
	&
$\great -0.0049$&\\

\midrule

$(3,3,3,N)$
	 & 

	& 
$1$
	&
$\great 0.016$&$=c_2$\\

\bottomrule
\end{tabular}
\caption{Vertices contributing nontrivially when $ 14\leq \protect\abs\face \leq 39$, $\protect\abs\face \neq 19$.}
\label{table:N}
\end{table}

If $v\in\Verts$ is a vertex special to $\face$, then $\vtype(v)=(3,a,b)$ for some $a,b\in\{7,8,9,10\}$. Therefore there exists a well defined $\Ntau_v\in\vfaces(v)$ such that $\abs{\Ntau_v}=3$. Moreover, there exists a well defined $\Nedge_v\in\fedges(\face)$ with $\Nedge_v\in\fedges(\Ntau_v)$, and so also $\etype(\Nedge_v)=(3,N)$. Now consider
\[
\begin{aligned}
  \mcl{A}&:=\{v\in \fverts(\face):\ \vtype(v)=(3,a,N), \ a\in\{7,8,9,10\}\},\\
	\NVerts&:=\{v\in \fverts(\face):\ 3\in \vtype(v)\}
	  \cup
		\{v\in \Verts:\ v \text{ is special to $\face$}\},\\
	\NEdges&:=\{e\in\fedges(\face):\ \etype(e)=(3,N)\}.
\end{aligned}
\]
In order to prove that $c(\face)\great 0$ we employ the same strategy that lies behind our proof of \cref{thm:main}. Namely, we are about to define a pairing $\Npair:\NVerts\times\NEdges\to\Q_{\geq 0}$, so that we can decompose the total contribution $c(\face)$ as a sum of easier local contributions $\Ncon(e)$, for $e\in\NEdges$ (see \cref{lemma:N:1} below).
Notice that $\NVerts$ and $\NEdges$ are sets, by \cref{lemma:distinct}.

\begin{definition}\label{def:lpair}
We define $\Npair:\NVerts\times\NEdges\to\Q_{\geq 0}$ to be the only function such that, for every  $v\in\NVerts$ and $e\in\NEdges$, we have:
\begin{itemize}
\item
if $\vtype(v)=(3,3,N)$ or $\vtype(v)=(3,3,3,N)$, and $e\in\vedges(v)$, then $\Npair(v,e)=\frac{1}{2}\pair(v,\face)$;
\item
if $3\in\vtype(v)$ but $(3,3)\not\subseteq \vtype(v)$, and $e\in\vedges(v)$, then $\Npair(v,e)=\pair(v,\face)$;
\item
if $v$ is special to $\face$ and $e=\Nedge_v$, then $\Npair(v,e)=\pair(v,\face)$; 
\item
otherwise $\Npair(v,e)=0$.
\end{itemize}
\end{definition}

For every $e\in\NEdges$ we let $\Ncon(e):=\sum_{v\in\NVerts} c_v\Npair(v,e)$. The analogue of \eqref{eq:pairing} is provided by the fact that for every $v\in\NVerts$ we have
\[
\sum_{e\in\NEdges} \Npair(v,e)=\pair(v,\face).
\]
Proceeding as in \cref{lemma:main}, we deduce the following lemma.
\begin{lemma}\label{lemma:N:1}
\[c(\face) = \sum_{e\in\NEdges} \Ncon(e)\]
\end{lemma}

Since we assumed $\pair(\face)\neq 0$, the (multi)set $\NEdges$ is nonempty. Therefore, to prove that $c(\face)\great 0$, it suffices to prove that $\Ncon(e)\great 0$ for every $e\in\NEdges$.
Let then $e\in\NEdges$ and $\everts(e)=\{v_1,v_2\}$. We distinguish, without loss of generality, between 3 cases.

\begin{figure}[ht]\centering
\input{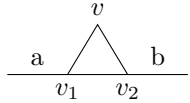}
\caption{Illustration for $\protect \abs \face = N$,  where $14\leq N\leq 39$ and $N\neq 19$}
\label{fig:N}
\end{figure}
\begin{description}
\item[Case $v_1,v_2\not\in\mcl{A}$]
Then, we see that $\Ncon(e)=c_{v_1}\Npair(v_1,e)+c_{v_2}\Npair(v_2,e)\great \frac{1}{2}c_2+\frac{1}{2}c_2=0.016$.
\item[Case $v_1\in\mcl{A}$, $v_2\not\in\mcl{A}$]
Then, $\Ncon(e)=c_{v_1}\Npair(v_1,e)+c_{v_2}\Npair(v_2,e)\great c_1+\frac{1}{2}c_2=0.0002$.
\item[Case $v_1,v_2\in\mcl{A}$]
Let $\efaces(e)=(\tau,\face)$. Since $\etype(e)=(3,N)$ the face $\tau$ is a triangle, so there exists $v\in\Verts$ such that $\fverts(\tau)=\langle v_1, v_2, v\rangle$. By definition of $\mcl{A}$ we have that $\vtype(v_1)=(3,a,N)$ and $\vtype(v_2)=(3,b,N)$ for some $a,b\in\{7,8,9,10\}$. Then we must have $\vtype(v)=\{3,a,b\}$ and by \cref{rule:3ab} we have that $v$ is special to $\face$. In particular, $\tau=\Ntau_v$ and $e=\Nedge_v$. Therefore $\Ncon(e)=c_{v_1}\Npair(v_1,e)+c_{v_2}\Npair(v_2,e)+c_v\Npair(v,e)\great c_1+c_1+c_3\great 0.007$.
\end{description}

In any case $\Ncon(e)\great 0.0002$ for all $e\in\NEdges$. Hence we get $c(\face)\great 0.0002$ as well by \cref{lemma:N:1}.


\section{Analysis of faces with 19 edges}\label{sec:19}

Let $\face\in \FFaces$ with $\abs{\face}=19$ and $\pair(\face)\neq 0$.
With our construction of the pairing, we have $\pair(v,\face)\neq 0$ if and only if $v\in\fverts(\face)$ and either $4\in \vtype(v)$ or $5\in \vtype(v)$. We list all the possibilities in \cref{table:19}, and we define the constants 
$c_1:=-0.00175$ and 
$c_2=0.038$. 

\begin{table}[H]\centering
\begin{tabular}{ccccl}
\toprule
$\vtype(v)$
	\ \ \ &\ \
\text{where}
	\ \ &\
$\pair(v,\face)$
	\ \ \ & \ 
 $c_v\pair(v,\face)$& \\

\midrule

$(3,4,19)$
	\ \ \ &\ \ \ 

	 \ & \ 
$\frac{1}{2}$
	\ \ \ &\ \ \ 
$\great 0.063$&\\

\midrule

$(3,5,19)$
	\ \ \ &\ \ \ 

	 \ & \ 
$\frac{1}{2}$
	\ \ \ &\ \ \ 
$\great 0.038$& $=c_2$\\

\midrule

$(4,4,19)$
	\ \ \ &\ \ \ 

	 \ & \ 
$1$
	\ \ \ &\ \ \ 
$\great 0.043$ &\\

\midrule

$(4,5,19)$
	\ \ \ &\ \ \ 

	 \ & \ 
$\frac{1}{4}$
	\ \ \ &\ \ \ 
$\great -0.00175$ &$=c_1$\\

\bottomrule
\end{tabular}
\caption{Vertices contributing nontrivially when $\protect \abs \face = 19$.}
\label{table:19}
\end{table}

Let $A=\#\mcl{A}$, $B=\#\mcl{B}$ and $L=\#\mcl{L}$, where
\[
\begin{aligned}
	\mcl{A}&=\{v\in\fverts(\face):\ \vtype(v)=\{(4,5,19)\},\\
	\mcl{B}&=\{v\in\fverts(\face):\ \vtype(v)=\{(3,5,19)\}\},\\
	\mcl{L}&=\{e\in\fedges(\face):\ \etype(e)=(5,19)\}.
\end{aligned}
\]
We notice that $c_-(\face)\geq A c_1$ and that $2L=A+B\leq 18$, since it is both even and at most 19.
We now prove that $c(\face)\great 0$ considering the following 2 cases.

\begin{description}
\item[Case $A=0$]
Here all the contributions are positive and so $c(\face)=c_+(\face)\great c_2=0.038$.
\item[Case $A\great 0$]
Since $1\leq A\leq 18$ we have that $\mcl{A}\neq \emptyset$ and $\mcl{A}\neq \fverts(\face)$. Therefore, there must exist $v\in\mcl{A}$ and $w\in\fverts(\face)\bsl\mcl{A}$ which are consecutive vertices of $\face$. In this case $\vtype(w)\neq(4,5,19)$, but $4\in\vtype(w)$ or $5\in \vtype(w)$, so is one of the other multisets listed in the table above. Hence $c_+(\face)\geq c_w\pair(w,\face)\great c_2=0.038$ and $c_-(\face)\great 18 c_1=-0.0315$, so we get $c(\face)\great 0.0065$.
\end{description}
In either case we obtain $c(\face)\great 0.0065$.


\section{Analysis of faces with 40 or 41 edges}\label{sec:4041}

Let $\face\in \FFaces$ with $\pair(\face)\neq 0$ and $\abs{\face}=N$, where $N\in\{40,41\}$.
With our construction of the pairing, we have $\pair(v,\face)\neq 0$ if and only if $v\in\fverts(\face)$ and $3\in \vtype(v)$, or if $v\not\in\fverts(\face)$ but $v$ is special to $\face$ as specified by rules \cref{rule:TS}, \cref{rule:3ab}, \cref{rule:333a} or \cref{rule:33335}. We list all the possibilities in \cref{table:4041}, and we set the constants 
$c_0=0.023$, 
$c_1=-0.009$, 
$c_2=-0.0084$, 
$c_3=0.181$, 
$c_4=0.049$, 
$c_5=0.048$, 
$c_6=0.0148$, 
$c_7=0.109$. 

\begin{table}[H]\centering
\begin{tabular}{ccccl}
\toprule
$\vtype(v)$
	\ \ \ &\ \ \ 
\text{by}
	\ &\
$\pair(v,\face)$
	\ \ \ & \ 
 $c_v\pair(v,\face)$& \\

\midrule

$\{\text{only 3,4}\}$ 
	\ \ \ &\ \ \ 
\cref{rule:TS}
	 \ & \ 
$\frac{1}{3}$
	\ \ \ &\ \ \ 
$\great 0.024$&$\great c_0$\\

\midrule

$(3,3,N)$
	\ \ \ &\ \ \ 

	 \ & \ 
$1$
	\ \ \ &\ \ \ 
$\great 0.181$&$=c_3$\\

\midrule

$(3,4,N)$
	\ \ \ &\ \ \ 

	 \ & \ 
$\frac{1}{2}$
	\ \ \ &\ \ \ 
$\great 0.049$&$=c_4$\\

\midrule

$(3,5,N)$
	\ \ \ &\ \ \ 

	 \ & \ 
$1$
	\ \ \ &\ \ \ 
$\great 0.048$&$=c_5$\\

\midrule

$(3,6,7)$
	\ \ \ &\ \ \ 
\cref{rule:3ab}
	 \ & \ 
$1$
	\ \ \ &\ \ \ 
$\great 0.133$&$\great c_7$\\

\midrule

$(3,6,N)$
	\ \ \ &\ \ \ 

	 \ & \ 
$1$
	\ \ \ &\ \ \ 
$\great 0.0148$&$=c_6$\\  

\midrule

$(3,7,7)$
	\ \ \ &\ \ \ 
\cref{rule:3ab}
	 \ & \ 
$1$
	\ \ \ &\ \ \ 
$\great 0.109$&$=c_7$\\

\midrule

$(3,7,40)$
	\ \ \ &\ \ \ 

	 \ & \ 
$1$
	\ \ \ &\ \ \ 
$\great -0.0084$&$=c_2$\\

\midrule

$(3,7,41)$
	\ \ \ &\ \ \ 

	 \ & \ 
$1$
	\ \ \ &\ \ \ 
$\great -0.009$&$=c_1$\\

\midrule

$(3,3,3,5)$
	\ \ \ &\ \ \ 
\cref{rule:333a}
	 \ & \ 
$1$
	\ \ \ &\ \ \ 
$\great 0.190$&$\great c_0$\\

\midrule

$(3,3,3,N)$
	\ \ \ &\ \ \ 

	 \ & \ 
$1$
	\ \ \ &\ \ \ 
$\great 0.0148$&$=c_6$\\

\midrule

$(3,3,3,3,5)$
	\ \ \ &\ \ \ 
\cref{rule:33335}
	 \ & \ 
$1$
	\ \ \ &\ \ \ 
$\great 0.023$&$=c_0$\\

\bottomrule
\end{tabular}
\caption{Vertices contributing nontrivially when $\protect \abs\face \in\{40,41\}$.}
\label{table:4041}
\end{table}

\vspace{3pt}

If $v\in\Verts$ is a vertex special to $\face$, then $\vtype(v)=(3,6,7)$, $\vtype(v)=(3,7,7)$, $\vtype(v)=(3,3,3,5)$, $\vtype(v)=(3,3,3,3,5)$ or $v$ is a TS vertex. In case $\vtype(v)=(3,6,7)$ or $\vtype(v)=(3,7,7)$, there exist well defined $\ltau_v\in\vfaces(v)$ and $\ledge_v\in\fedges(\face)$ such that $\abs{\ltau_v}=3$, $\ledge_v\in\fedges(\ltau_v)$, and so also $\etype(\ledge_v)=(3,N)$. Now we consider
\[
\begin{aligned}
	\mcl{S}&:=\{v\in \fverts(\face):\ \text{$v$ is special to $\face$},\ \vtype(v)\neq (3,6,7),\ \vtype(v)\neq (3,7,7)\},\\
  \mcl{A}&:=\{v\in \fverts(\face):\ \vtype(v)=(3,7,N)\},\\
	\mcl{L}&:=\{e\in\fedges(\face):\ \etype(e)=(7,N)\},\\
	\lVerts&:=\{v\in \fverts(\face):\ 3\in \vtype(v)\}
	  \cup
		\{v\in \Verts:\ v \text{ is special to $\face$}\},\\
	\lEdges&:=\{e\in\fedges(\face):\ \etype(e)=(3,N)\}.
\end{aligned}
\]
Notice that $\lVerts$ and $\lEdges$ are sets. Indeed, $\Verts$ is a set by definition, while $\fverts(\face)$ and $\fedges(\face)$ are sets by \cref{lemma:distinct}.
In order to prove that $c(\face)\great 0$ we elaborate on the strategy used in \cref{sec:N}.

\begin{definition}\label{def:lpair:4041}
We introduce an auxiliary symbol $\lddface$ and we define $\lpair:\lVerts\times(\lEdges\cup\{\lddface\})\to\Q_{\geq 0}$ to be the only function such that, for every  $v\in\lVerts$ and $e\in\lEdges\cup\{\lddface\}$, we have:
\begin{itemize}
\item
if $\vtype(v)=(3,3,N)$ or $\vtype(v)=(3,3,3,N)$, and $e\in\vedges(v)$, then $\lpair(v,e)=\frac{1}{2}\pair(v,\face)$;
\item
if $(3,N)\subseteq \vtype(v)$ but $(3,3)\not\subseteq \vtype(v)$, and $e\in\vedges(v)$, then $\lpair(v,e)=\pair(v,\face)$;
\item
if $v$ is special to $\face$, $v\not \in\mcl S$ and $e=\ledge_v$, then $\lpair(v,e)=\pair(v,\face)$; 
\item
if $v\in\mcl{S}$ and $e=\lddface$, then $\lpair(v,e)=\pair(v,\face)$;
\item
otherwise $\lpair(v,e)=0$.
\end{itemize}
\end{definition}

The analogue of \eqref{eq:pairing} is provided by the fact that for every $v\in\lVerts$ we have
\[
\sum_{e\in\lEdges\cup\{\lddface\}} \lpair(v,e)=\pair(v,\face).
\]
For every $e\in\lEdges\cup\{\lddface\}$ we let $\lcon(e):=\sum_{v\in\lVerts} c_v\lpair(v,e)$. Let also
\[
\lcon_-(\face):=\sum_{\substack{e\in\lEdges\cup\{\lddface\}\\ \lcon(e)\less 0}} \lcon(e)
\quad
\text{ and }
\quad
\lcon_+(\face):=\sum_{\substack{e\in\lEdges\cup\{\lddface\}\\ \lcon(e)\geq 0}} \lcon(e).
\]
Proceeding as in \cref{lemma:main}, we deduce the following lemma.
\begin{lemma}\label{lemma:4041:1}
\[c(\face) = \lcon_-(\face)+\lcon_+(\face).\]
\end{lemma}

We immediately notice that $\lcon(\lddface)\geq\#\mcl{S} c_0\geq 0$. 
Moreover, we have that $\lEdges$ is necessarily nonempty. Indeed, otherwise, for every $v\in\fverts(\face)$ we would have $\vtype(v)=(4,4,N)$, and so no vertex could be special to $\face$; this would imply $\pair(\face)= 0$, against our assumption. 
Let now $e\in\lEdges$ and $\everts(e)=\{v_1,v_2\}$. We estimate $\lcon(e)$ from below, considering, without loss of generality, the following 3 cases.

\begin{description}
\item[Case $v_1,v_2\not\in\mcl{A}$]
Then, we see that $\lcon(e)=c_{v_1}\lpair(v_1,e)+c_{v_2}\lpair(v_2,e)\great \frac{1}{2}c_6+\frac{1}{2}c_6=0.0148$.
\item[Case $v_1\in\mcl{A}$, $v_2\not\in\mcl{A}$]
If $N=41$, then $\lcon(e)=c_{v_1}\lpair(v_1,e)+c_{v_2}\lpair(v_2,e)\great c_1+\frac{1}{2}c_6=-0.0016$. If $N=40$, then $\lcon(e)=c_{v_1}\lpair(v_1,e)+c_{v_2}\lpair(v_2,e)\great c_2+\frac{1}{2}c_6=-0.001$.
\item[Case $v_1,v_2\in\mcl{A}$]
Let $\efaces(e)=(\tau,\face)$. Since $\etype(e)=(3,N)$ the face $\tau$ is a triangle, so there exists $v\in\Verts$ such that $\fverts(\tau)=\langle v_1, v_2, v\rangle$. By definition of $\mcl{A}$ we have that $\vtype(v_1)=(3,7,N)$ and $\vtype(v_2)=(3,7,N)$. Then we must have $\vtype(v)=\{3,7,7\}$ and by \cref{rule:3ab} we have that $v$ is special to $\face$. In particular, $\tau=\ltau_v$ and $e=\ledge_v$. Therefore $\lcon(e)=c_{v_1}\lpair(v_1,e)+c_{v_2}\lpair(v_2,e)+c_v\lpair(v,e)\great \min\{c_1,c_2\}+\min\{c_1,c_2\}+c_7= 0.091$.
\end{description}

From the above analysis we deduce that $\lcon_-(\face)\great 41\cdot(-0.0016)\great -0.066$.
We now rule out the following 6 cases, in which $\lddface$ plays no role. 

\begin{figure}[ht]\centering
\input{4041_4.tikz}
\quad\quad
\input{4041_5.tikz}

\input{4041_76.tikz}
\ 
\input{4041_77.tikz}
\caption{Illustrations for $\protect \abs\face\in\{40,41\}$.}
\label{fig:4041}
\end{figure}
\begin{description}
\item[Case $\exists e\in\fedges(\face)$ with  $\etype(e)=(4,N)$]
We observe that if $e_1,e_2\in\fedges(\face)$ are consecutive edges on $\face$, with $\etype(e_2)=(4,N)$, then $\etype(e_1)=(3,N)$ or $\etype(e_1)=(4,N)$. Now fix an orientation $\pi$ of the sphere. Since $\lEdges\neq\emptyset$, if we visit counterclockwise the edges in $\fedges(\face)$, we must eventually find a pair of edges $e_1,e_2\in\fedges(\face)$ clockwise consecutive on $\face$, with $\etype(e_1)=(3,N)$ and $\etype(e_2)=(4,N)$. Analogously, if we visit the edges in $\fedges(\face)$ clockwise, starting from $e_2$, we must find a pair of edges $e_3,e_4\in\fedges(\face)$ clockwise consecutive on $\face$, with $\etype(e_3)=(4,N)$ and $\etype(e_4)=(3,N)$. 
In case $e_1=e_4$ we get $\lcon_+(\face)\geq \lcon(e_1)\great c_4+c_4=0.098$. If otherwise $e_1\neq e_4$, we get $\lcon_+(\face)\geq \lcon(e_1)+\lcon(e_4)\great 2c_4+2c_1=0.080$. In either case we obtain $c(\face)\great 0.014$.
\item[Case $\exists e\in\fedges(\face)$ with $\etype(e)=(5,N)$]
Let $\everts(e)=\{v_1,v_2\}$ and let $e_1,e_2\in\fedges(\face)\setminus\{e\}$ with $e_i\in\vedges(v_i)$ for $i=1,2$. For $i=1,2$ we must have $\vtype(v_i)=(3,5,N)$ and $e_i\in\lEdges$, so $\lcon(e_i)\great c_5+c_1=0.039$. Therefore $\lcon_+(\face)\geq \lcon(e_1)+\lcon(e_2)\great 0.078$, and $c(\face)\great 0.012$.
\item[Case $\exists \langle e_1,e_2,e_3\rangle\subseteq\fedges(\face)$ with $\etype(e_1)=\etype(e_3)=(7,N)$]
This implies that $\etype(e_2)=(3,N)$ and $\everts(e_2)\subseteq\mcl{A}$, and we already saw that in this case we have $c_+(\face)\geq c(e_2)\great 0.091$ and so $c(\face)\great 0.025$.
\item[Case $\exists \langle e_1,e_2,e_3\rangle\subseteq\fedges(\face)$ with $\etype(e_1)=(7,N)$ and $\etype(e_3)=(6,N)$]
This case is similar to the above one. Let $\efaces(e_2)=\{\tau,\face\}$ and $\everts(e_2)=\{v_1,v_2\}$. We must have that $\etype(e_2)=(3,N)$, so the face $\tau$ is a triangle, hence there exist $v\in\Verts$ such that $\fverts(\tau)=\langle v_1, v_2, v\rangle$. Then $\vtype(v_1)=(3,7,N)$ and $\vtype(v_2)=(3,6,N)$, so we must have $\vtype(v)=(3,6,7)$, and by \cref{rule:3ab}  $v$ is special to $\face$. In particular, $\tau=\ltau_v$ and $e_2=\ledge_v$. Therefore $\lcon(e_2)=c_{v_1}\lpair(v_1,e_2)+c_{v_2}\lpair(v_2,e_2)+c_v\lpair(v,e_2)\great c_1+c_6+c_7\great 0.114$, hence $c(\face)\great 0.048$.
\item[Case $\exists v\in\fverts(\face)$ with $\vtype(v)=(3,3,N)$]
Take $e\in\vedges(v)$ with $e\in\lEdges$. Then $\lcon_+(\face)\geq \lcon(e)\great c_3+c_1=0.172$, hence $c(\face)\great 0.106$.
\item[Case $\#\mcl{L}\leq 12$]
The other cases being already ruled out, for every $v\in\fverts(\face)$ we have that $\vtype(v)=(3,6,N)$, $\vtype(v)=(3,7,N)$ or $\vtype(v)=(3,3,3,N)$. Hence for $N=41$ we directly check that $c(\face)\great 24 c_1 + 17 c_6 = 0.0356$, while for $N=40$ we get $c(\face)\great 24 c_2 + 16 c_6 = 0.0352$.
\end{description}

If we are in one of the cases listed above, then $c(\face)\great 0.012$, so we can suppose we are not.
We can make the following easy observation. 

\begin{remark}\label{rmk:4041}
If the above 9 cases don't apply, and if $e_1,e_2,e_3\in\fedges(\face)$ are consecutive edges on $\face$, with $\etype(e_1)=(7,N)$, then we necessarily have $\etype(e_2)=\etype(e_3)=(3,N)$. 
\end{remark}
The conclusion of \cref{rmk:4041} implies that $12\less \#\mcl L \leq \frac{N}{3}\less 14$, so $\#\mcl{L}=13$ by the above case-analysis. 
In particular, it's not difficult to deduce that, if $N=40$,  we must have $c_-(\face)\great 26 c_2$ and $c_+(\face)\great 14 c_6+\lcon(\lddface)$, while if $N=41$ we get  $c_-(\face)\great 26 c_1$ and $c_+(\face)\great 15 c_6+\lcon(\lddface)$. In either case we get $c(\face)\great \lcon(\lddface)-0.012$.

\begin{figure}[ht]\centering
\input{4041_7333.tikz}
\caption{Illustration of Lemma 12.4.}
\label{fig:4041:lemma}
\end{figure}
\begin{lemma}\label{lemma:4041:7333}
There exist $e_0,e_1,e_2,e_3\in\fedges(\face)$ consecutive edges on $\face$, with $\etype(e_0)=(7,N)$ and $\etype(e_1)=\etype(e_2)=\etype(e_3)=(3,N)$.
\end{lemma}

\begin{proof}
This lemma is easy to prove.
We fix an orientation $\pi$ of the sphere. For every $e_0\in\mcl{L}$ we can define $n(e_0)$ to be the smallest integer $n\geq 3$ such that if $e_0,e_1,\ldots,e_n\in\fedges(\face)$ are clockwise consecutive edges of $\face$, then $e_n\in\mcl{L}$.
Clearly, $\sum_{e\in\mcl{L}} n(e)=N$. Moreover, 40 can be written as a sum of 13 integers $\geq 3$ in only one way: $40=12\cdot 3 + 4$; instead, 41 can be written in two ways: $41=12\cdot 3 + 5 = 11\cdot 3+4+4$. Therefore, to complete our analysis it suffices to consider the following two cases (suppose an orientation $\pi$ of the sphere is chosen).
\begin{description}
\item[Case $\exists e\in\fedges(\face)$ with $n(e)=4$]
Let $e_0,e_1,e_2,e_3,e_4\in\fedges(\face)$ be clockwise consecutive edges on $\face$, with $e_0,e_4\in\mcl{L}$ and $e_1,e_2,e_3\not\in\mcl{L}$. 
Since $e_0,e_1,e_2$ and $e_4,e_3,e_2$ are triples of consecutive edges on $\face$, we have $\etype(e_1)=\etype(e_2)=\etype(e_3)=(3,N)$ by \cref{rmk:4041}. 
\item[Case $\exists e\in\fedges(\face)$ with $n(e)=5$]
Let $e_0,e_1,e_2,e_3,e_4,e_5\in\fedges(\face)$ be clockwise consecutive edges on $\face$, with $e_0,e_5\in\mcl{L}$ and $e_1,e_2,e_3,e_4\not\in\mcl{L}$.
Since $e_0,e_1,e_2$ and $e_5,e_4,e_3$ are triples of consecutive edges on $\face$, we have $\etype(e_1)=\etype(e_2)=\etype(e_3)=\etype(e_4)=(3,N)$ by \cref{rmk:4041}.
\end{description}
\end{proof}

Let $e_0,e_1,e_2,e_3\in\fedges(\face)$ be as in \cref{lemma:4041:7333} and let $v_1,v_2,v_3,v_4\in\Verts$ such that $\everts(e_1)=\{v_1,v_2\}$, $\everts(e_2)=\{v_2,v_3\}$ and $\everts(e_3)=\{v_3,v_4\}$. 
They are mutually distinct vertices, by \cref{lemma:distinct}.
 Let $\tau_1,\tau_2,\tau_3,\tau_4,\tau_5\in\Faces$ be triangles such that $\vfaces(v_2)=\langle \tau_3,\tau_2,\tau_1,\face\rangle_\pi$ and $\vfaces(v_3)=\langle \tau_5,\tau_4,\tau_3,\face\rangle_\pi$. 
It's not difficult to see, with \cref{lemma:distinct}, that they also must be mutually distinct. 
Let $w_1,w_2,w_3\in\Verts$ such that $\fverts(\tau_2)=\langle w_1,v_2,w_2\rangle_\pi$ and $\fverts(\tau_4)=\langle w_2,v_3,w_3\rangle_\pi$. 
Again by \cref{lemma:distinct}, $w_2$ is distinct from $w_1$ and $w_3$ 
(although we don't need it, we also have $w_1\neq w_3$, because otherwise we could deduce that $(3,3,3,3,7)\subseteq \vtype(w_1)$, which is not admissible). 
We observe that $\{3,3,3\}\subseteq\vtype(w_2)$, but $\vtype(w_2)\neq (3,3,3,a)$ for $a\geq 6$, because otherwise $\{3,3,a,7\}\subseteq \vtype(w_1)$, which is not admissible. It then suffices to consider the following 3 cases.
	\begin{description}
	\item[Case $w_2$ is a TS vertex]
	In this case $w_2$ is special to $\face$ by \cref{rule:TS}, so $\lcon(\lddface)\great c_0$.
	\item[Case $\vtype(w_2)=(3,3,3,5)$]
	In this case $w_2$ is special to $\face$ by \cref{rule:333a}, so $\lcon(\lddface)\great c_0$.
	\item[Case $\vtype(w_2)=(3,3,3,3,5)$]
	In this case $w_2$ is special to $\face$ by \cref{rule:33335}, so $\lcon(\lddface)\great c_0$.
	\end{description}
Hence we obtain $\lcon(\lddface)\great c_0=0.023$, so $c(\face)\great 0.011$.
In conclusion, summing up the analysis of all the above cases, we have shown that $c(\face)\great 0.011$.


\section{There are no faces with more than 41 edges}\label{sec:big}

In this section we push our analysis a bit further than necessary in order to prove a result of independent interest in the classification of PCC graphs.
More precisely, we prove that a PCC graph cannot have a face of size greater or equal to 42. %
We state this fact as a theorem, the proof of which occupies the whole section.

\begin{theorem}\label{thm:big}
Suppose $G$ is a PCC  graph. Then for all $\face\in\Faces$ we have $\abs{\face}\leq 41$.
\end{theorem}

Let $\face\in \Faces$ with $\abs{\face}=N\geq 42$.
For $k\in\{1,\ld,6\}$ define $\mcl A_k$ as follows:  
\[
\begin{aligned}
\mcl{A}_1&:=\{v\in\fverts(\face):\ \vtype(v)=(3,3,3,N)\},\\  
\mcl{A}_2&:=\{v\in\fverts(\face):\ \vtype(v)=(4,4,N)\},\\
\mcl{A}_k&:=\{v\in\fverts(\face):\ \vtype(v)=(3,k,N)\},\quad\quad \quad \text{ if $3\leq k\leq 6$.}
\end{aligned}
\]
We notice that $\fverts(\face) = \bigcup_{k=1}^6 \mcl A_k$, and that all these multisets are actually sets by \cref{lemma:distinct}. 
Moreover, we have: 
$K(v)=\frac 1 N + \frac 1 6$ if $v\in\mcl A_3$; 
$K(v)=\frac 1 N + \frac 1 {12}$ if $v\in\mcl A_4$;
$K(v)=\frac 1 N + \frac 1 {30}$ if $v\in\mcl A_5$; and
$K(v)=\frac 1 N$ if $v\in\mcl A_1\cup\mcl A_2\cup\mcl A_6$.
We now define
\[
\begin{aligned}
\mcl B &:= \{e\in \fedges(\face):\ \etype(e)=(3,N)\};\\
\mcl T &:= \{\tau\in \Faces:\ \fverts(\tau)=\{v,v_1,v_2\} \text{ with } v_1,v_2\in\mcl A_1\cup\mcl A_5\cup\mcl A_6\};\\
\mcl C_1 &:= \{v\in \Verts\setminus \fverts(\face):\ \exists\tau\in\mcl T\text{ with } v\in \fverts(\tau)\};\\
\mcl C_2 &:= \{v\in \Verts\setminus \fverts(\face):\ \exists v'\in\mcl A_2\text{ with } vv'\in\Edges\}.
\end{aligned}
\]

Notice that if $\tau\in \mcl T$, $\fverts(\tau)=\{v,v_1,v_2\}$ and $v_1,v_2\in\fverts(\face)$ then necessarily $v_1v_2\in\fedges(\face)$ and $v\not\in \fverts(\face)$, by \cref{lemma:big:close} and \cref{lemma:two:edges}.

\begin{figure}[ht]\centering
\input{big_c1_3.tikz}
\ 
\input{big_c1_5.tikz}
\ 
\input{big_c1_6.tikz}
\caption{Illustrations for Lemma 13.2.}
\label{fig:big:c1}
\end{figure}

\begin{lemma}\label{big:c1}
For every $v\in\mcl C_1$ there is a unique $\tau_v\in\mcl T$ with $v\in\fverts(\tau_v)$.
Conversely, for all $\tau\in\mcl T$ there exists exactly one $v\in\mcl C_1$ with $v\in\fverts(\tau)$.
\end{lemma}
\begin{proof}
Let $v\in\mcl C_1$ and suppose there are $\tau_1,\tau_2\in\mcl T$ with $\tau_1\neq\tau_2$ and $v\in\fverts(\tau_1)\cap\fverts(\tau_2)$. 
We cannot have another $w\in\fverts(\tau_1)\cap\fverts(\tau_2)$, $w\neq v$, because otherwise we get $f(w)=(3,3,N)$, contrary to the definition of $\mcl T$. 
Since moreover $\vfaces(v)$ has at most 5 elements, we deduce that there is some $\kappa\in\vfaces(v)$ and some orientation of the sphere $\pi$ with $\langle \tau_1,\kappa,\tau_2\rangle_\pi\subseteq \vfaces(v)$, i.e. $\tau_1,\kappa,\tau_2$ cyclically consecutive around $v$.
Let $\fverts(\tau_1)=\langle v,v_1,v_2\rangle_\pi$ and $\fverts(\tau_2)=\langle v,v'_1,v'_2\rangle_\pi$
Since $v_1\in\mcl A_1\cup\mcl A_5\cup\mcl A_6$ and $v_1\in\fverts(\kappa)$, we have the following 2 cases.
\begin{description}
\item[Case $\abs{\kappa}=3$]
Then $\oppf(v,\kappa)$ must be triangle that shares two edges with $\face$, absurd by \cref{lemma:two:edges}.
\item[Case $\abs{\kappa}\in\{5,6\}$]
Then $\kappa$ shares two edges with $\face$, absurd by \cref{lemma:two:edges}.
\end{description}
So the first assertion is proved.  
The second one is a consequence of the remarks preceding this lemma.
\end{proof}

If we write $\fverts(\tau_v)=\{v,v_1,v_2\}$ for $v\in\mcl C_1$, we define $e_v:=v_1v_2\in\mcl B$.

\begin{lemma}\label{big:c2}
For every $v\in\mcl C_2$ there is a unique $v'\in\mcl A_2$ with $vv'\in\Edges$.
Conversely, for every $v\in\mcl A_2$ there is a unique $v'\in\mcl C_2$ with $vv'\in\Edges$.
\end{lemma}
\begin{proof}
Suppose there are $v_1,v_2\in\mcl A_2$ with $v_1\neq v_2$ and $vv_1,vv_2\in\Edges$. 
Then $\etype(vv_1)=\etype(vv_2)=(4,4)$ and so $\vtype(v)\in\{(4,4,4),(3,4,4,4)\}$.
In any case there exists $\kappa\in\efaces(vv_1)\cap\efaces(vv_2)$. 
We have that $\abs{\kappa}=4$ and we notice that $\kappa$ shares with $\face$ the two edges distinct from $vv_1$ and $vv_2$.
This contradicts \cref{lemma:two:edges}, so the first assertion is proved.
For the converse, notice that  for every $v'\in\mcl A_2$ there is exactly one $v\in\Verts$ such that $vv'\in\Edges\setminus\fedges(\face)$. By \cref{lemma:big:close} we have that $v\not\in\fverts(\face)$, and so $v\in \mcl C_2$.
Such $v$ is unique because $v\not\in\fverts(\face)$ and $vv'\in\Edges$ imply $vv'\in\Edges\setminus\fedges(\face)$.
\end{proof}

Consider also
\[
\begin{aligned}
\mcl C_{15}&:=\{v\in\mcl C_1:\ \fverts(\tau_v)=\{v,v_1,v_2\}\text{ with }v_1\in\mcl A_1\text{ and }v_2\in\mcl A_5\};\\
\mcl C_{16}&:=\{v\in\mcl C_1:\ \fverts(\tau_v)=\{v,v_1,v_2\}\text{ with }v_1\in\mcl A_1\text{ and }v_2\in\mcl A_6\};\\
\mcl D &:= \{v\in\Verts:\ \exists v'\in\mcl C_{16}\text{ with } vv'\in\Edges\text{ and } \etype(vv')=(5,6)\}.
\end{aligned}
\]

Notice that $(4,4)\subseteq f(v)$ for all $v\in\mcl C_2$ and $(5,6)\subseteq f(v)$ for all $v$ in $\mcl D$. 
We infer from this that $\mcl D$ is disjoint from $\mcl C_2$ and from $\fverts(\face)$. 
We could prove that it is disjoint from $\mcl C_1$ as well, but it's shorter to show only the following.

\begin{figure}[ht]\centering
\input{big_cd.tikz}
\caption{Illustration for Lemma 13.4.}
\label{fig:big:cd}
\end{figure}
\begin{lemma}\label{big:cd}
If $v\in \mcl C_1\cap\mcl D$, then $v\in\mcl C_{15}$.
\end{lemma}
\begin{proof}
If  $v\in \mcl C_1\cap\mcl D$ we have $(3,5,6)\subseteq \vtype(v)$, and so $\vtype(v)\in\{(3,5,6),(3,3,5,6)\}$.
Let $\face_1,\face_2\in\vfaces(v)$ with $\abs{\face_1}=6$ and $\abs{\face_2}=5$,  and let $\fverts(\face_1)=\langle v,v_1,v_2,v_3,v_4,v_5\rangle$ so that $\fedges(vv_1)=\{\face_1,\face_2\}$.
By the definition of $\mcl D$, we have that $v_1\in\mcl C_{16}$.
This easily implies that $v_2\in\mcl A_6$, and so necessarily $\etype(v_2v_3)=(6,N)$.
Let now $\tau_v\in\vfaces(v)$ as in \cref{big:c1}.
We notice that we cannot have $v_5\in\fverts(\tau_v)$. 
Indeed, otherwise we would have $v_5\in\mcl A_6$ and so $\etype(v_4v_5)=(6,N)$. 
But this is in contradiction with \cref{lemma:two:edges}.
Therefore we must have $\vfaces(v)=\langle \tau_v,\tau',\face_1,\face_2\rangle$ for some $\tau'\in\Faces$ with $\abs{\tau'}=3$.
Let $\fverts(\tau_v)=\{v,w_1w_2\}$ with $\etype(vw_1)=(3,5)$.
Then $\{\tau_v,\face_1\}\subseteq \vfaces(w_1)$ and $\{\tau_v,\tau'\}\subseteq \vfaces(w_2)$. 
Since $\tau_v\in\mcl T$, we must have $w_1\in\mcl A_5$ and $w_2\in\mcl A_1$, so $v\in\mcl C_{15}$.
\end{proof}


Our strategy to prove that $\face$ cannot exist, is to show that it forces the total curvature of the graph to exceed 2 (see \eqref{eq:total_curvature}).
In order to estimate, with sufficient precision, this global quantity from easier local computations, we employ once more (a weighted version of) our improved  discharging technique. Let
\[
\begin{aligned}
\Vbig&:=\fverts(\face)\cup\mcl C_1\cup\mcl C_2\cup\mcl D;\\
\FFbig&:=\mcl A_2\cup\mcl B.
\end{aligned}
\]

Now we define the pairing $\pbig:\Vbig\times\FFbig\to\Q_{\geq 0}$ as follows. Let $v\in\Vbig$.
\begin{description}
\item[Case $v\in\mcl A_1\cup\mcl A_3$]
Then there are exactly two elements $e_1,e_2\in\mcl B$ with $v\in\everts(e_1)\cap\everts(e_2)$.
We define $\pbig(v,e_1)=\pbig(v,e_2)=\frac 1 2 $ and  $\pbig(v,x)=0$ for all $x\in\FFbig\setminus\{e_1,e_2\}$.
\item[Case $v\in\mcl A_4\cup\mcl A_5\cup\mcl A_6$]
Then there is exactly one element $e\in\mcl B$ with $v\in\everts(e)$.
We define $\pbig(v,e)=1$ and  $\pbig(v,x)=0$ for all $x\in\FFbig\setminus\{e\}$.
\item[Case $v\in\mcl A_2$]
Then $v\in\FFbig$.
We define $\pbig(v,v)=1$ and  $\pbig(v,x)=0$ for all $x\in\FFbig\setminus\{v\}$.
\item[Case $v\in\mcl C_2$]
Then by \cref{big:c2} there is exactly one element $v'\in\mcl A_2$ with $vv'\in\Edges$.
We define $\pbig(v,v')=1$ and  $\pbig(v,x)=0$ for all $x\in\FFbig\setminus\{v'\}$.
\item[Case $v\in\mcl C_1\setminus\mcl D$]
Then, by the definition following \cref{big:c1}, we have  $e_v\in\mcl B$.
We define $\pbig(v,e_v)=1$ and  $\pbig(v,x)=0$ for all $x\in\FFbig\setminus\{e_v\}$.
\item[Case $v\in\mcl C_1\cap\mcl D$]
Then by \cref{big:cd} we have $v\in\mcl C_{15}$ and so necessarily $f(v)=(3,3,5,6)$. 
Hence there is exactly one $v'\in\Verts$ with $vv'\in\Edges$ and $\vtype(vv')=(5,6)$.
It follows from the definition of $\mcl D$ that $v'\in\mcl C_{16}\subseteq\mcl C_1$.
Therefore we have $e_{v},e_{v'}\in\mcl B$, and $e_{v}\neq e_{v'}$ because $v\neq v'$.
We define $\pbig(v,e_{v}) = \pbig(v,e_{v'})=\frac 1 2$ and  $\pbig(v,x)=0$ for all $x\in\FFbig\setminus\{e_v,e_{v'}\}$.
\item[Case $v\in\mcl D\setminus\mcl C_1$]
Let $\Dbig(v):=\{v'\in\mcl C_{16}\text{ with } vv'\in\Edges\text{ and } \etype(vv')=(5,6)\}$, and put $\nDbig_v:=\#\Dbig(v)$.
Notice that $\nDbig_v\leq 2$, where the equality might hold a priori in case $f(v)=(5,5,6)$ or $f(v)=(5,6,6)$.
Moreover for every $v'\in\Dbig(v)$ we have $e_{v'}\in\mcl B$.
We define $\pbig(v,e_{v'})=\frac 1 {\nDbig_v}$ for every $v'\in\Dbig(v)$ and  $\pbig(v,x)=0$ for all other elements $x\in\FFbig$.
\end{description}

Notice that for every  $v\in \Vbig$ we have $\sum_{x\in\FFbig} \pbig(v,x)=1$.
For every $x\in\FFbig$ we define
\[
\begin{aligned}
\cbig_x&:=\sum_{v\in\Vbig} \pbig(v,x) K(v);\\
\wbig_x&:=\sum_{v\in\fverts(\face)} \pbig(v,x),
\end{aligned}
\]
so that 
\begin{equation}\label{eq:big:sum}
\begin{aligned}
 \sum_{x\in\FFbig} \cbig_x &= \sum_{x\in\FFbig} \sum_{v\in\Vbig} \pbig(v,x)K(v) = \sum_{v\in\Vbig} K(v);\\
 \sum_{x\in\FFbig} \wbig_x &= \sum_{x\in\FFbig}\sum_{v\in\fverts(\face)} \pbig(x,v) = \sum_{v\in\fverts(\face)} 1 = N.
\end{aligned}
\end{equation}
Notice also that for every $x\in\FFbig$ all summands in the definition of $\cbig_x$ are nonnegative.
Therefore any partial sum of them will give an estimate of $\cbig$ from below. 

\begin{lemma}\label{big:main}
For every $x\in \FFbig$ we have $\cbig_x\great \wbig_x\left(\frac 1 N + \frac 1 {42}\right)$.
\end{lemma}
\begin{proof}
Let $x\in\FFbig$. We first check 1 easy case:
\begin{description}
\item[Case $x\in\mcl A_2$]
Then $\wbig_x=\pbig(x,x)=1$.
Moreover by \cref{big:c2} and the definition of $\pbig$ there is $v\in\mcl C_2$ with $\pbig(v,x)=1$.
We have that $(4,4)\subseteq f(v)$, and by \cref{lemma:big:close} we cannot have $f(v)=(4,4,M)$ with $M\geq 20$, so $K(v)\geq \frac 1 {19}$ if $\deg v = 3$. 
Otherwise we have $K(v)\geq\frac 1 {30}$ with equality when $f(v)=(3,4,4,5)$.
Hence $\cbig_x \geq \frac 1 N + \frac 1 {30}\great  \wbig_x\left(\frac 1 N + \frac 1 {42}\right)$.
\end{description}
If $x\not \in\mcl A_2$, we have $x=v_1v_2$ for some $v_1,v_2\in\fverts(\face)\setminus\mcl A_2$. 
We consider without loss of gernerality the following 4 similar cases:
\begin{description}
\item[Case $v_1\in\mcl A_3$ and $v_2\in\mcl A_3\cup\mcl A_1$]
Then $\wbig_x=\frac 1 2 + \frac 1 2 =1$.
Moreover $\cbig_x \geq\frac 1 2\left( \frac 1 N +\frac 1 6\right)+ \frac 1 2 \cdot \frac 1 N = \frac 1 N+\frac 1 {12}\great  \wbig_x\left(\frac 1 N + \frac 1 {42}\right)$.
\item[Case $v_1\in\mcl A_3$ and $v_2\in\mcl A_4\cup\mcl A_5\cup\mcl A_6$]
Then $\wbig_x=\frac 1 2 + 1 =\frac 3 2$.
Moreover $\cbig_x \geq\frac 1 2\left( \frac 1 N +\frac 1 6\right)+  \frac 1 N =\frac 3 2\left( \frac 1 N+\frac 1 {18}\right)\great  \wbig_x\left(\frac 1 N + \frac 1 {42}\right)$.
\item[Case $v_1\in\mcl A_4$ and $v_2\in\mcl A_3\cup\mcl A_1$]
Then $\wbig_x=1 + \frac 1 2 =\frac 3 2$.
Moreover $\cbig_x \geq\left( \frac 1 N +\frac 1 {12}\right)+ \frac 1 2 \cdot \frac 1 N = \frac 3 2\left(\frac 1 N+\frac 1 {18}\right)\great  \wbig_x\left(\frac 1 N + \frac 1 {42}\right)$.
\item[Case $v_1\in\mcl A_4$ and $v_2\in\mcl A_4\cup\mcl A_5\cup\mcl A_6$]
Then $\wbig_x=1+1=2$.
Moreover $\cbig_x \geq\left( \frac 1 N +\frac 1 {12}\right)+  \frac 1 N = 2\left(\frac 1 N+\frac 1 {24}\right)\great  \wbig_x\left(\frac 1 N + \frac 1 {42}\right)$.
\end{description}
If the above cases do not apply, we have that $\efaces(x)=\{\tau,\face\}$ for some $\tau\in\mcl T$.
Then, by \cref{big:c1}, there is a well-defined $v\in\mcl C_1$ with $\tau_v=\tau$. 
We have $\pbig(v,x)=1$ except when $v\in\mcl D\cap\mcl C_1$, in which case $\pbig(v,x)=\frac 1 2$.
We recall that, by \cref{big:cd}, the possibility $v\in\mcl D\cap\mcl C_1$ can occur only if $v_1\in\mcl A_1$ and $v_2\in\mcl A_5$ (or the opposite).
We now consider, without loss of generality, the following last 7 cases:
\begin{description}
\item[Case $v_1\in\mcl A_5$ and $v_2\in\mcl A_5$]
Then $\wbig_x=1+1=2$.
Moreover $\cbig_x \geq\left( \frac 1 N +\frac 1 {30}\right)+  \left(\frac 1 N+\frac 1 {30}\right) = 2\left(\frac 1 N+\frac 1 {30}\right)\great  \wbig_x\left(\frac 1 N + \frac 1 {42}\right)$.
\item[Case $v_1\in\mcl A_6$ and $v_2\in\mcl A_6$]
Then $\wbig_x=1+1=2$.
Moreover $f(v)=(3,6,6)$, so $K(v)=\frac 1 6$.
Hence $\cbig_x \geq \frac 1 N +\frac 1 N + \frac 1 6 = 2\left(\frac 1 N+\frac 1 {12}\right)\great  \wbig_x\left(\frac 1 N + \frac 1 {42}\right)$.
\item[Case $v_1\in\mcl A_1$ and $v_2\in\mcl A_1$]
Then $\wbig_x=\frac 1 2+\frac 1 2=1$.
Moreover $(3,3,3)\subseteq f(v)$, and by \cref{lemma:big:close} we cannot have $f(v)=(3,3,3,M)$ with $M\geq 20$.
Therefore we get, looking at the table of admissible face vectors, that $K(v)\geq \frac 1 {30}$, with equality if $f(v)=(3,3,3,3,5)$.
Hence $\cbig_x \geq \frac 1 2 \cdot\frac 1 N +\frac 1 2 \cdot\frac 1 N + \frac 1 {30} = \frac 1 N+\frac 1 {30}\great  \wbig_x\left(\frac 1 N + \frac 1 {42}\right)$.
\item[Case $v_1\in\mcl A_1$ and $v_2\in\mcl A_5$]
Then $\wbig_x=\frac 1 2+1=\frac 3 2$.
Moreover $(3,3,5)\subseteq f(v)$. 
If $v\in\mcl D$, then $f(v)=(3,3,5,6)$ and $\pbig(v,x)=\frac 1 2$, so $\pbig(v,x) K(v)=\frac 1 2 \cdot\frac 1 {30}=\frac 1 {60}$.
If otherwise $v\not\in\mcl D$, $\pbig(v,x)=1$ and $K(v)\geq \frac 1 {210}$ with equality if $f(v)=(3,3,5,7)$.
In any case $\cbig_x \geq \frac 1 2 \cdot\frac 1 N +\left(\frac 1 N + \frac 1 {30}\right) + \frac 1 {210} = \frac 3 2 \left(\frac 1 N+\frac 2 3\frac 8 {210}\right)\great  \wbig_x\left(\frac 1 N + \frac 1 {42}\right)$, because $\frac 1 {42}=\frac {15}{630}\less \frac {16}{630}$.
\item[Case $v_1\in\mcl A_5$ and $v_2\in\mcl A_6$]
Then $\wbig_x=1+1=2$.
Moreover $(3,5,6)\subseteq f(v)$, so $K(v)\geq \frac 1 {30}$ with equality if $f(v)=(3,3,5,6)$.
Hence $\cbig_x \geq\left(\frac 1 N +\frac 1{30}\right)+\frac 1 N + \frac 1 {30} = 2 \left(\frac 1 N+\frac 1{30}\right)\great  \wbig_x\left(\frac 1 N + \frac 1 {42}\right)$.
\item[Case $v_1\in\mcl A_1$, $v_2\in\mcl A_6$ and $f(v)\neq (3,3,5,6)$]
Then $\wbig_x=\frac 1 2+1=\frac 3 2$.
Moreover $(3,3,6)\subseteq f(v)$, so $K(v)\geq \frac 1 {12}$ with equality when $f(v)=(3,3,4,6)$.
Hence $\cbig_x \geq \frac 1 2 \cdot\frac 1 N +\frac 1 N + \frac 1 {12} = \frac 3 2 \left(\frac 1 N+\frac 1 {18}\right)\great  \wbig_x\left(\frac 1 N + \frac 1 {42}\right)$.
\item[Case $v_1\in\mcl A_1$, $v_2\in\mcl A_6$ and $f(v)= (3,3,5,6)$]
Then $\wbig_x=\frac 1 2+1=\frac 3 2$ and $K(v)=\frac 1 {30}$.
Moreover, there is $e\in\vedges(v)$ with $\etype(e)=(5,6)$. 
Writing $\everts(e)=\{v,w\}$ we find that $w\in\mcl D$. Notice that $(5,6)\subseteq f(w)$.
We consider the following 3 subcases.
\begin{description}
\item[Case $w\in\mcl D\cap \mcl C_1$]
Then $\pbig(w,x)=\frac 1 2$ and $K(w)= \frac 1 {30}$, because $f(w)=(3,3,5,6)$.
\item[Case $w\in\mcl D\setminus \mcl C_1$ and $\nDbig_w=1$]
Then $\pbig(w,x)=1$ and $K(w)\geq \frac 1 {210}$, with equality if $f(w)=(5,6,7)$.
\item[Case $w\in\mcl D\setminus \mcl C_1$ and $\nDbig_w=2$]
Then $\pbig(w,x)=\frac 1 2$ and $K(w)\geq \frac 1 {30}$, with equality if $f(w)=(5,6,6)$.
\end{description}
In any case $\pbig(w,x) K(w)\geq \frac 1 {210}$ and so $\cbig_x \geq \frac 1 2 \cdot\frac 1 N +\frac 1 N + \frac 1 {30} + \frac 1 {210} = \frac 3 2 \left(\frac 1 N+\frac 2 3\frac 8 {210}\right)\great  \wbig_x\left(\frac 1 N + \frac 1 {42}\right)$, because $\frac 1 {42}=\frac {15}{630}\less \frac {16}{630}$.
\end{description}
Thus, summing up, we have proved that the inequality $\cbig_x\great \wbig_x\left(\frac 1 N + \frac 1 {42}\right)$ holds for all $x\in\FFbig$.
\end{proof}

Now notice that $\Vbig\subseteq \Verts$ and so $\sum_{v\in\Vbig} K(v)\leq 2$ by \eqref{eq:total_curvature}.
Then, by \eqref{eq:big:sum} and \cref{big:main} we get
\[
2\geq  \sum_{x\in\FFbig} \cbig_x \great \sum_{x\in\FFbig} \wbig_x\left(\frac 1 N + \frac 1 {42}\right) =  1+\frac N {42},
\]
which is a contradiction if $N\geq 42$.
Thus \cref{thm:big} is proved.

\begin{remark}
	The idea of looking at the neighborhood of a big face to control its size can be traced back to \cite{DeVosMohar:3444}.
\end{remark}


\section{Analysis of the auxiliary face $\ddface$}\label{sec:ddface}

Suppose that $\pair(\ddface)\neq 0$ and let $v\in\Verts$ with $\pair(v,\ddface)\neq 0$.
We notice that $v$ cannot be a regular vertex or a $\dface$-vertex.
Moreover, $v$ cannot be a big vertex by \cref{thm:big}.
Thus, we consider the following 4 cases. 
See \cref{sec:pair:1} for the definition of vertex types.
\begin{description}
\item[Case $v$ is semi-regular]
Then $\vtype(v)=(3,5,a)$ for $a=11$ or $20\leq a\leq 39$. Therefore $\pair(v,\ddface)=\frac 1 2$ and $c_v\great 0.049$, so $c_v\pair(v,\ddface)\great 0.024$.
\item[Case $v$ is a TS-vertex]
Then $\pair(v,\ddface)\geq\frac 1 3$ and $c_v\great 0.073$, so $c_v\pair(v,\ddface)\great 0.024$.
\item[Case $v$ is a $\ddface$-vertex]
Then $\pair(v,\ddface)=1$. We observe that $c_v$ is smaller when the entries of $\vtype(v)$ are larger. 
Therefore we can estimate $c_v$ by checking the cases $\vtype(v)=(3,8,19)$, $(3,10,13)$, $(4,4,41)$, $(4,6,10)$, $(5,5,9)$, $(5,6,6)$, $(3,3,3,19)$, $(3,3,4,10)$ and $(3,3,5,6)$. 
We get, respectively: 
$c_v\great 0.001, 0.0006, 0.0148, 0.007, 0.0015, 0.023, 0.043, 0.007, 0.023$.
\item[Case $v$ is potentially-special]
Then $\pair(v,\ddface)\geq \frac 1 3$. 
Arguing  as above, we estimate $c_v$ by checking the cases 
$\vtype(v)=(3,4,41)$, $(3,10,10)$, $(4,5,6)$, $(3,3,3,12)$, $(3,3,4,7)$, $(3,3,3,3,5)$.
We get, respectively:
 $c_v\great 0.098, 0.023, 0.107, 0.073, 0.049, 0.023$, and so 
$c_v\pair(v,\ddface)\great 0.0007$. 
\end{description}
In all these cases we obtain $c_v\pair(v,\ddface)\great 0.0006$.
For the arguments of \cref{sec:208} we will need the above analysis performed with more accuracy under the additional assumption $5\in\vtype(v)$. 
We consider the following 3 cases.

\begin{description}
\item[Case $v$ is semi-regular]
Then $c_v\pair(v,\ddface)\great 0.024$.
\item[Case $v$ is a $\ddface$-vertex and $\vtype(v)\neq (5,5,9)$]
Then $\pair(v,\ddface)=1$. 
Moreover $\vtype(v)=\{3,5,a\}$ for some $a\leq 13$, or $\vtype(v)=(5,5,b)$ for some $b\leq 8$, or $\vtype(v)\in\{(5,6,6), (3,3,5,6), (3,3,5,5)\}$.
 In any case we get $c_v\pair(v,\ddface)\great 0.015$.
\item[Case $v$ is potentially-special]
Then $\vtype(v)\in\{(3,3,3,5),(3,3,3,3,5)\}$ and $\pair(v,\ddface) = 1$, or  $\vtype(v)\in\{(3,4,5),(4,5,6),(3,3,4,5)\}$ and $\pair(v,\ddface)= \frac 1 2$.  
Hence $c_v\pair(v,\ddface)\great 0.023$.
\end{description}

Therefore, we deduce that $c(\ddface)=c_+(\ddface)\great 0.0006$ unconditionally, or $c(\ddface)\great 0.015$ if there is $v\in\Verts$ with $\pair(v,\ddface)\neq 0$, $5\in\vtype(v)$ and $\vtype(v)\neq (5,5,9)$.


\section{Analysis of the auxiliary face $\dface$, and proving $\#\Verts\leq 210$}\label{sec:210}

Let $Z=\#\mcl{Z}$, where
\[
  \mcl{Z} = \{v\in\Verts:\ \vtype(v)\in\{(5,6,7),(3,3,5,7)\}\}.
\]
With our construction of the pairing, we have $\pair(v,\dface)\neq 0$ if and only if $v\in\mcl{Z}$. For every $v\in\mcl{Z}$ we have $K(v)=\frac{2}{210}$, $\pair(v,\dface)=1$ and $c_v=\left(\frac{2}{210}-\frac{2}{209}\right)=-\frac{2}{210\cdot 209}$. 

\begin{lemma}\label{lemma:Z}
If $G$ is a PCC graph, then $Z\leq 210$.
\end{lemma}
\begin{proof}
We already observed in equation \eqref{eq:total_curvature} that $\sum_{v\in\Verts} K(v)=2$.
Therefore \[Z \frac{2}{210}=\sum_{v\in\mcl{Z}} K(v)\leq 2\]
\end{proof}
As an immediate consequence of \cref{lemma:Z} we obtain, if $\pair(\dface)\neq 0$, that 
\begin{equation}\label{eq:dface}
	c(\dface)=-\frac{2 Z}{210\cdot 209}\geq -\frac{2}{209}\great -0.0096.
\end{equation}
In particular, we have just concluded the proof of \cref{prop:main}.

\begin{corollary}\label{cor:main}
For every $\face\in\FFaces\setminus\{\dface\}$ we have $c(G)\geq c(\face) -0.01$.
\end{corollary}
\begin{proof}
This is immediate from $c(\dface)\great -0.01$ and $c(\face')\geq 0$ for all $\face'\in\FFaces\setminus\{\dface\}$.
\end{proof}

\begin{corollary}\label{lemma:210}
If $G$ is a PCC graph, then $\#\Verts \leq 210$. Moreover, we have equality if and only if $Z=210$.
\end{corollary}

\begin{proof}
Combining \cref{prop:main} and \cref{lemma:Z} with formulas \eqref{eq:dface} and \eqref{eq:contribution}, we get
\[
\frac{2(209-\#\Verts)}{209}\geq -\frac{2 Z}{210\cdot 209}\geq \frac{-2}{209},
\]
from which the thesis follows.
\end{proof}


\section{Proving the upper bound $\#\Verts\leq 209$}\label{sec:209}


The above \cref{lemma:210} forces a rigid combinatorial description in case $\#\Verts=210$.
We rule out this possibility by means of a double-counting argument.

\begin{figure}[ht]
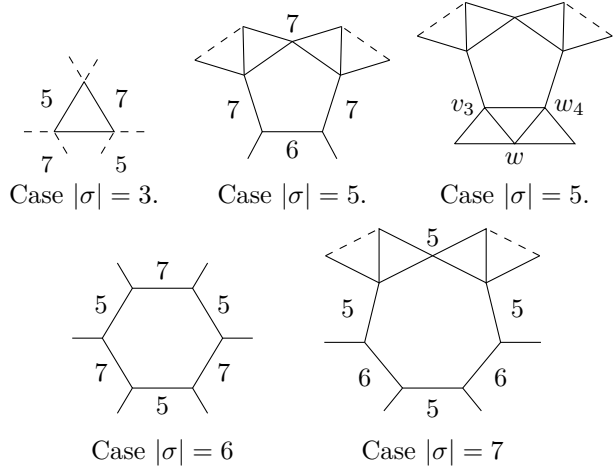
\centering
\input{210_3.tikz}
\ 
\input{210_5.tikz}
\ 
\input{210_5no.tikz}

\input{210_6.tikz}
\quad \quad
\input{210_7.tikz}
\caption{Illustrations for Lemma 16.1.}
\label{fig:209}
\end{figure}

\begin{lemma}\label{lemma:210:lemma}
Suppose that $\#\Verts=210$. Then there exist $\face_5,\face_6,\face_7\in\Faces$ with $\abs{\face_5}=5$, $\abs{\face_6}=6$ and $\abs{\face_7}=7$. Moreover, if $\face_5,\face_6,\face_7$ are such faces, we have
\begin{itemize}
\item 
$\#\{e\in\fedges(\face_5):\ \etype(e)=(5,6)\}=1$ 
and 
$\#\{e\in\fedges(\face_5):\ \etype(e)=(5,7)\}=2$;
\item 
$\#\{e\in\fedges(\face_6):\ \etype(e)=(5,6)\}=3$ 
and 
$\#\{e\in\fedges(\face_6):\ \etype(e)=(6,7)\}=3$;
\item 
$\#\{e\in\fedges(\face_7):\ \etype(e)=(5,7)\}=3$ 
and 
$\#\{e\in\fedges(\face_7):\ \etype(e)=(6,7)\}=2$.
\end{itemize}
\end{lemma}
\begin{proof}
Let $\face\in\Faces$ with $\abs{\face}=n\in\{3,5,6,7\}$ and $\fedges(\face)=\langle e_1,\ldots,e_n\rangle$. We define
\[  \ftype(\face):=\langle \etype(e_1), \ld , \etype(e_n) \rangle ,\]
with the cyclic ordering induced by the elements of $\fedges(\face)$. 
We recall that by \cref{lemma:210}  we have $\vtype(v)=(5,6,7)$ or $\vtype(v)=(3,3,5,7)$ for all $v\in\Verts$. 
We will show that, up to cyclic reordering and up to orientation: 
\begin{enumerate}[(i)]
\item 
if $\abs{\face}=3$, $\ftype(\face)=\langle (3,3), (3,5), (3,7) \rangle$;
\item
if $\abs{\face}=5$, $\ftype(\face)=\langle (5,6), (5,7), (3,5), (3,5), (5,7) \rangle$;
\item
if $\abs{\face}=6$, $\ftype(\face)=\langle (5,6), (6,7), (5,6), (6,7), (5,6), (6,7) \rangle$;
\item
if $\abs{\face}=7$, $\ftype(\face)=\langle (5,7), (6,7), (5,7), (3,7), (3,7), (5,7), (6,7) \rangle$.
\end{enumerate}
Hence, we consider the following 4 cases. 
\begin{description}
\item[Case $\abs{\face}=3$]
Then $\forall\, v\in\fverts(\face)$ we have $\vtype(v)=(3,3,5,7)$, so (i) is an easy check.
\item[Case $\abs{\face}=5$]
Since the cardinality of $\fedges(\face)$ is odd, there must be a pair of  edges $e,e'\in \fedges(\face)$, consecutive on $\face$ and meeting at $v_1$, with $\etype(e),\etype(e')\neq (5,7)$.
This shows that $\vtype(v_1)=\langle 3,5,3,7\rangle$ and $\langle (3,5), (3,5) \rangle\subseteq \ftype(\face)$.
 Let $\fverts(\face)=\langle v_1,\ldots,v_5\rangle$ with $e=v_1v_2$.
Let $\efaces(e)=\{\tau,\face\}$, so $\abs{\tau}=3$. 
Then (i) above forces $\oppf(v_1,\tau)$ to be a triangle, so $\vtype(v_2)=\langle 5,3,3,7\rangle$ and $\etype(v_2v_3)=(5,7)$. 
 If we do the same for $e'$, we get that $\langle (5,7), (3,5), (3,5), (5,7) \rangle\subseteq \ftype(\face)$. 
The remaining edge $e''=v_3v_4$ satisfies a priori $\etype(e'')=(3,5)$ or $\etype(e'')=(5,6)$, so suppose by contradiction that $\etype(e'')=(3,5)$. Let $\efaces(e'')=\{\tau'',\face\}$, with $\tau''=3$, and let $\fverts(\tau'')=\{v_3,v_4,w\}$.  Then $\vtype(v_3),\vtype(v_4)=(3,3,5,7)$ and $(3,3,3)\subseteq \vtype(w)$, which is not admissible if $\#\Verts=210$.
\item[Case $\abs{\face}=6$]
Then $\forall\, v\in\fverts(\face)$ we have $\vtype(v)=(5,6,7)$, so (iii) is an easy check.
\item[Case $\abs{\face}=7$]
Since the cardinality of $\fedges(\face)$ is odd, we deduce analogously to the previous case that $\langle (5,7), (3,7), (3,7), (5,7) \rangle\subseteq \ftype(\face)$. Then (ii) above, already proved, forces $\langle (6,7), (5,7), (3,7), (3,7), (5,7), (6,7) \rangle\subseteq \ftype(\face)$, and (iii) forces (iv).
\end{description}
Since $\Verts=\mcl{Z}$, we clearly have the existence of faces with 5 and 7 sides. The above analysis then implies the existence of triangles and of faces with 6 sides. The remaining assertions follow from (ii)-(iv).
\end{proof}

\begin{corollary}\label{lemma:210:no}
There is no PCC graph with exactly 210 vertices.
\end{corollary}

\begin{proof}
Let $A=\#\mcl{A}$, $B=\#\mcl{B}$, $C=\#\mcl{C}$, where
\[
\begin{aligned}
  \mcl{A}& = \{\face\in\Faces: \abs{\face}=5\},\\
	\mcl{B}& = \{\face\in\Faces: \abs{\face}=6\},\\
  \mcl{C}& = \{\face\in\Faces: \abs{\face}=7\}.\\
\end{aligned}
\]
From \cref{lemma:210:lemma} we derive the following system of equalities, by double-counting the edges $e\in\Edges$ with $\etype(e)=(5,6), (5,7)$, or $(6,7)$: 
\[
\left\{
\begin{aligned}
&A= 3B \\
&2A=3C\\
&3B=2C
\end{aligned}
\right.
\]
which is evidently inconsistent.
\end{proof}


\section{{\red}triangles and conclusion}\label{sec:208}

In the following lemma we use the data that we have acquired so far to considerably reduce the combinatorial complexity surrounding pentagons in an hypothetical PCC graph with 209 vertices.

\begin{lemma}\label{lemma:208:5}
Suppose that $\#\Verts = 209$, let $\face\in\Faces$ with $\abs{\face}=5$ and let
\[
\mcl{S} = \{(4,4,5),(3,4,4,5)\} \cup\{(4,5,a):\ 14\leq a\leq 19\}.
\]
Then either all $v\in\fverts(\face)$ satisfy $\vtype(v)\in\mcl{S}$, or all $v\in\fverts(\face)$ are $\dface$-vertices.
In particular, if $v_1,v_2\in\fverts(\face)$ and $v_1$ is a $\dface$-vertex, then $v_2$ is a $\dface$-vertex as well.
\end{lemma}

\begin{proof}
Let $w\in\fverts(\face)$. By looking at our construction of the pairing we find that at least one of the following 4 cases apply.

\begin{description}
\item[Case $\pair(\face)\neq 0$] 
Then by the arguments of \cref{sec:5} we see that either $c(\face)\great 0.015$ or the \emph{Case $A=4$} (of \cref{sec:5}) applies to $\face$.
In the first case we get $c(G)\great 0.005$ by \cref{cor:main}, while in the second case we get $\vtype(v)\in\mcl{S}$ for all $v\in\fverts(\face)$. 
\item[Case $\exists\face'\in\Faces$ with $\pair_2(w,\face')\neq 0$ and $\abs{\face'}\in\{40,41\}$] 
Then by \cref{prop:main} we have $c(\face')\great 0.011$ and $c(G)\great 0.001$ by \cref{cor:main}.
\item[Case $\exists\face'\in\Faces$ with $\pair_2(w,\face')\neq 0$ and $\abs{\face'}=11$] 
We can assume that $\pair_1(w,\face)=0$, since we already discussed the case $\pair(\face)\neq 0$.
Then there are only three subcases.
\begin{description}
\item[Case $\vtype(w)=(3,3,3,5)$] Then $c_w\pair_2(w,\face')\great 0.190$, and so by the comments before \cref{lemma:11:0} we get $c(\face')\great 0.168$. This gives $c(G) \great 0.158$ by \cref{cor:main}. 
\item[Case $\vtype(w)=(3,4,4,5)$] Then by \cref{rule:3445} $w$ is consecutive on $\face$ to $w_1,w_2\in\fverts(\face)$ with $4\in\vtype(w_1),\vtype(w_2)$. Since $\pair(\face)=0$ we necessarily have that $w_1$ and $w_2$ are special vertices with $\vtype(w_1)=\vtype(w_2)=(3,4,4,5)$. By repeating the argument we deduce that all $ v\in\fverts(\face)$ satisfy $\vtype(v)=(3,4,4,5)$.
\item[Case $\vtype(w)=(3,3,3,3,5)$] Then by \cref{rule:33335} $w$ is consecutive on $\face$ to $w_1\in\fverts(\face)$ with $(3,4,5)\subseteq\vtype(w_1)$. 
Since $\etype(ww_1)=(3,5)$, we cannot have $\vtype(w_1)=\langle 4,3,4,5\rangle$.
Thus, by \cref{rule:3445} we have that $w_1$ is not special with $\vtype(w_1)=(3,4,4,5)$. 
As a consequence $\pair(w_1,\face)\neq 0$, and so the case $\pair(\face)\neq 0$ above applies.
\end{description}
\item[Case $\pair(w,\ddface)\neq 0$] 
Then by the analysis of \cref{sec:ddface} we have that either $\vtype(w)=(5,5,9)$, or $c(\ddface)\great 0.015$.
In the second case, we have $c(G)\great 0.005$ by \cref{cor:main}.
\item[Case $\pair(w,\dface)\neq 0$] 
In this case $w$ is a $\dface$-vertex.
\end{description}

Since we assume $\#\Verts=209$, we have $c(G)=0$ by \eqref{eq:contribution}.
Hence the above case analysis shows that either for all $v\in\fverts(\face)$ we have $\vtype(v)\in\mcl S$, or  all $v\in\fverts(\face)$ satisfy $\vtype(v)\in\{(5,6,7), (3,3,5,7), (5,5,9)\}$.
However,  if $\fverts(\face)=\langle v_1,\ld,v_5\rangle$ and $\vtype(v_1)=(5,5,9)$, we necessarily get $\vtype(v_i)=(5,5,9)$ for all $1\leq i \leq 5$, but this is impossible since $\#\fverts(\face)$ is odd.
The thesis follows.
\end{proof}

%

We recall from \cref{def:red} that a triangle $\tau\in\Faces$ is a {\red}triangle if all its vertices are $\dface$-vertices.
It's easy to see, as in \cref{lemma:210:lemma}, that if $\tau$ is a {\red}triangle then there exist uniquely $v_\tau\in\fverts(\tau)$ and $e_\tau\in\fedges(\tau)$  with $\vtype(v_\tau)=\langle 3,5,3,7\rangle$ and $\etype(e_\tau)=(3,3)$.
Conversely, we have the following.

\begin{lemma}\label{lemma:208:red}
Suppose that $\#\Verts=209$, 
let $\tau\in\Faces$ with $\abs{\tau}=3$ and let $v\in\fverts(\face)$ with $\vtype(v)=\langle 3,5,3,7\rangle$.
Then $\tau$ is a {\red}triangle. Moreover, also $\oppf(v,\tau)$  is a {\red}triangle.
\end{lemma}

\begin{proof}
Let $v_1,v_2\in\Verts$ and $\face_1, \face_2\in\Faces$ with $\fverts(\tau)=\{v,v_1,v_2\}$, $\etype(vv_1)=(3,5)$, $\etype(vv_2)=(3,7)$, $\efaces(vv_1)=\{\tau,\face_1\}$ and $\efaces(vv_2)=\{\tau,\face_2\}$.
Since $\abs{\face_1}=5$ and $v$ is a $\dface$-vertex, we see by \cref{lemma:208:5} that also $v_1$ is a $\dface$-vertex. If $\vtype(v_1)=\langle 3, 5, 3, 7\rangle$ then $\etype(v_1v_2)=(3,7)$ and so $\vtype(v_2)=(3,7,7)$. 
By rule \cref{rule:3ab} we have that $v_2$ is not special, so  $c(\ddface)\geq c_{v_2}\great 0.1$.
By \cref{cor:main} we get $c(G)\great 0.09$, which implies $\#\Verts\leq 208$ by \cref{lemma:main}. 
Therefore $\vtype(v_1)=\langle 3,3,5,7\rangle$, so $\etype(v_1v_2)=(3,3)$ and $(3,3,7)\subseteq \vtype(v_2)$. 
Now, it's easy to rule out the following three cases.
\begin{description}
\item[Case $\vtype(v_2)=(3,3,7)$] 
Then $\pair(v_2,\ddface)=1$ and $c(\ddface)\geq c_{v_2}\great 0.299$. 
\item[Case $\vtype(v_2)\in\{(3,3,3,7),(3,3,4,7)\}$ and $v_2$ is not special] 
Then $\pair(v_2,\ddface)\geq \frac 3 4$ and $c(\ddface)\geq \frac 3 4 c_{v_2}\great 0.037$. 
\item[Case $\vtype(v_2)\in\{(3,3,3,7),(3,3,4,7)\}$ and $v_2$ is special] 
Then there is $\face_3\in\Faces$ with $\abs{\face_3}=11$ such that $\pair_2(v_2,\face_3)\geq \frac 3 4$.
Since $c_-(\face_3)\great -0.022$ by the comments preceding \cref{lemma:11:0}, we get $c(\face_3)\geq \frac 3 4 c_{v_2}-0.022\great  0.015$. 
\end{description}
In all the above cases we obtain $c(G)\geq 0.015+c(\dface)\great 0.005$ by \cref{cor:main}, which implies $\#\Verts\leq 208$ by \cref{lemma:main}. 
Hence $\vtype(v_2)=(3,3,5,7)$ and $\tau$ is a {\red}triangle. 

Let $\tau'=\oppf(v,\tau)$. 
Since $\etype(v_1v_2)=(3,3)$ we have $\abs{\tau'}=3$. 
Let $w\in\Verts$ such that $\fverts(\tau')=\{v_1,v_2,w\}$ 
and let $\face'=\oppf(v_1,\tau')$.
We already proved that $v_1$ and $v_2$ are $\dface$-vertices. 
Moreover, since $\vfaces(v_2)=\langle \face_2,\tau,\tau',\face' \rangle$, we see that $\abs{\face'}=5$. 
Since $v_2\in\fverts(\face')$ is a $\dface$-vertex, we deduce that $w$ is a $\dface$-vertex as well by \cref{lemma:208:5}.
Therefore $\tau'$ is a {\red}triangle.
\end{proof}

\begin{figure}[th]
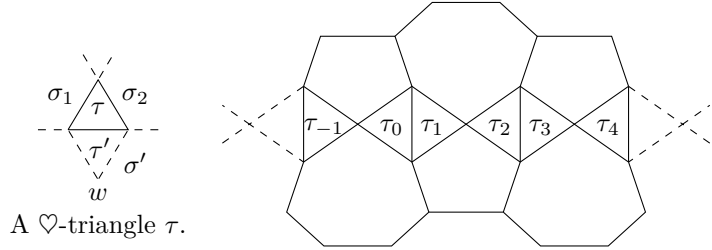
\centering
\input{208_lemma.tikz}
\ 
\input{208_chain.tikz}
\caption{An illustration for Lemma 17.2 and (a portion of) a chain of {\red}triangles.}
\label{fig:chain}
\end{figure}
A remarkable consequence of \cref{lemma:208:red} is that, given a {\red}triangle $\tau\in\Faces$, there is an unique sequence of {\red}triangles 
\[
\ld,\tau_{-2},\tau_{-1},\tau=\tau_0,\tau_1,\tau_2,\ld
\]
such that for every  $n\in\mathbb{Z}$ the triangles $\tau_{2n}$ and $\tau_{2n-1}$ meet at the vertex $v_{\tau_{2n-1}}=v_{\tau_{2n}}$ with $\vtype(v_{\tau_{2n}})=\langle 5,3,7,3\rangle$, while $\tau_{2n}$ and $\tau_{2n+1}$ share the common edge $e_{\tau_{2n}}=e_{\tau_{2n+1}}$.  
Since the set $\Faces$ is finite, every such sequence of {\red}triangles must become periodic. In other words, $\exists L\geq 1$ such that $\tau_L=\tau_0$. If $L$ is the least positive integer with this property, we say that $\{\tau_k\}_{k\in\mathbb{Z}}$ is a \emph{chain of {\red}triangles of length} $L$.
It is clear that the length of a chain of {\red}triangles must be even, but we can be more precise.

\begin{lemma}\label{lemma:208:red:4}
The length of a chain of {\red}triangles is a multiple of 4. A chain of {\red}triangles of length $4m$ consists of exactly $4m$ triangles and involves exactly $10m$ pairwise distinct edges and $6m$ distinct vertices, organized in a 2-cell complex embedded in the sphere, homeomorphic to the union of $2m$ disks glued together at $2m$ boundary points in a circular structure.
\end{lemma}

\begin{proof}
Let $\{\tau_k\}_{k\in\mathbb{Z}}$ be the chain of {\red}triangles, and for every $k\in\mathbb{Z}$ let $v_{0,k},v_{1,k},v_{2,k}\in\Verts$ such that $\fverts(\tau_k)=\{v_{0,k},v_{1,k},v_{2,k}\}$ and $v_{0,k}=v_{\tau_k}$. Let now $n\in\mathbb{Z}$ and let $\pi$ be the orientation of the sphere for which $\vtype(v_{1,n})=\langle 3,3,5,7 \rangle_\pi$. Then a careful inspection reveals that 
$\vtype(v_{2,n})=\langle 3,3,5,7 \rangle_\pi$, but $\vtype(v_{1,n+2})=\vtype(v_{2,n+2})=\langle 3,3,7,5 \rangle_\pi$. 
Since $v_{1,n+L}$ is equal to either $v_{1,n}$ or $v_{2,n}$, we must have $L$ mutiple of 4.
The rest of the statement is clear from the previous discussion and \cref{lemma:distinct}.
\end{proof}

We now notice the following.

\begin{lemma}\label{lemma:209:no}
Suppose that there exists a PCC graph with 209 vertices. Then there also exists a PCC graph with at least 210 vertices.
\end{lemma}

\begin{proof}
Let $G$ be a PCC graph with 209 vertices. From \cref{lemma:main} and \cref{prop:main} we see that the set of its $\dface$-vertices is non-empty.
Moreover, from \cref{lemma:208:5} we deduce that there is a face $\face$ with $\abs{\face}=5$ for which all $v\in\fverts(\face)$ are $\dface$-vertices.
Since the number of edges of $\face$ is odd, there must exist two edges $e_1,e_2\in\fedges(\face)$ consecutive on $\face$ with $\etype(e_1),\etype(e_2)\neq (5,7)$. Then the common endpoint of $e_1,e_2$ is a vertex $v$ with $\vtype(v)=\langle 3,5,3,7\rangle$. 
By \cref{lemma:208:red} this implies the existence of a {\red}triangle, and the previous discussion shows the existence of a chain of {\red}triangles.
By \cref{lemma:208:red:4} this chain contains $6m$ vertices for some $m\in\N$. 

Let $C$ be the geometric support of this chain, that is the closed subset of the sphere which is the union of the vertices, the edges and the triangle of the chain. By the Jordan-Sch\"onflies theorem (a version for 2-cell complexes suffices) the complement of $C$ consists of exactly two connected open sets $U_1,U_2$, each of them homeomorphic to the unit ball, with respective closures $\wb{U_1},\wb{U_2}$ homeomorphic to the unit disk.
For $i=1,2$ let $n_i$ be the number of vertices of $G$ contained in $U_i$, and suppose without loss of generality that $n_1\geq n_2$. We have that $n_1+n_2+6m=209$ is odd, thus $n_1\geq n_2+1$.

The embedding of $G$ in the sphere induces a 2-cell complex structure on both $\wb{U_1},\wb{U_2}$. We notice that their boundary structure (including the face vectors of all the vertices) is isomorphic. 
Therefore, we can perform a \emph{graph surgery} and replace the 2-cell complex structure on $\wb{U_2}$ with an homeomorphic copy of the 2-cell complex structure on $\wb{U_1}$, without altering the geometric and combinatorial data in a neighborhood of $C$. This construction gives a new graph $G'$, 2-cell embedded in the sphere, with $n_1+n_1+6m\geq 210$ vertices.

It is clear that $G'$ doesn't contain loops and that all its vertices have positive combinatorial curvature and degree at least 3. Moreover, since $G$ contains at least a vertex $v$ with $\vtype(v)=\langle 3,5,3,7\rangle$, it cannot be a prism or an antiprism. Therefore $G'$ is a PCC graph. 
\end{proof}

However, we already proved that no PCC graph can exist with at least 210 vertices. Therefore by \cref{lemma:209:no}, \cref{lemma:210}, and \cref{lemma:210:no} we deduce \cref{thm:main}.


\section{PCC graphs with 208 vertices and faces with given size}\label{sec:example}

In \cref{fig:example} we exhibit a PCC graph with 208 vertices. 
This graph was discovered in 2011 by the author (private communication with Prof. Jamie Sneddon) and later independently re-discovered by Oldridge \cite{Oldridge:244}.

\begin{figure}[ht]\centering
\includegraphics[scale = 0.7]{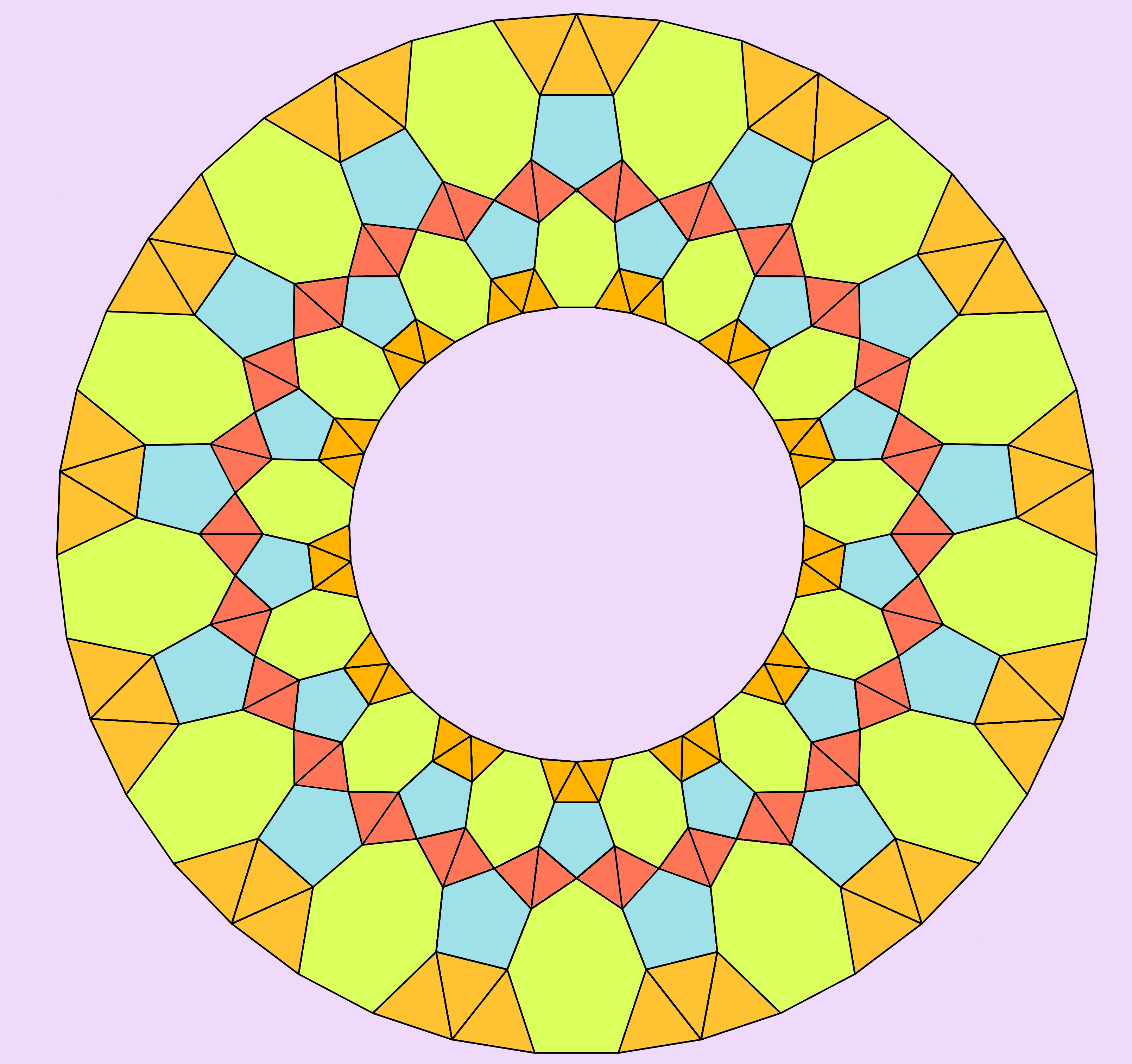} 
\caption{A PCC graph with 208 vertices and 3,5,7,39-sided faces.}
\label{fig:example}
\end{figure}

This graph contains a chain of {\red}triangles and faces with 3,5,7 and 39 edges. In total it has 208 vertices, 390 edges and 184 faces.
More precisely, there are 130 faces $\face$ with $\abs{\face}=3$; 26 with $\abs{\face}=5$; 26 with $\abs{\face}=7$; and 2 with $\abs{\face}=39$ (the inner and outer regions).  
There are 52 vertices $v$ with $\vtype(v)=(3,7,39)$ and $K(v)=\frac{1}{546}$; 130 with $\vtype(v)=(3,3,5,7)$ and $K(v)=\frac{1}{105}$; 26 with $\vtype(v)=(3,3,3,39)$ and $K(v)=\frac{1}{39}$.

\begin{figure}[ht]\centering
	\includegraphics[scale = 0.3]{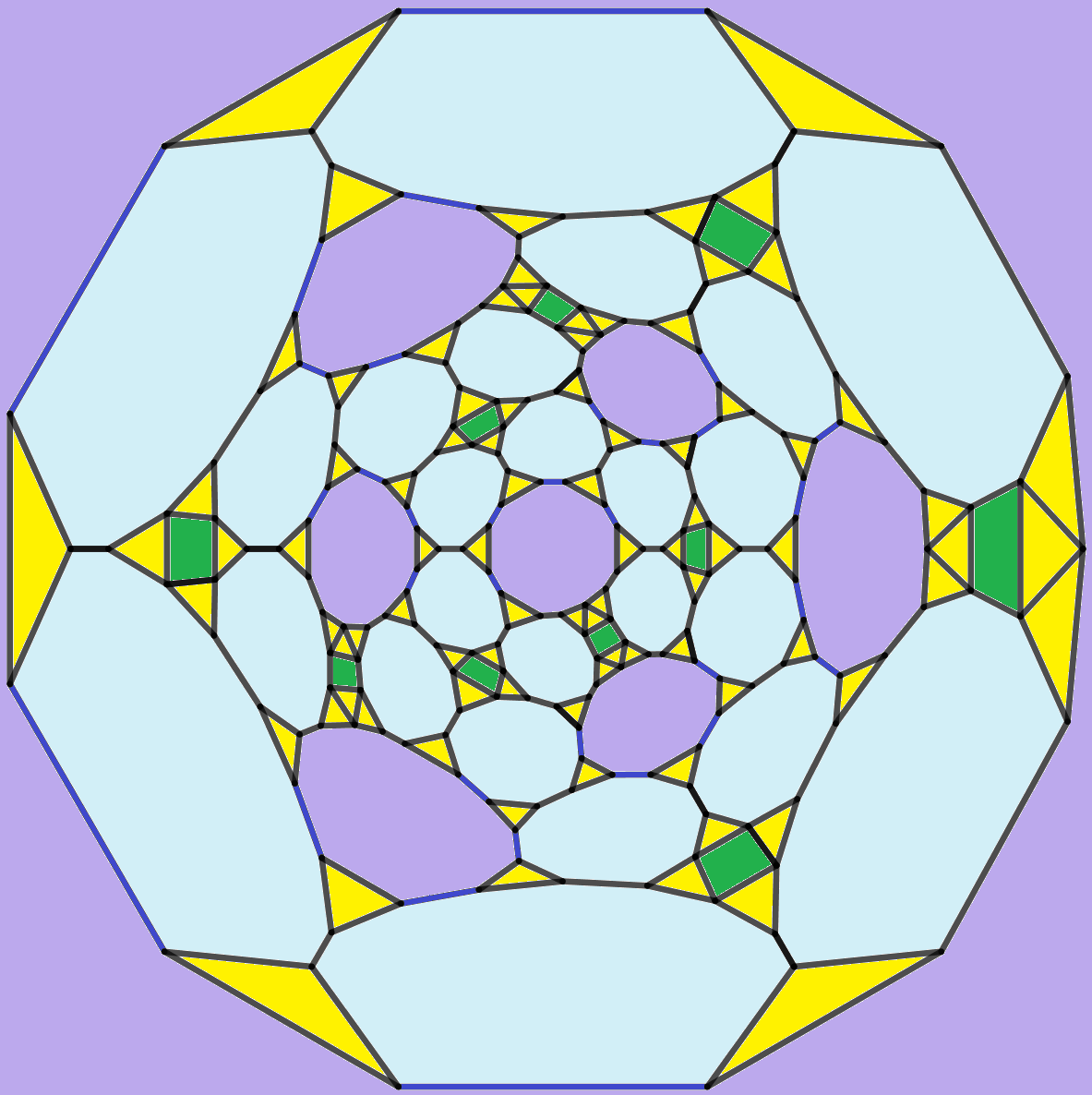} 
	\caption{A PCC graph with 208 vertices and 3,4,11,13-sided faces.}
	\label{fig:example:1113}
\end{figure}

In \cref{fig:example:1113} we show another of the PCC graphs with 208 vertices found by Nicholson and Sneddon \cite{Sneddon:208}. See also \cite{planar:thesis:marissa} for discussions on these graphs. 
This graph contains faces with 3,4,11 and 13 edges, organized as in \cref{fig:11:final}. 
In total it has 208 vertices, 336 edges and 130 faces.
More precisely, there are 88 faces $\face$ with $\abs{\face}=3$; 10 with $\abs{\face}=4$; 24 with $\abs{\face}=11$; and 8 with $\abs{\face}=13$ (included the outer region).  
There are 96 vertices $v$ with $\vtype(v)=(3,11,13)$ and $K(v)=\frac{1}{858}$; 40 with $\vtype(v)=(3,3,4,11)$ and $K(v)=\frac{1}{132}$; 64 with $\vtype(v)=(3,11,11)$ and $K(v)=\frac{1}{66}$; 8 with $\vtype(v)=(3,3,3,13)$ and $K(v)=\frac{1}{13}$.

The PCC graph in \cref{fig:example} is constructed modularly around a closed chain of {\red}triangles, by repeating 26 times a ``{\red}motif'' that consists of 5 triangles, a 5-sided face and a 7-sided face. 
Oldridge observed that by allowing only $2N$ repetitions in this construction, we could exhibit PCC graphs containing a pair of faces with size $\abs{\face}=3N$, for each $1\leq N\leq 13$.  
This disproves a previous conjecture: in \cite[pag. 29]{Reti:138} the authors made the prediction that no PCC graph could contain a face $\face$ with $\abs{\face}\geq 23$. 

In \cite[Sec. 6.3]{Oldridge:244} it is proposed the open problem of exhibiting a PCC graph containing a face with size $\abs{\face}\geq 23$ not divisible by three. 
We now provide a simple solution to this problem. 
It suffices to repeat the {\red}motif as before, but around an \emph{open} chain of {\red}triangles, as in \cref{fig:disproof:example}. 
In this way each {\red}motif contributes three edges to the ``outer'' face. 
Moreover, in order to end up with a PCC graph, it is necessary to add two ``closing caps'' at the extremities of the chain, in such a way that only   admissible vertices are produced. 
This can be done without difficulty: in \cref{fig:disproof:example} we use 4 triangles and two pentagons, so each cap contributes five edges to the outer face. By using one less triangle it is possible to construct a cap that contributes only four edges.  

\begin{figure}[ht]\centering
	\input{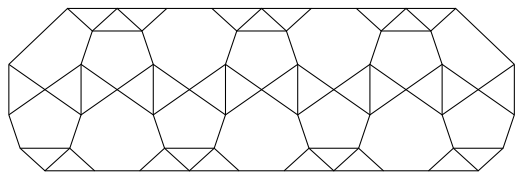}
	\caption{An example of a PCC graph containing a face $\face$ with $\protect \abs\face=25$.}
	\label{fig:disproof:example}
\end{figure}

With this construction we are able to produce, for any $8\leq N\leq 41$, a PCC graph $G_N$ that contains a face $\face$ with $\abs\face=N$. 
These graphs also contain 3-sided, 5-sided and 7-sided faces. 
Oldridge's construction for $N=2$ provides a PCC-graph with a 6-sided face, and there are plenty of PCC graphs with 4-sided faces (for example, the Nicholson-Sneddon graphs, or some platonic solids). 
In conclusion, all sizes $3\leq N\leq 41$ are admissible in a PCC graph.

\subsection*{Acknowledgements}

This work was supported in part by the full scholarship (Corso Ordinario) granted by the Scuola Normale Superiore (Pisa, Italy), and in part by the full International scholarship awarded by the Faculty of Graduate and Postdoctoral Studies (Ottawa, Canada). 
The author is grateful to Bobo Hua for noticing that our main result on the size of planar PCC graphs resolves also the corresponding problem for projective PCC graphs.  


\bibliography{biblio_PCC_graph_mendeley}


\end{document}